\documentclass[reqno,12pt]{amsart}

\usepackage[a4paper,
left=2.2cm,right=2.2cm,
top=2.3cm,bottom=2.3cm]{geometry}

\usepackage{graphicx}
\usepackage{amsmath,amssymb,amsfonts,amsthm,amscd}
\usepackage{mathrsfs}
\usepackage{array}
\usepackage{enumerate}
\usepackage{xcolor}
\usepackage{textcomp}
\usepackage[colorlinks=true,linkcolor=blue,citecolor=blue,urlcolor=blue]{hyperref}
\usepackage{float}
\usepackage{subcaption}
\usepackage{longtable}
\usepackage{rotating}
\usepackage{booktabs}
\usepackage{adjustbox}
\usepackage{threeparttable}
\usepackage{multirow}
\usepackage{listings}

\usepackage{algorithm}
\usepackage{algorithmicx}
\usepackage{algpseudocode}

\usepackage{tikz}
\usepackage{tikz-cd}
\usetikzlibrary{arrows.meta,positioning,calc,intersections,backgrounds,matrix}
\usepackage{pgfplots}
\pgfplotsset{compat=1.16}

\usepackage{caption}

\graphicspath{{figs/}}

\newcolumntype{L}[1]{>{\raggedright\arraybackslash}p{#1}}

\DeclareMathOperator{\diag}{diag}
\newcommand{\R}{\mathbb{R}}

\numberwithin{equation}{section}

\newtheorem{theorem}{Theorem}[section]

\newtheorem{lemma}[theorem]{Lemma}
\newtheorem{corollary}[theorem]{Corollary}

\theoremstyle{definition}
\newtheorem{definition}[theorem]{Definition}
\newtheorem{assumption}[theorem]{Assumption}
\newtheorem{example}[theorem]{Example}
\newtheorem{remark}[theorem]{Remark}

\title[Yau's Affine Normal Descent]
{Yau's Affine Normal Descent: Algorithmic Framework and Convergence Analysis}

\author{Yi-Shuai Niu$^{1}$, Artan Sheshmani$^{1,3}$, and Shing-Tung Yau$^{1,2}$}

\begin{document}

\begin{abstract}
We propose Yau's Affine Normal Descent (YAND), a geometric framework for smooth unconstrained optimization in which search directions are defined by the equi--affine normal of level-set hypersurfaces.
The resulting directions are invariant under volume-preserving affine transformations and intrinsically adapt to anisotropic curvature.
Using the analytic representation of the affine normal from affine differential geometry, we establish its equivalence with the classical slice--centroid construction under convexity.
For strictly convex quadratic objectives, affine-normal directions are collinear with Newton directions, implying one-step convergence under exact line search.
For general smooth (possibly nonconvex) objectives, we characterize precisely when affine-normal directions yield strict descent and develop a line-search-based YAND.
We establish global convergence under standard smoothness assumptions, linear convergence under strong convexity and Polyak--{\L}ojasiewicz conditions, and quadratic local convergence near nondegenerate minimizers.
We further show that affine-normal directions are robust under affine scalings, remaining insensitive to arbitrarily ill-conditioned transformations.
Numerical experiments illustrate the geometric behavior of the method and its robustness under strong anisotropic scaling.
\end{abstract}



\maketitle
\tableofcontents

\section{Introduction}\label{sec:intro}

Designing optimization algorithms that are robust to scaling, conditioning,
and coordinate transformations remains a central challenge in mathematical programming.
A fundamental difficulty lies in the mismatch between the geometry of the objective
function and the geometry implicitly imposed by the algorithm.
Classical methods such as gradient descent and Newton's method rely on
Euclidean or locally quadratic models, whose performance can degrade severely
under affine transformations of the variables.
Even for simple convex problems, affine scalings can arbitrarily distort
the geometry of level sets, leading to ill-conditioning and poor algorithmic behavior.

A recurring geometric principle is that optimization performance is governed
by the shape of level sets.
Gradient descent performs well when level sets are nearly spherical,
while Newton's method is exact for strictly convex quadratic functions
with ellipsoidal level sets.
However, these methods do not derive their search directions from the intrinsic
geometry of level sets themselves.
This raises a fundamental question:

\emph{Can one design optimization methods whose search directions are determined
directly by the intrinsic geometry of level sets, and are therefore inherently
affine invariant?}

Affine invariance is widely regarded as a desirable structural property
for mitigating sensitivity to scaling and reparameterization.
Existing approaches achieve invariance by prescribing
a geometry on the ambient space, such as through
barrier functions, norm-based structures, or Minkowski gauges
\cite{NesterovNemirovskii1994IP,NesterovTodd1997SelfScaled,
dAspremontGuzmanJaggi2018AffineInvariant,DoikovNesterov2023ContractingPoint}.
In contrast, we derive invariance directly from the intrinsic
geometry of level sets, without relying on any externally prescribed metric.

Specifically, we consider the smooth unconstrained problem
\[
    \min_{x\in\R^{n+1}} f(x),
\]
where $f$ is sufficiently smooth.
At a current iterate $x_k$, we examine the level-set hypersurface
\[
    \mathcal{L}_k := \{x : f(x)=f(x_k)\},
\]
and extract a search direction from its equi--affine geometry.

The geometric object underlying our construction is the \emph{affine normal},
a classical notion in affine differential geometry, defined as a canonical transversal direction
determined solely by the local shape of a hypersurface,
independent of parametrization,
and invariant under unimodular (volume-preserving) affine transformations.
Unlike the Euclidean normal (the gradient),
the affine normal encodes higher-order geometric information
in a coordinate-free manner. For quadratic functions with ellipsoidal level sets,
the affine normal at every point points directly toward the unique critical point.
Consequently, affine normal directions coincide with Newton directions
in the quadratic case,
although the affine normal itself is defined independently of
second-order Taylor expansions.
This geometric exactness motivates its use
as a descent direction beyond quadratic models.

The use of affine-normal geometry for optimization was initiated by
Cheng--Cheng--Yau~\cite{ChengChengYau2005},
who proposed deriving search directions from equi--affine normal directions
of level sets.
This work reveals a fundamental geometric principle:
optimization directions can be defined intrinsically from level-set geometry,
rather than from local Taylor expansions.
For convenience, we refer to methods that follow this principle as
\emph{Yau-type} (or \emph{Yau-like}) methods.
To date, however, no general algorithmic framework or convergence theory
based on this principle has been established.

We develop a complete optimization framework
based on affine-normal geometry.
We introduce the Yau's \emph{affine normal descent} (YAND) algorithm
for smooth unconstrained optimization and establish its fundamental
algorithmic properties.
Our analysis shows that the affine-normal direction provides a
curvature-aware search direction derived from the intrinsic geometry
of level sets, leading to strong invariance properties and favorable
local convergence behavior.

At each iteration, the search direction is given by the affine normal
to the current level set, followed by a line search.
Since the affine normal is defined only up to scale,
step sizes are determined independently using standard rules
(e.g., exact line search, Armijo, or strong Wolfe conditions).
By construction, the resulting direction is invariant under
unimodular affine transformations of variables. Conceptually, YAND provides a geometric interpretation of
curvature-aware optimization directions that depend only on
the intrinsic structure of level sets rather than on a particular
coordinate representation.

A central subtlety is that the affine normal is not automatically
a descent direction.
We show that strict descent holds precisely at \emph{elliptic points},
where the Hessian restricted to the tangent space of the level set
is positive definite.
At non-elliptic points,
the affine normal may fail to point inward,
which necessitates a computable ellipticity test
and a principled orientation correction.
This yields a well-defined algorithm applicable to general smooth (possibly nonconvex) objectives.

We establish global and local convergence guarantees for YAND.
Globally, under standard smoothness assumptions and appropriate line-search rules,
the method admits convergence guarantees comparable to first-order methods;
in particular, we obtain linear convergence under strong convexity
and Polyak--{\L}ojasiewicz conditions,
and gradient convergence under strong Wolfe conditions
for general nonconvex objectives.
Locally, near a nondegenerate minimizer,
the affine normal direction coincides with the Newton direction
up to second-order terms,
which implies local quadratic convergence under standard assumptions. This relationship provides a geometric interpretation of
Newton's method: in the quadratic case the affine normal
exactly recovers the Newton direction,
while beyond the quadratic setting it yields a curvature-aware
search direction derived directly from level-set geometry.
Importantly, affine-normal directions are defined intrinsically
from level-set geometry and do not rely on explicit Hessian inversion or quadratic models.
Thus, YAND combines robust global behavior with high local efficiency, while being intrinsically equi--affine invariant.

In addition, we study the behavior of affine-normal directions
under affine-scaled quadratic models and show that the resulting
search directions are unaffected by arbitrarily ill-conditioned
linear scalings. This provides a geometric explanation of the
robustness of affine-normal descent with respect to anisotropic
affine transformations. 

The main contributions of this paper are as follows:

\begin{itemize}
\item
We establish the YAND framework, which defines
optimization directions intrinsically from the equi--affine geometry
of level sets.

\item
We characterize precisely when affine-normal directions yield strict descent,
leading to a well-defined algorithm for general smooth (possibly nonconvex) objectives.

\item
We establish global convergence under standard line-search rules,
linear convergence under strong convexity and Polyak--{\L}ojasiewicz conditions,
and quadratic local convergence near nondegenerate minimizers.

\item
We prove that affine-normal directions are inherently robust under
arbitrarily ill-conditioned affine scalings.

\end{itemize}

Finally, we present a series of numerical experiments that illustrate
the geometric behavior of the proposed method.
Rather than performing large-scale benchmarking,
the experiments are designed to highlight characteristic phenomena,
including robustness under affine scalings,
behavior on ill-conditioned quadratic models,
and representative convex/nonconvex test problems.
Comparisons with classical methods such as gradient descent and
Newton's method further illustrate the distinct convergence behavior
of affine-normal descent predicted by the theory. Large-scale implementation and benchmarking are deferred to future work.

Table~\ref{tab:comparison} positions YAND relative to representative
geometry-aware optimization paradigms,
highlighting the geometric object defining each direction,
the required information,
and key structural properties.

\begin{table}[t]
\centering
\caption{High-level comparison of YAND with representative geometric optimization paradigms. 
Convergence rates refer to standard theoretical regimes (e.g., strong convexity, PL condition, or local analysis).}
\label{tab:comparison}

\begin{adjustbox}{max width=\textwidth}
\begin{minipage}{\linewidth}
\footnotesize
\setlength{\tabcolsep}{2.5pt}
\renewcommand{\arraystretch}{1.1}

\begin{tabular}{L{2.2cm} L{2.8cm} L{2.6cm} L{3.0cm} L{2.6cm} L{2.8cm}}
\toprule
Method 
& Direction-defining object 
& Information needed 
& Typical guarantee 
& Affine invariance
& Typical limitation \\ 
\midrule

Newton / damped Newton
& Local quadratic model via $\nabla^2 f(x)$
& Gradient + Hessian solve
& Quadratic local convergence; linear global with damping
& Linear affine invariant
& Requires SPD (or regularization); unstable far from minimizers \\

\addlinespace

Quasi--Newton (BFGS/L-BFGS)
& Secant-based metric approximation
& Gradients; curvature pairs
& Superlinear local convergence (BFGS); linear global under standard assumptions
& Not affine invariant
& Sensitive to scaling \\

\addlinespace

Natural gradient \cite{Amari1998} / Riemannian methods \cite{Absil2008}
& Riemannian metric (e.g., Fisher information)
& Gradient + metric (or inverse)
& Typically linear convergence
& Coordinate invariant (metric-dependent)
& Requires problem-specific metric; curvature assumptions \\

\addlinespace

Mirror descent \cite{BeckTeboulle2003}
& Bregman divergence (mirror map)
& Gradient; prox/mirror step
& $O(1/k)$ for convex problems; $O(1/k^2)$ with acceleration
& Not affine invariant
& Performance depends on mirror choice \\

\addlinespace

Interior-point (self-concordant)
& Barrier geometry (Dikin ellipsoids)
& Barrier + derivatives; Newton steps
& Polynomial-time global complexity; quadratic local convergence
& Affine invariant under barrier geometry
& Restricted to structured convex problems \\

\addlinespace

YAND
& Equi-affine normal of level sets
& First/second derivatives or moment approximation
& Linear under PL/strong convexity; quadratic local convergence; exact on strictly convex quadratics
& Equi-affine invariant (volume-preserving maps)
& True affine normal inward only at elliptic points; correction needed otherwise \\

\bottomrule
\end{tabular}

\end{minipage}
\end{adjustbox}
\end{table}
The remainder of the paper is organized as follows.
Section~\ref{sec:affine-normal} reviews the affine normal construction
and its analytic representation for level sets.
Section~\ref{sec:newton-equivalencequadratic}
establishes the correspondence between affine normal directions
and Newton steps for strictly convex quadratic objectives.
Section~\ref{sec:AN-descent}
characterizes descent and ellipticity conditions.
Section~\ref{sec:examples}
provides illustrative examples for computing the affine normal.
Section~\ref{sec:AND}
introduces the YAND algorithm together with
line-search strategies.
Section~\ref{sec:YAND-convergence}
establishes global and local convergence results.
Section~\ref{sec:super-quadratic}
discusses the potential for beyond-quadratic convergence rates,
highlighting the interplay between local order and global geometry.
Section~\ref{sec:affine-scaling}
analyzes affine-scaling models and explains the robustness of
affine-normal directions with respect to condition numbers.
Numerical experiments are reported in Section~\ref{sec:numerics},
followed by concluding remarks.

\section{Affine normal direction}\label{sec:affine-normal} 
The concept of the affine normal emerged from affine differential geometry in the early 20th century through the work of Blaschke, Berwald, and others; see, e.g., \cite{NomizuSasaki1994,LiSimonZhao1991}. Unlike Euclidean geometry, which privileges orthogonal transformations, affine geometry studies properties invariant under the larger group of volume-preserving affine transformations. The affine normal represents the natural ``normal direction'' from this affine-invariant perspective.

\subsection{Two formulas for the affine normal direction}

\subsubsection{Derivative formula (analytical expression)}
Let $f:\mathbb{R}^{n+1}\rightarrow\mathbb{R}$, and at a point $z$, consider the level set hypersurface $M=\{x: f(x)=f(z)\}$. Rotate the coordinate system so that the last axis aligns with $\nabla f(z)$ (the ``normal-aligned coordinates''). Define
\[
f_{i}=\partial_{x_{i}}f,\quad f_{ij}=\partial_{x_{i}}\partial_{x_{j}}f,\quad f_{pqr}=\partial_{x_{p}}\partial_{x_{q}}\partial_{x_{r}}f,
\]
and denote $[f^{ij}]=[f_{ij}]^{-1}$ as the inverse of the \textbf{tangent-tangent} block of the Hessian $(i,j,p,q,r\in\{1,\ldots,n\})$. Then the affine normal direction (up to scale) can be written as (cf. Cheng-Cheng-Yau~\cite{ChengChengYau2005})
\begin{equation}
    \label{eq:AN-analytic}
\boxed{\quad
d_{\mathrm{AN}}(z)\ \propto\
\begin{pmatrix}
\displaystyle f^{ij}\Big(-\frac{1}{n+2}\,\|\nabla f\|\,f^{pq}f_{pqi}\ +\ f_{n+1,i}\Big)\\[6pt]
-1
\end{pmatrix}
\quad}
\end{equation}
This gives the coordinate components in the ``normal-aligned'' system; the mean curvature factor appearing in the full geometric derivation is ignored here since only the direction is relevant.

\subsubsection{Slice-centroid formula (geometric approximation)}
At a point $z$, consider the tangent plane
\[
P(x)=\nabla f(z)\cdot(x-z)=0,\qquad\text{and its parallel family }P(x)=C.
\]
The sublevel set is
\[
\Omega_{z}:=\{x:\ f(x)\leq f(z)\}.
\]
For each $C$, define the slice
\[
S_{C}:=\{x:\ P(x)=C\}\cap\Omega_{z},
\]
and let $g(C)$ be the centroid of $S_{C}$, when this region is a convex body. Choosing the normal so that $C<0$ corresponds to the interior of $\Omega_{z}$, we define the slice-centroid direction by
\[
d_{\mathrm{SC}}(z)
\ \propto\
\lim_{C\uparrow 0}\frac{g(C)-z}{-C},
\qquad
\text{if the limit exists.}
\]
Numerically, for small $\delta>0$,
\[
\widehat d_{\mathrm{SC}}(z)\ \propto\ \frac{g(-\delta)-z}{\delta},
\]
whose truncation error is $O(\delta)$. This formula yields the analytic affine normal (up to scale) precisely when $z$ is an elliptic point\footnote{A point $z$ is called \textbf{elliptic} if the tangent--tangent Hessian at $z$ is positive definite; in this case $M$ is locally strictly convex. It is called \textbf{hyperbolic} if that block is indefinite, and \textbf{parabolic} if it is singular.}, i.e., when the slice $S_{C}$ is convex for $C<0$ and shrinks smoothly to $z$. Outside the elliptic region (hyperbolic or parabolic points), the slices may be nonconvex, disconnected, or even unbounded, and $d_{\rm SC}(z)$ is no longer guaranteed to align with the analytic affine normal. Under the inward convention, whenever the centroid exists, the slice-centroid direction still has an inward/descent normal component, since
\[
\Bigl\langle \nabla f(z),\,\frac{g(C)-z}{-C}\Bigr\rangle=\frac{P(g(C))}{-C}=\frac{C}{-C} =-1.
\]
Thus the real difficulty outside the elliptic regime is geometric well-posedness and consistency, not the descent sign alone. Note that this method requires only first-order information (for constructing the tangent plane) and avoids third derivatives and matrix inversion, though computing the high-dimensional centroid $g(C)$ efficiently remains the main bottleneck.

\begin{remark}[Moment viewpoint]
From the perspective of affine integral geometry, the three constructions above fit naturally into a hierarchy of \textbf{geometric moments}. The $0$-th moment $\int_{K}1\,dx$ records only the total mass (volume) of a region $K$; the first moment $\int_{K}x\,dx$ determines its centroid; and the centered second moment $\int_{K}(x-G)(x-G)^{\top}dx$ captures local shape via an inertia (ellipsoid) tensor. In this hierarchy, volume, centroid, and curvature (encoded by the affine metric and ultimately the affine normal) correspond, respectively, to $0$-th, $1$-st, and $2$-nd order geometric moments.

At an \textbf{elliptic} point of the level set, the analytic derivative formula, the slice-centroid construction, and the cap-centroid construction all probe the same second-order geometry. When expressed in a normal-aligned coordinate chart and expanded to quadratic order, each construction recovers the same affine metric and therefore yields the same affine normal direction.

Thus the slice- and cap-based directions may be interpreted as integral (moment-based) realizations of the analytic affine normal at elliptic points. In the next section, we make this precise by showing that, after appropriate normalization, both the slice and the cap directions converge in the limit to the analytic affine normal, agreeing with it up to a positive scalar factor.
\end{remark}

\subsection{Equivalence of the analytical and slice-centroid formulas under convexity}

The slice-centroid formula captures the affine normal direction only when the level hypersurface is locally strictly convex. In general nonconvex situations, the slices may cease to be convex bodies: they can become disconnected, nonconvex, or unbounded, and the centroid trajectory may fail to agree with the analytical affine-normal direction. Under the inward convention, whenever the centroid exists, the corresponding slice-centroid direction is automatically inward in the normal component; the real issue is therefore geometric well-posedness and consistency, not descent sign alone.

\begin{theorem}[Consistency of slice-centroid and analytical affine normal under convexity]
Let $M=\{x: f(x)=f(z)\}$ be the level hypersurface at $z$, and assume that $M$ is \textbf{locally strictly convex} at $z$ (equivalently, the tangent-tangent Hessian of $f$ is positive definite at $z$). Let $g(C)$ be the centroid of the slice
\[
S_{C}=\{x:\ P(x)=C\}\cap\Omega_{z},\qquad\Omega_{z}=\{f\leq f(z)\},
\]
where $P(x)=\nabla f(z)\cdot(x-z)$ and $C<0$ corresponds to the interior. Then the inward one-sided limit defining $d_{\mathrm{SC}}(z)$ is well defined, and
\[
d_{\mathrm{SC}}(z)\ \parallel\ d_{\mathrm{AN}}(z).
\]
\end{theorem}

\begin{proof}
\textbf{Step 0 (Setup and choice of coordinates)}. By an equi-volume affine change of variables, translate $z$ to the origin, align the tangent plane with $\{t=0\}$, and align the normal $\nu=\nabla f(z)/\|\nabla f(z)\|$ with $e_{n+1}$. Denote the local coordinates by $(u,t)\in\mathbb{R}^{n}\times\mathbb{R}$.

\textbf{Step 1 (Local graph of the hypersurface)}. In these coordinates, after a further linear transformation that diagonalizes the second fundamental form, the hypersurface $M$ has the expansion
\[
t=\Phi(u)=\tfrac{1}{2}|u|^{2}+\tfrac{1}{6}\,C_{ijk}u_{i}u_{j}u_{k}+O(|u|^{4}),\qquad u\in\mathbb{R}^{n},
\]
where $(C_{ijk})$ is the totally symmetric Pick cubic form, satisfying the apolar (trace-free) condition
\[
\sum_{j=1}^{n}C_{ijj}=0,\quad i=1,\ldots,n.
\]
The absence of linear terms follows from the tangent plane condition, and the quadratic term is normalized to $\frac{1}{2}|u|^2$ by the choice of coordinates that diagonalizes the second fundamental form.

Under this equi-volume affine normalization, the analytical affine-normal direction at $z$ is the $+t$ direction.

\textbf{Step 2 (Sublevel set and slices)}. The local ``interior'' sublevel set is
\[
\Omega_{z}=\{(u,t):\ t\geq\Phi(u)\}.
\]
Take the parallel slice planes $P_{r}=\{(u,t):\ t=r\}$ with small $r>0$. Then
\[
S_{r}=\{(u,r):\ \Phi(u)\leq r\}
\]
is a bounded convex set in $\mathbb{R}^{n}$. Denote its centroid by
\[
g(r)=(\bar{u}(r),r),\qquad \bar{u}(r)=\frac{1}{V(r)}\int_{\{\Phi(u)\leq r\}}u\,du,\quad V(r)=\int_{\{\Phi(u)\leq r\}}1\,du.
\]

\textbf{Step 3 (Scaling)}. Let $u=\sqrt{r}\,y$ to obtain
\[
\Phi(\sqrt{r}\,y)=\tfrac{1}{2}r|y|^{2}+\tfrac{1}{6}r^{3/2}C_{ijk}y_{i}y_{j}y_{k}+O(r^{2}).
\]
Set $\varepsilon=\sqrt{r}$ and define
\[
\mathbb{D}_{\varepsilon}=\Big{\{}y:\ \tfrac{1}{2}|y|^{2}+\tfrac{1}{6}\varepsilon\,C(y,y,y)+O(\varepsilon^{2})\leq 1\Big{\}},\qquad \mathbb{B}=\{y:\tfrac{1}{2}|y|^{2}\leq 1\}.
\]
Then
\[
V(r)=r^{n/2}|\mathbb{D}_{\varepsilon}|,\qquad \int_{\{\Phi(u)\leq r\}}u\,du=r^{(n+1)/2}\int_{\mathbb{D}_{\varepsilon}}y\,dy,
\]
so
\[
\bar{u}(r)=\sqrt{r}\,\frac{\int_{\mathbb{D}_{\varepsilon}}y\,dy}{|\mathbb{D}_{\varepsilon}|}.
\]

\textbf{Step 4 (Boundary perturbation and Hadamard variation)}. Write the boundary as $y=\rho(\theta)\theta$, $\theta\in\mathbb{S}^{n-1}$. From
\[
F(\rho,\varepsilon;\theta)=\tfrac{1}{2}\rho^{2}+\tfrac{1}{6}\varepsilon\,\rho^{3}C(\theta,\theta,\theta)+O(\varepsilon^{2})-1=0,
\]
at $\varepsilon=0$ we have $\rho_{0}=\sqrt{2}$ and
\[
\delta\rho(\theta):=\left.\frac{d\rho}{d\varepsilon}\right|_{\varepsilon=0}=-\frac{\partial F/\partial \varepsilon}{\partial F/\partial \rho}\Big|_{\varepsilon=0,\rho=\rho_0}=-\tfrac{1}{3}C(\theta,\theta,\theta).
\]

By a Hadamard-type variation formula,
\[
\frac{d}{d\varepsilon}\Big{|}_{\varepsilon=0}\int_{\mathbb{D}_{\varepsilon}}\psi(y)\,dy=\int_{\partial\mathbb{B}}\psi(y)\,\delta\rho(\theta)\,d\sigma(\theta),\qquad y=\rho_{0}\theta.
\]

\textbf{Step 5 (Apply to $\psi(y)=y_{i}$)}. With $\psi(y)=y_{i}$,
\[
\frac{d}{d\varepsilon}\Big{|}_{\varepsilon=0}\int_{\mathbb{D}_{\varepsilon}}y_{i}\,dy=\int_{\partial\mathbb{B}} y_i \delta\rho(\theta)\,d\sigma(\theta)=-\tfrac{1}{3}\int_{\mathbb{S}^{n-1}} \theta_i C(\theta,\theta,\theta)\,d\sigma(\theta).
\]

Using the spherical moment identity from Lemma 3:
\[
\int_{\mathbb{S}^{n-1}}\theta_i\theta_a\theta_b\theta_c\,d\sigma(\theta)=\alpha_n\left(\delta_{ia}\delta_{bc}+\delta_{ib}\delta_{ac}+\delta_{ic}\delta_{ab}\right),
\]
where $\alpha_n=\frac{\sigma_{n-1}}{n(n+2)}$ and $\sigma_{n-1}$ denotes the surface area of $\mathbb{S}^{n-1}$, we compute:
\[
\int_{\mathbb{S}^{n-1}} \theta_i C(\theta,\theta,\theta)\,d\sigma(\theta)=\int_{\mathbb{S}^{n-1}} \theta_i \theta_a\theta_b\theta_c C_{abc}\,d\sigma(\theta)=\alpha_n C_{abc}\left(\delta_{ia}\delta_{bc}+\delta_{ib}\delta_{ac}+\delta_{ic}\delta_{ab}\right).
\]

This simplifies to:
\[
\alpha_n\left(C_{ibb} + C_{ibi} + C_{iib}\right)=3\alpha_n C_{ibb}=0,
\]
by the apolar condition $\sum_b C_{ibb}=0$. Therefore,
\[
\int_{\mathbb{D}_{\varepsilon}}y\,dy=O(\varepsilon^{2}),\quad \text{and}\quad |\mathbb{D}_{\varepsilon}|=|\mathbb{B}|+O(\varepsilon^{2}).
\]

\textbf{Step 6 (Centroid asymptotics and direction)}. Therefore
\[
\bar{u}(r)=\sqrt{r}\,\frac{O(\varepsilon^{2})}{|\mathbb{B}|+O(\varepsilon^{2})}=O(r^{3/2}),
\]
and hence
\[
g(r)=(O(r^{3/2}),\,r).
\]
Thus
\[
\lim_{r\downarrow 0}\frac{g(r)-z}{r}=(0,\ldots,0,1)\in\mathbb{R}^{n+1}.
\]

\textbf{Step 7 (Recovering the affine normal)}. Under the equi-volume affine normalization, $(0,\ldots,0,1)$ is the analytical affine-normal direction. Undoing the normalization preserves direction up to scale. Returning to the original notation $C=-r<0$ for inward slices gives
\[
\lim_{C\uparrow 0}\frac{g(C)-z}{-C}\ \parallel\ d_{\mathrm{AN}}(z).
\]
\end{proof}

The following Lemma is used in Step 5 to compute the first variation of the centroid.

\begin{lemma}[Fourth-moment tensor on the sphere]
For the unit sphere $\mathbb{S}^{n-1}$, the fourth-order spherical moment satisfies
\[
\int_{\mathbb{S}^{n-1}}\theta_{i}\theta_{a}\theta_{b}\theta_{c}\,d\sigma(\theta)=\alpha_{n}\left(\delta_{ia}\delta_{bc}+\delta_{ib}\delta_{ac}+\delta_{ic}\delta_{ab}\right),
\]
where $\alpha_{n}=\frac{\sigma_{n-1}}{n(n+2)}$ and $\sigma_{n-1}$ denotes the surface measure of $\mathbb{S}^{n-1}$. This identity follows from $O(n)$-invariance and the standard spherical moment formulas; see, for example, Fang and Zhang~\cite[Chapter~2]{FangZhang1990}.
\end{lemma}

\paragraph{Comparison of the two formulas}

\begin{itemize}
\item \textbf{Derivative formula:}
\begin{itemize}
\item[+] Pros: Exact and deterministic direction; for convex quadratic forms, third derivatives vanish and the direction is parallel to the Newton direction, leading to one-step convergence under line search (see Theorem 4). Convenient when automatic differentiation is available.
\item[--] Cons: Requires inversion of the tangent-tangent Hessian ($O(n^{3})$) and evaluation of third derivatives; sensitive to noise or degenerate Hessians; direction degenerates when Hessian is singular.
\end{itemize}

\item \textbf{Slice-centroid formula:}
\begin{itemize}
\item[+] Pros: Only requires first derivatives; bypasses third derivatives and matrix inversion; more robust under noise; practical alternative when higher derivatives are unavailable.
\item[--] Cons: The main difficulty lies in computing the centroid of $S_{C}$, which is expensive in high dimensions and generally requires approximation via sampling, numerical integration, or minimal-volume ellipsoid estimation; accuracy and convergence depend on sampling quality.
\end{itemize}
\end{itemize}

\subsection{Role of convexity}

Convexity plays a fundamental role in the agreement of the two affine-normal constructions and in ensuring that the affine normal serves as a descent direction for $f$. We summarize these relationships below.

\textbf{Equivalence of the two constructions.} The analytic affine normal agrees (up to a positive scalar) with the slice-centroid construction \textbf{precisely at elliptic points} of the level set, i.e., points where the tangent-tangent Hessian is positive definite and the level set is locally strictly convex. At such points, small slices are convex bodies, their moments are well defined, and the centroid trajectories reproduce the analytic affine normal in the limit. Thus the equivalence of the derivative formula and slice formula holds exactly on elliptic patches of the hypersurface.

\textbf{Failure at non-elliptic points.} At hyperbolic points, the tangent-tangent Hessian is indefinite: the level set bends in opposite directions, slices become nonconvex, disconnected, or unbounded, and their centroids no longer reliably reflect the analytic affine normal. At parabolic points, the tangent-tangent Hessian is singular and the affine metric degenerates, so neither the analytic affine normal nor the moment-based construction yield a meaningful normal direction. Example~\ref{ex:nonconvex-counterexample} illustrates that in the nonconvex case the slice-centroid construction may remain descent in its normal component while still becoming geometrically unreliable because the slices cease to be convex bodies.

\textbf{Implications for descent directions in optimization.} The analytic affine normal remains formally well defined at any nondegenerate point of the hypersurface, but it represents an \textbf{inward} geometric direction only at elliptic points. Consequently, the analytic affine normal is a guaranteed strict descent direction for $f$ if each iterate $x_{k}$ lies on a locally strictly convex patch of the level set. By contrast, under the inward slice convention, the slice-centroid direction automatically carries a descent normal component whenever the centroid is well defined; however, outside the elliptic regime its slices may fail to be convex bodies and the resulting moment-based direction need not approximate the analytic affine normal. Within an elliptic neighborhood, the slice-centroid and cap-centroid constructions do provide consistent approximations of the analytic affine normal, since all three directions coincide up to scaling.

\subsection{Geometric illustration}

To close this section, we include a simple two-dimensional picture (Figure 1) that visualizes the difference between the Euclidean normal and the affine normal defined by the slice-centroid formula. On a convex curve (an ellipse is shown), at a point $p$, the \textbf{Euclidean normal} (red arrow) is given by the gradient $\nabla f(p)$, usually pointing outward, whereas the \textbf{affine normal} (blue arrow) can be obtained via the slice-centroid construction: shifting the tangent line inward and taking the centroid of the chord segment inside the ellipse. As $C\to 0$, the tangent direction of the centroid trajectory at $p$ gives the affine normal direction. For ellipses (affine spheres), the affine normals always point toward the center, showing the geometric distinction--Euclidean normal reflects local orthogonality, while affine normal encodes the global centroidal trend rather than merely local orthogonality.

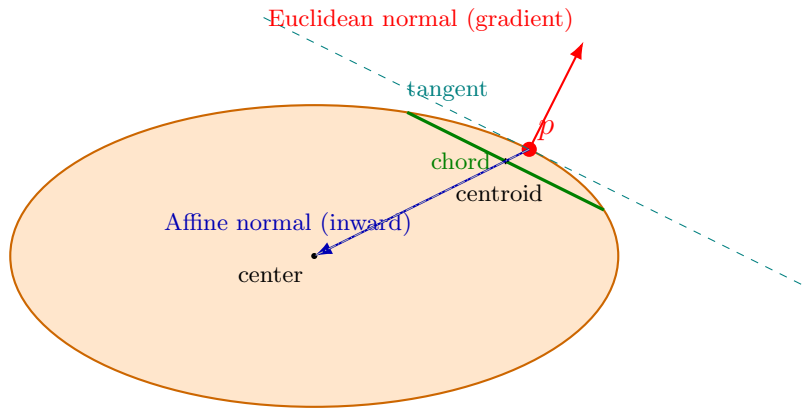
\begin{figure}[htbp]
  \centering
  \begin{tikzpicture}[scale=2,>=Latex]
  \def\a{2}
  \def\b{1}
  \path[name path=ellipse] (0,0) ellipse [x radius=\a, y radius=\b];
  \fill[orange!20] (0,0) ellipse [x radius=\a, y radius=\b];
  \draw[orange!80!black, thick] (0,0) ellipse [x radius=\a, y radius=\b];

  \fill (0,0) circle (0.02) node[below left] {\scriptsize center};

  \pgfmathsetmacro{\px}{\a/sqrt(2)}
  \pgfmathsetmacro{\py}{\b/sqrt(2)}
  \fill[red] (\px,\py) circle (0.05) node[above right=-1pt] {$p$};

  \pgfmathsetmacro{\gx}{2*\px/(\a*\a)}
  \pgfmathsetmacro{\gy}{2*\py/(\b*\b)}
  \pgfmathsetmacro{\gnorm}{sqrt(\gx*\gx+\gy*\gy)}
  \pgfmathsetmacro{\gux}{\gx/\gnorm}
  \pgfmathsetmacro{\guy}{\gy/\gnorm}

  \pgfmathsetmacro{\tx}{-\guy}
  \pgfmathsetmacro{\ty}{\gux}

  \draw[dashed, teal]
    (\px-2*\tx,\py-2*\ty) -- (\px+2*\tx,\py+2*\ty)
    node[pos=0.65,above=-1pt] {\scriptsize tangent};

  \def\s{0.14}
  \pgfmathsetmacro{\qx}{\px - \s*\gux}
  \pgfmathsetmacro{\qy}{\py - \s*\guy}
  \path[name path=pline] (\qx-2.2*\tx,\qy-2.2*\ty) -- (\qx+2.2*\tx,\qy+2.2*\ty);

  \path[name intersections={of=ellipse and pline, by={A,B}}];
  \draw[green!50!black, very thick] (A) -- (B) node[midway,left=1pt] {\scriptsize chord};

  \path let \p1=(A), \p2=(B) in coordinate (C) at ($ (A)!0.5!(B) $);
  \draw[line width=0.6pt] ($(C)+(-0.02,0)$) -- ($(C)+(0.02,0)$);
  \draw[line width=0.6pt] ($(C)+(0,-0.02)$) -- ($(C)+(0,0.02)$);
  \node[font=\scriptsize, anchor=north]
    at ($ (C) + (-0.09*\gux, -0.09*\guy) $) {centroid};

  \draw[->,red,thick] (\px,\py) -- ++(0.8*\gux,0.8*\guy)
    node[above left=-1pt] {\scriptsize Euclidean normal (gradient)};

  \draw[->,blue!70!black,thick] (\px,\py) -- (0,0)
    node[midway,below left=-1pt] {\scriptsize Affine normal (inward)};

  \draw[densely dashdotted, gray!60] (\px,\py) -- (0,0);

  \end{tikzpicture}
  \caption{Geometric comparison between the Euclidean and affine normals}
  \label{fig:affine-normal}
\end{figure}

\subsection{Computational complexity}

The analytic expression of the affine-normal direction
involves first-, second-, and third-order derivatives of
the objective function. In particular, evaluating the
analytic formula requires computing derivatives of the
Hessian, which may be computationally expensive in
high-dimensional settings.

In the present work our primary focus is on the geometric
structure and convergence properties of affine-normal
descent. The development of efficient techniques for
computing or approximating the affine-normal direction
in large-scale problems is an important topic for future
research.

\section{Equivalence of affine normal and Newton directions for strictly convex quadratics}\label{sec:newton-equivalencequadratic}

\begin{theorem}[Affine normal coincides with the Newton direction on strictly convex quadratics]\label{thm:ANeqNewton}
Let $f:\mathbb{R}^{n+1}\to\mathbb{R}$ be a strictly convex quadratic
\[
f(x)=\tfrac12\,x^\top A x + b^\top x + c,\qquad A\succ 0.
\]
For any $x$ with $\nabla f(x)\neq 0$, the affine normal direction $d_{\mathrm{AN}}(x)$ of the level set $\{y:\,f(y)=f(x)\}$ is collinear with the Newton direction
\[
d_{\mathrm{N}}(x)=-(\nabla^2 f(x))^{-1}\nabla f(x)=-A^{-1}\nabla f(x),
\]
that is, there exists $\lambda(x)>0$ such that
\[
d_{\mathrm{AN}}(x)=\lambda(x)\,d_{\mathrm{N}}(x).
\]
When $\nabla f(x)=0$, both vanish, trivially satisfying the claim.
\end{theorem}

\noindent\textbf{Proof 1 (Geometric Argument).}
When $A\succ0$, the level sets of $f$ are concentric ellipsoids centered at
$x^\star=-A^{-1}b$. By affine differential geometry, the affine normals of an ellipsoid point toward its center. Hence for any $x\neq x^\star$, the affine normal direction is along $x^\star-x$. Meanwhile,
\[
d_{\mathrm{N}}(x)=-A^{-1}\nabla f(x)=-A^{-1}(Ax+b)=x^\star-x,
\]
hence they are parallel. \qed

\medskip
\noindent\textbf{Proof 2 (Analytical Argument via Block Matrix and Schur Complement).}\\
\noindent\noindent\textbf{Step 1 (Change of basis and notation).}
At $x$, choose an orthonormal basis so that the last axis $e_{n+1}$ is parallel to $\nabla f(x)$ and the first $n$ axes are tangent to the level set. Write the Hessian $\nabla^2 f(x)=A$ in block form
\[
A=\begin{pmatrix}B & c\\ c^\top & d\end{pmatrix},
\]
where $B\in\mathbb{R}^{n\times n}$ (the tangent–tangent block $[f_{ij}]$), $c\in\mathbb{R}^n$ (the mixed normal–tangent block $[f_{n+1,i}]$), and $d=f_{n+1,n+1}$.

\noindent\noindent\textbf{Step 2 (Explicit affine-normal direction in this basis).}
In the normal-aligned coordinates, the affine normal reads
\[
d_{\mathrm{AN}}(x) \propto \begin{pmatrix} f^{ij}\big(-\tfrac{1}{n+2}\,|\nabla f|\, f^{pq} f_{pq i} + f_{n+1,i}\big) \\[2pt] -1 \end{pmatrix},
\]
where $[f^{ij}]=[f_{ij}]^{-1}$ is the inverse of the tangent–tangent block. For a quadratic form, all third derivatives vanish, hence $f_{pq i}\equiv 0$, and we get
\[
\boxed{\quad d_{\mathrm{AN}}(x) \propto \begin{pmatrix} B^{-1}c\\[2pt]-1\end{pmatrix}.\quad}
\]

\noindent\noindent\textbf{Step 3 (Parallelism with the Newton direction).}
In the same basis, $\nabla f=\|\nabla f\|\,e_{n+1}$. Therefore
\[
d_{\mathrm{N}}(x) \;=\; -A^{-1}\nabla f(x) \;=\; -\|\nabla f(x)\|\cdot \big(\text{last column of }A^{-1}\big).
\]
Let $S:=d-c^\top B^{-1}c>0$ be the Schur complement (since $A\succ0$). Then
\[
A^{-1}
=\begin{pmatrix}
B^{-1}+B^{-1}c\,S^{-1}c^\top B^{-1} & -\,B^{-1}c\,S^{-1}\\[2pt]
-\,S^{-1}c^\top B^{-1} & \ \ S^{-1}
\end{pmatrix},
\]
so the last column is $\big(-B^{-1}c\,S^{-1},\ S^{-1}\big)^\top$. Hence
\[
d_{\mathrm{N}}(x)= \|\nabla f(x)\|\,S^{-1}\begin{pmatrix} B^{-1}c\\[2pt]-1\end{pmatrix},
\]
which shows
\[
\boxed{\, d_{\mathrm{N}}\ \parallel\ d_{\mathrm{AN}} \,}
\]
and they differ only by the positive scalar $\|\nabla f\|\,S^{-1}$. \qed

\begin{corollary}[One-step convergence with exact line search]\label{cor:onestep}
Let $f(x)=\tfrac12 x^\top A x+b^\top x+c$ with $A$ symmetric positive definite and minimizer $x^\star=-A^{-1}b$.
For any $x$ with $\nabla f(x)\neq0$, an \emph{exact} line search along the affine-normal direction reaches $x^\star$ in one step.
\end{corollary}

\begin{proof}
Take any $x$ with $\nabla f(x)\neq0$. By Theorem~\ref{thm:ANeqNewton}, $d_{\mathrm{AN}}(x)$ is collinear with $d_{\mathrm{N}}(x)$.  For a quadratic,
\[
d_{\mathrm{N}}(x)=-(\nabla^2 f)^{-1}\nabla f(x)=-A^{-1}(Ax+b)=x^\star-x.
\]
Hence there exists a scalar $\lambda(x)>0$ such that $$d_{\mathrm{AN}}(x)=\lambda(x)d_{\mathrm{N}}(x) = \lambda(x)(x^\star-x).$$
Consider the univariate function  $\phi(\alpha):=f(x+\alpha d_{\mathrm{AN}}(x))$, Since $f$ is quadratic,
\[
\nabla f(x+\alpha d_{\mathrm{AN}}(x))
=
Ax+b+\alpha A d_{\mathrm{AN}}(x),
\]
and thus
\[
\phi'(\alpha)
=
d_{\mathrm{AN}}(x)^\top \bigl(Ax+b+\alpha A d_{\mathrm{AN}}(x)\bigr).
\]
Substituting $d_{\mathrm{AN}}(x)=\lambda(x)(x^\star-x)$ and 
$Ax+b=A(x-x^\star)$ gives
\[
\phi'(\alpha)
=
\lambda(x)(x^\star-x)^\top 
\bigl(A(x-x^\star)+\alpha \lambda(x) A (x^\star-x)\bigr)
=
\lambda(x)^2\,\|x^\star-x\|_A^2\,(\alpha \lambda(x)-1),
\]
where 
\(
\|v\|_A^2=v^\top A v.
\)
Since $\|x^\star-x\|_A^2>0$ and $\lambda(x)>0$, the derivative $\phi'(\alpha)$
vanishes if and only if
\[
\alpha^*=\frac{1}{\lambda(x)},
\]
which is the unique minimizer of $\phi$.  
The corresponding update is
\[
x^+
=
x+\alpha^* d_{\mathrm{AN}}(x)
=
x+\frac{1}{\lambda(x)}\lambda(x)(x^\star-x)
=
x^\star.
\]
If $\nabla f(x)=0$, then $x=x^\star$ and the statement holds trivially.
\end{proof}

This shows that, on strictly convex quadratic objectives,
the affine normal direction coincides with the Newton direction
(up to a positive scalar multiple).
In particular, with exact line search,
the resulting one-step update reaches the minimizer.
Hence, in the quadratic setting,
Newton direction appears as a special case
of the affine-normal direction framework.
 
\section{When is the affine normal a descent direction?}
\label{sec:AN-descent}

The affine normal direction is always well-defined when the 
tangent--tangent block of the Hessian is invertible. 
However, its \emph{descent} property depends crucially on the 
local convexity of the level set. 
In this section we characterize precisely when the affine normal is a 
strict descent direction, and explain why moment-based constructions 
(slice--centroid and cap--centroid) require convexity.

\subsection{Strict descent holds exactly at elliptic points}

\begin{theorem}[Strict descent at elliptic points]
\label{thm:ANDescent}
Let $f\in C^3$ and $\nabla f(z)\neq0$, and assume the tangent--tangent Hessian 
at $z$ is invertible. 
Then the analytic affine normal $d_{\mathrm{AN}}(z)$ satisfies
\[
\langle \nabla f(z),\, d_{\mathrm{AN}}(z)\rangle < 0 
\qquad\Longleftrightarrow\qquad 
z\ \text{is elliptic}.
\]
\end{theorem}

\begin{proof}
In a normal-aligned frame, the affine normal has the form
\[
d_{\mathrm{AN}}(z)=\lambda(z)\begin{pmatrix}\tau(z)\\ -1\end{pmatrix},
\]
where $\tau$ is determined by the tangent--tangent Hessian and 
\[
\lambda(z)=\bigl(\det h_{ij}(z)\bigr)^{-1/(n+2)}
\]
is the affine normalization factor.  
The sign of $\lambda(z)$ is determined by the sign of 
$\det(h_{ij})$.  
If $z$ is elliptic, then $h_{ij}$ is positive definite and 
$\det h_{ij}>0$, hence $\lambda(z)>0$.  
Since $\nabla f(z)=\|\nabla f(z)\|e_{n+1}$,
\[
\langle\nabla f(z),\,d_{\mathrm{AN}}(z)\rangle
=-\lambda(z)\|\nabla f(z)\|<0.
\]
If $z$ is hyperbolic, then $\det(h_{ij})<0$; 
no choice of $\lambda(z)$ produces an inward-pointing 
real affine normal.  
If $z$ is parabolic, $h_{ij}$ is singular and the affine normal 
is not defined. 
\end{proof}

\begin{remark}[Geometric meaning]
Ellipticity means that the level set is locally strictly convex; in this case 
the affine normal points strictly into the interior of the sublevel set and 
its normalization factor is positive.  
When the tangent--tangent Hessian is indefinite (hyperbolic points), the level 
set is not locally convex, and nearby slices may become nonconvex or 
multi-component.  The affine normal no longer represents an inward variational 
direction in this situation.  
At parabolic points the tangent--tangent Hessian is singular and the affine 
metric degenerates, so the affine normal is not well defined.
\end{remark}

\subsection{Moment-based constructions require convexity}

The analytic affine normal exists without convexity, 
but its \emph{moment-based} approximations require it.  
Indeed, the slice--centroid formulation assumes that 
the intersection of the level set with a nearby plane is 
a convex body with a well-behaved centroid trajectory.  
This fails when the level set is nonconvex.

\begin{example}[Disconnected slices in the nonconvex case]
\label{ex:nonconvex-counterexample}
Consider in $\mathbb{R}^2$ the nonconvex function
\[
f(x,y)=x^2(x^2-1)(x^2-4)(x^2-9)-y,\qquad z=(0,0),
\]
for which $f(z)=0$ and $\nabla f(z)=(0,-1)$.  
The tangent line at $z$ is $y=0$, and
\[
\Omega_z=\{(x,y):f(x,y)\le0\}
={\Bigl\{}(x,y):y\ge x^2(x^2-1)(x^2-4)(x^2-9){\Bigr\}}
\]
is nonconvex.
For $C<0$, the slice
\[
S_C
=
{\Bigl\{}(x,-C):x^2(x^2-1)(x^2-4)(x^2-9)\le -C{\Bigr\}}
\]
may already be a union of several disconnected symmetric intervals.
For example, when $C=-24$, the slice is disconnected, as shown in
Figure~\ref{fig:nonconvex-disconnected-slices}. By symmetry, whenever the
centroid exists it satisfies $g(C)=(0,-C)$.
Thus the slice-centroid direction
\[
\widehat d_{\mathrm{SC}}(z)
\propto \frac{g(C)-z}{-C}
=(0,1)
\]
satisfies
\[
\langle \nabla f(z), \widehat d_{\mathrm{SC}}(z) \rangle = -1 < 0,
\]
so under the inward convention it is indeed a descent direction.
However, the slices are \emph{disconnected} and not convex bodies.
Thus the moment-based construction is no longer a faithful local geometric proxy
for the analytic affine normal: the issue in the nonconvex case is not
automatic loss of descent, but rather the loss of convexity, connectedness,
and geometric well-posedness of the slices.
\end{example}

\begin{figure}[!ht]
\centering
\includegraphics[width=0.74\linewidth]{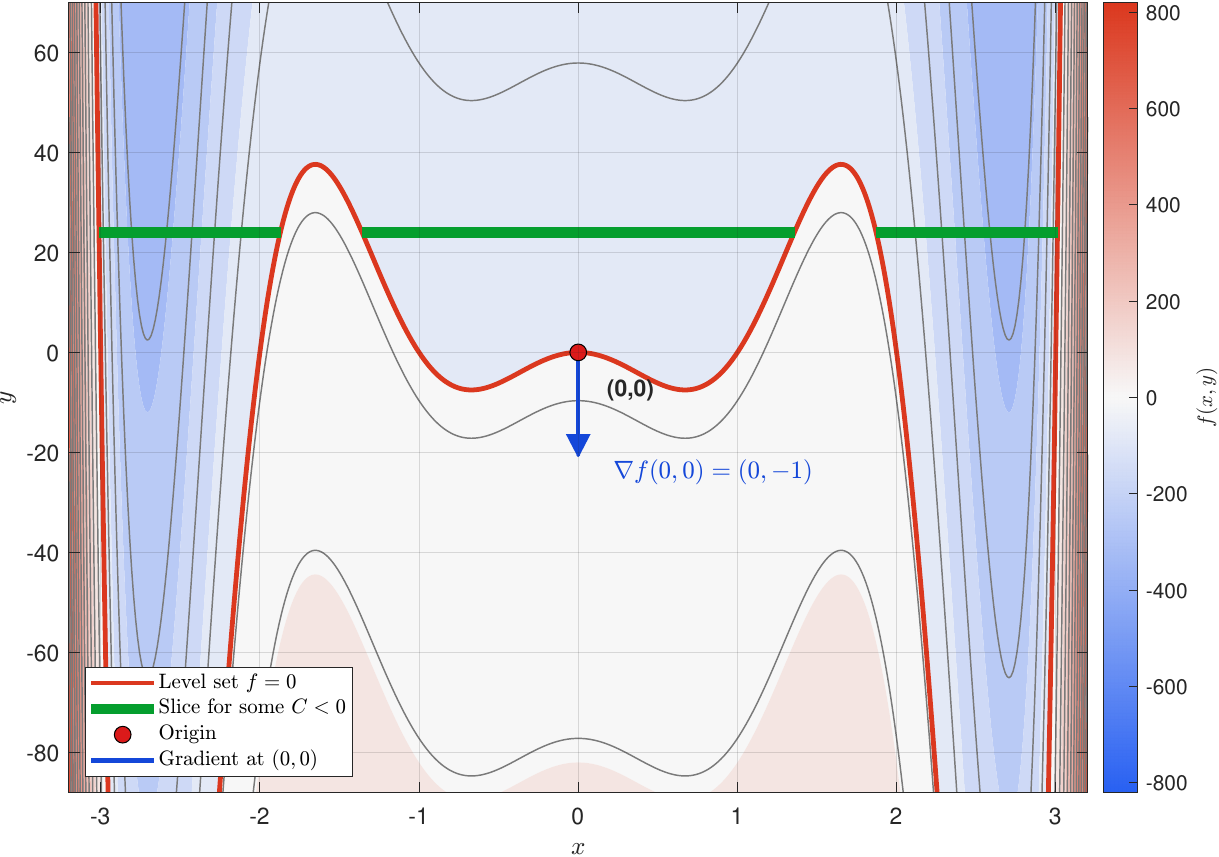}
\caption{Disconnected inward slices for
$f(x,y)=x^2(x^2-1)(x^2-4)(x^2-9)-y$ at $z=(0,0)$.
The red curve is the zero level set $f(x,y)=f(z)$, and the green set is a
representative slice $S_C$ for some $C<0$. Under the inward convention,
the centroid direction remains inward/descent, but the slice is disconnected and
thus ceases to be a faithful local geometric proxy for the analytic affine normal.}
\label{fig:nonconvex-disconnected-slices}
\end{figure}

\section{Examples for computing the affine normal}\label{sec:examples}

Although the affine normal originates from affine differential geometry,
its computation can be made completely explicit in standard optimization
settings. To facilitate understanding for readers in optimization,
we begin with a few low--dimensional examples in which the affine normal
can be computed analytically and directly compared with classical
directions such as the gradient and the Newton.

\subsection{Quadratic convex function in two Variables}
Let
\[
f(x,y)=\tfrac12\,(x^2+4y^2)\;-\;x\;-\;4y,
\]
with $A=\mathrm{diag}(1,4)\succ0,\ b=(-1,-4)$, minimizing at $x^\star=(1,1)$.
Pick $p=(2,0)$.

\medskip
\noindent\textbf{Gradient and Hessian.}
\[
\nabla f(x,y)=(x-1,\ 4y-4),\qquad
\nabla f(p)=(1,-4),\qquad
H=\nabla^2 f=A=\begin{pmatrix}1&0\\0&4\end{pmatrix}.
\]

\medskip
\noindent\textbf{Unit normal/tangent ($n=1$).}
\[
\|\nabla f(p)\|=\sqrt{17},\quad
\hat{\boldsymbol n}=\frac{\nabla f(p)}{\|\nabla f(p)\|}=\Big(\tfrac{1}{\sqrt{17}},-\tfrac{4}{\sqrt{17}}\Big),\quad
\hat{\boldsymbol t}=\Big(\tfrac{4}{\sqrt{17}},\tfrac{1}{\sqrt{17}}\Big).
\]

\medskip
\noindent\textbf{Recall ($n=1$).}
In the orthonormal frame $(\hat{\boldsymbol t},\hat{\boldsymbol n})$,
\[
\boxed{\quad \boldsymbol d_{\rm AN}\ \propto\ (\tau,\,-1), \qquad
\tau=\frac{f_{21}}{f_{11}}-\frac{\|\nabla f\|}{3}\,\frac{f_{111}}{f_{11}^2}\quad}
\]
with
\[
f_{11}=D^2 f[\hat{\boldsymbol t},\hat{\boldsymbol t}],\quad
f_{21}=D^2 f[\hat{\boldsymbol n},\hat{\boldsymbol t}],\quad
f_{111}=D^3 f[\hat{\boldsymbol t},\hat{\boldsymbol t},\hat{\boldsymbol t}].
\]

\medskip
\noindent\textbf{Directional derivatives (quadratic case).}
Here $f_{111}=0$, so $\tau={f_{21}}/{f_{11}}$ with
\[
f_{11}=\hat{\boldsymbol t}^{\!\top} H \hat{\boldsymbol t}=\tfrac{20}{17},\quad
f_{21}=\hat{\boldsymbol n}^{\!\top} H \hat{\boldsymbol t}=-\tfrac{12}{17},\quad
\tau=-\tfrac{3}{5}.
\]

\medskip
\noindent\textbf{Affine-normal direction.}
\[
\boldsymbol d_{\mathrm{AN}}
\ \propto\
\tau\,\hat{\boldsymbol t}-\hat{\boldsymbol n}
=\frac{1}{\sqrt{17}}(-3.4,\,3.4)
\ \parallel\ (-1,\,1).
\]

\medskip
\noindent\textbf{Newton direction.}
\[
\boldsymbol d_{N}
= -H^{-1}\nabla f(p)
= (-1,\,1).
\]
Thus $\boldsymbol d_{\mathrm{AN}}\parallel \boldsymbol d_{\mathrm{N}}$.

\subsection{Quadratic convex function in three variables}
Let
\[
f(x,y,z)=\tfrac12\,(x^2+4y^2+9z^2)\;+\;(-1,\,0,\,0)\!\cdot\!(x,y,z),
\]
so $A=\mathrm{diag}(1,4,9)\succ0,\ b=(-1,0,0)$ and $x^\star=(1,0,0)$.
Take $p=(2,0,0)$.

\medskip
\noindent\textbf{Gradient and Hessian.}
\[
\nabla f(x,y,z)=(x-1,\;4y,\;9z),\qquad \nabla f(p)=(1,0,0),\qquad
H=\mathrm{diag}(1,4,9).
\]

\medskip
\noindent\textbf{Normal alignment and tangents ($n=2$).}
\[
\hat{\boldsymbol n}=(1,0,0),\qquad
\hat{\boldsymbol t}_1=(0,1,0),\quad \hat{\boldsymbol t}_2=(0,0,1).
\]

\medskip
\noindent\textbf{Blocks.}
\[
B=\mathrm{diag}(4,9),\quad c=(0,0).
\]

\medskip
\noindent\textbf{Affine-normal direction ($n=2$).}
For quadratics (no third derivatives),
\[
\boldsymbol d_{\rm AN}\ \propto\ \sum_{i=1}^2 (B^{-1}c)_i\,\hat{\boldsymbol t}_i\ -\ \hat{\boldsymbol n}
= -\,\hat{\boldsymbol n}=(-1,0,0).
\]

\medskip
\noindent\textbf{Newton direction.}
\[
\boldsymbol d_{N}=-H^{-1}\nabla f(p)=(-1,0,0).
\]
Thus $\boldsymbol d_{\rm AN}\parallel\boldsymbol d_{\mathrm{N}}$.

\subsection{Strictly convex non-quadratic example}
Consider $n=1$ and
\[
f(x,y)=\frac12\,x^2+2y^2+\frac1{12}x^4 .
\]
Then
\[
\nabla^2 f(x,y)=
\begin{pmatrix}
1+x^2 & 0\\
0 & 4
\end{pmatrix}\succ0,
\]
so $f$ is strictly convex. At $p=(1,1)$, in the orthonormal frame $(\hat{\boldsymbol t},\hat{\boldsymbol n})$ the affine normal direction is
\[
\boxed{\quad \boldsymbol d_{\rm AN}\ \propto\ (\tau,\,-1), \qquad
\tau=\frac{f_{21}}{f_{11}}-\frac{\|\nabla f\|}{3}\,\frac{f_{111}}{f_{11}^2}\quad}
\]
with the ingredients computed in the text; numerically one finds $\tau\approx0.7687$ and
\[
\boldsymbol d_{\rm AN}\ \propto\ (-1.0454,\,-0.7056),
\qquad
\nabla f(1,1)\cdot \boldsymbol d_{\rm AN}\approx -4.2164<0,
\]
confirming strict descent.

\section{YAND}
\label{sec:AND}

We now transition from the geometric theory of affine normals to the 
optimization algorithm that uses these directions as search directions.  
Since the analytic affine normal is defined only up to a nonzero scalar 
(and in particular, up to sign), we must ensure that the chosen direction
is always a descent direction for $f$.

\subsection{Affine normal descent direction}\label{subsec:dk}

Given an iterate $x_k$ with $\nabla f(x_k)\neq 0$, we first compute the affine normal 
$d_{\mathrm{AN}}(x_k)$ of the level set $\{f=f(x_k)\}$ and define the search direction by
\begin{equation}\label{eq:dk}
d_k := 
\begin{cases}
    d_{\mathrm{AN}}(x_k), & \text{if } \langle \nabla f(x_k), d_{\mathrm{AN}}(x_k)\rangle < 0,\\[1.5mm]
    -d_{\mathrm{AN}}(x_k), & \text{if } \langle \nabla f(x_k), d_{\mathrm{AN}}(x_k)\rangle > 0,\\[1.5mm]
    -\nabla f(x_k)/\|\nabla f(x_k)\|, & \text{otherwise}.
\end{cases}
\end{equation}
The third case corresponds to an affine-degenerate point where the equi--affine curvature vanishes and the affine normal collapses into the tangent space.
At such points, any vector $v_{\mathrm{tan}}$ in the tangent space satisfies
$\langle \nabla f(x_k), v_{\mathrm{tan}} \rangle = 0$, so the affine normal (even if formally computable) cannot serve as a descent direction.
Since no representative with a negative normal component exists, we fall back to the steepest–descent direction $-\nabla f(x_k)/\|\nabla f(x_k)\|$. The sign correction ensures
\[
\langle \nabla f(x_k), d_k\rangle < 0,
\]
so $d_k$ is always a strict descent direction of $f$ at $x_k$.

\medskip
To analyze the geometry of $d_k$, we work in the \emph{normal–aligned frame} at $x_k$.  
Let
\[
e_{n+1}:=\frac{\nabla f(x_k)}{\|\nabla f(x_k)\|},
\]
where $\{e_i\}_{i=1}^n$ is an orthonormal basis for the tangent space of the level set $\{f=f(x_k)\}$.
In this frame, every descent direction $d_k$ constructed from \eqref{eq:dk} 
can be expressed as
\[
d_k = \sum_{i=1}^n (\tau_k)_i\, e_i \; - \; e_{n+1}, 
\qquad \tau_k \in \mathbb{R}^n,
\]
as illustrated in Figure~\ref{fig:dk}.
Let $T_k := \|\tau_k\|$.  
Then 
\[
\|d_k\|^2 = 1 + T_k^2,
\]
so the scalar $T_k$ measures the tangential magnitude of the affine–normal direction.
Its boundedness plays a central role in the convergence analysis developed in Section~\ref{sec:YAND-convergence}.
\begin{figure}[ht!]
    \centering
    \begin{tikzpicture}[scale=0.86,
        axis/.style={->,thick},
        levelset/.style={thick, draw=blue},
        gradvec/.style={->,thick, draw=red!70!black},
        gradvecdash/.style={->,thick, draw=red!70!black, dashed},
        anvec/.style={->,thick, draw=green!60!black},
        anvecdash/.style={->,thick, draw=green!60!black, dashed},
        projline/.style={dashed},
        >=Stealth,
        every node/.style={font=\footnotesize}
    ]

    \begin{scope}[shift={(-5.8,0)}]
        \draw[axis] (-1.9,0) -- (1.9,0) node[right] {$e_i$};
        \draw[axis] (0,-1.6) -- (0,1.8) node[above] {$e_{n+1}$};

        \draw[levelset, domain=-1.4:1.4, smooth] plot (\x, {-0.1*(\x)^2});
        \node at (-1.45,-0.3) {\color{blue}{$\{f=f(x_k)\}$}};

        \fill (0,0) circle (1.5pt) node[below left=2pt] {$x_k$};

        \draw[gradvec] (0,0) -- (0,1.5)
            node[pos=0.9,left=3pt] {\color{red!70!black}{$\nabla f(x_k)$}};

        \draw[anvec] (0,0) -- (1.1,-1.1)
            node[pos=0.8,below right=4pt] {\color{green!60!black}{$d_k=d_{\mathrm{AN}}(x_k)$}};

        \draw[projline] (1.1,-1.1) -- (1.1,0);  
        \draw[projline] (1.1,-1.1) -- (0,-1.1); 

        \node[above=1pt] at (1.1,0) {$(\tau_k)_i$};
        \node[left=1pt]  at (0,-1.1) {$-1$};

        \node at (0,-1.9) {Case 1: $\langle \nabla f(x_k), d_{\mathrm{AN}}(x_k)\rangle < 0$};
    \end{scope}

    \begin{scope}[shift={(0,0)}]
        \draw[axis] (-1.9,0) -- (1.9,0) node[right] {$e_i$};
        \draw[axis] (0,-1.6) -- (0,1.8) node[above] {$e_{n+1}$};

        \draw[levelset, domain=-1.4:1.4, smooth] plot (\x, {-0.1*(\x)^2});
        \node at (-1.45,-0.3) {\color{blue}{$\{f=f(x_k)\}$}};

        \fill (0,0) circle (1.5pt) node[below left=2pt] {$x_k$};

        \draw[gradvec] (0,0) -- (0,1.5)
            node[pos=0.9,left=3pt] {\color{red!70!black}{$\nabla f(x_k)$}};

        \draw[anvecdash] (0,0) -- (-1.1,1.1)
            node[pos=0.8,below left=2pt] {\color{green!60!black}{$d_{\mathrm{AN}}(x_k)$}};

        \draw[anvec] (0,0) -- (1.1,-1.1)
            node[pos=0.8,below right=4pt] {\color{green!60!black}{$d_k = -d_{\mathrm{AN}}(x_k)$}};
            
        \draw[projline] (1.1,-1.1) -- (1.1,0);  
        \draw[projline] (1.1,-1.1) -- (0,-1.1); 

        \node[above=1pt] at (1.1,0) {$(\tau_k)_i$};
        \node[left=1pt]  at (0,-1.1) {$-1$};

        \node at (0,-1.9) {Case 2: $\langle \nabla f(x_k), d_{\mathrm{AN}}(x_k)\rangle > 0$};
    \end{scope}

    \begin{scope}[shift={(5.8,0)}]
        \draw[axis] (-1.9,0) -- (1.9,0) node[right] {$e_i$};
        \draw[axis] (0,-1.6) -- (0,1.8) node[above] {$e_{n+1}$};

        \draw[levelset, domain=-1.4:1.4, smooth] plot (\x, {-0.1*(\x)^2});
        \node at (-1.45,-0.3) {\color{blue}{$\{f=f(x_k)\}$}};

        \fill (0,0) circle (1.5pt) node[below left=2pt] {$x_k$};

        \draw[gradvec] (0,0) -- (0,1.5)
            node[pos=0.9,left=3pt] {\color{red!70!black}{$\nabla f(x_k)$}};
        \draw[anvec] (0,0) -- (0,-1.2)
            node[pos=0.9,right=3pt] {\color{green!60!black}{$d_k = -\nabla f(x_k)/\|\nabla f(x_k)\|$}};

        \node[above=1pt] at (0.9,-0.05) {$(\tau_k)_i = 0$};
        \node[left=1pt]  at (0,-1.1) {$-1$};

        \node at (0,-1.9)
            {Case 3: $\langle \nabla f(x_k), d_{\mathrm{AN}}(x_k)\rangle = 0$};
    \end{scope}

    \end{tikzpicture}
    \caption{
Normal–aligned frame at $x_k$ illustrating three typical constructions of $d_k$.
The analytic affine normal $d_{\mathrm{AN}}(x_k)$ is represented in the frame
$\{e_1,\dots,e_n,e_{n+1}\}$ with its $(n+1)$-st component normalized to $-1$.
Case~1: the affine normal is already a descent direction 
($\langle\nabla f(x_k), d_{\mathrm{AN}}(x_k)\rangle<0$), hence $d_k=d_{\mathrm{AN}}(x_k)$.
Case~2: the affine normal points uphill 
($\langle\nabla f(x_k), d_{\mathrm{AN}}(x_k)\rangle>0$), so we flip the sign and set 
$d_k=-d_{\mathrm{AN}}(x_k)$.
Case~3: the affine normal is orthogonal to the gradient 
($\langle\nabla f(x_k), d_{\mathrm{AN}}(x_k)\rangle=0$); in this degenerate case we revert to 
the steepest–descent direction 
$d_k=-\nabla f(x_k)/\|\nabla f(x_k)\|$.
}
\label{fig:dk}
\end{figure}
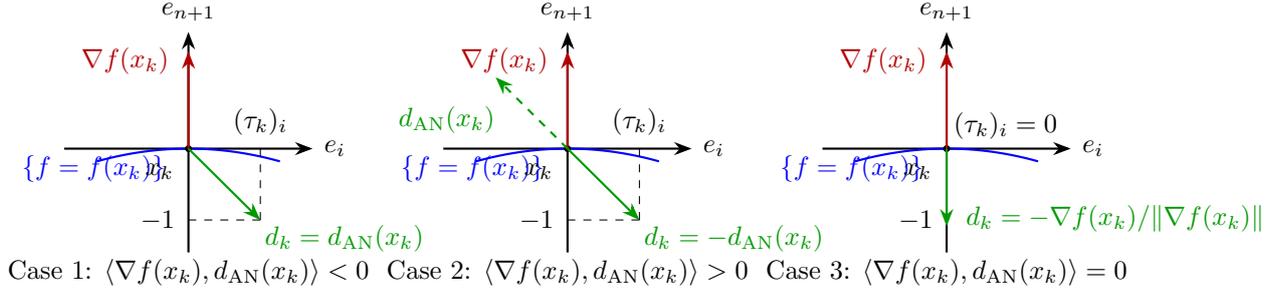

\begin{remark}[Why using the \emph{negative} affine--normal direction in Case~2]
The sign correction in \eqref{eq:dk} ensures descent by flipping the affine normal direction when $\langle \nabla f(x_k), d_{\mathrm{AN}}(x_k) \rangle > 0$. This is not merely an algorithmic fix, and rather is governed by a deeper geometric justification. Roughly speaking, at non-elliptic points, the level set possesses saddle-like or degenerate geometry, which implies that while the analytical affine normal may point uphill, its construction persists to encode the ``valuable" curvature information. Specifically, the affine normal direction $d_{\mathrm{AN}}$ is derived from the affine metric and the cubic form, which together suppress contributions from directions where the level set curvature is extreme (whether positive or negative). Consequently, the negative direction $-d_{\mathrm{AN}}$ is aligned with the axis along which the Monge-Ampère measure of the local sublevel set, governed by $\det(\nabla^2 f)$, contracts most sharply. This makes $-d_{\mathrm{AN}}$ a geometry-aware descent direction that is often more stable and effective than both the raw negative gradient (which ignores the curvature) and the Newton direction (which can be misled by negative eigenvalues). Therefore, ``\textit{the sign flip}", leverages the geometric strength of the affine normal construction, even outside the elliptic regime. 

Note that the affine normal is fundamentally linked to the Monge-Ampère operator defined by the affine metric. At a point \( x \), the Monge-Ampère measure of the sublevel set \( \{\,y: f(y) \le f(x)\,\} \), to second order, is proportional to \( \det(\nabla^2 f(x)) \). The affine normal direction \( \xi \) is characterized by the condition that the volume form \( \omega_\xi \) induced by \( \xi \) is parallel with respect to the affine connection. When \( \langle \nabla f, \xi \rangle > 0 \), the vector \( -\xi \) points in the direction where this canonical volume form contracts most rapidly, i.e., the direction along which the local sublevel-set volume decreases fastest, providing a geometrically intrinsic and stable descent direction even in nonconvex regions.
\end{remark}

\subsection{YAND algorithm}
We now summarize the proposed YAND algorithm as follows:
\begin{enumerate}
  \item Initialize $x_0$, tolerance $\varepsilon>0$, and a step strategy (exact line search / strong Wolfe / Armijo).
  \item For $k=0,1,2,\dots$:
  \begin{enumerate}
    \item If $\|\nabla f(x_k)\|\le\varepsilon$, stop.
    \item Compute the descent direction $d_k$ at $x_k$
          according to \eqref{eq:dk}.
    \item Line search: take $\alpha_k>0$ satisfying one of
    \begin{itemize}
        \item \textbf{Exact:} $\alpha_k\in\arg\min_{\alpha\ge0} f(x_k+\alpha d_k)$.
        \item \textbf{Armijo \cite{Armijo1966}:}
          \[
          f(x_k+\alpha_k d_k)\le f(x_k)+\sigma\alpha_k\,\nabla f(x_k)^\top d_k, \quad (0<\sigma <1).
          \]
        \item \textbf{Strong Wolfe \cite{Wolfe1969}:}
          \[
          \begin{cases}
          f(x_k+\alpha_k d_k)\le f(x_k)+c_1\alpha_k\,\nabla f(x_k)^\top d_k,\\
          |\nabla f(x_k+\alpha_k d_k)^\top d_k|\le c_2|\nabla f(x_k)^\top d_k|
          \end{cases}\quad (0<c_1<c_2<1).
          \]
    \end{itemize}
    \item Update $x_{k+1}=x_k+\alpha_k d_k$.
  \end{enumerate}
\end{enumerate}

\begin{remark}[Gradient is still required]
Although the search direction $d_k$ is defined by the affine normal, the gradient
$\nabla f(x_k)$ remains essential: 
(i) the affine normal construction requires the Euclidean normal of the level set,
which is exactly $\nabla f(x_k)$; 
(ii) line–search conditions (Armijo, strong Wolfe, exact) all depend on 
$\nabla f(x_k)^\top d_k$; 
(iii) the affine normal modifies and preconditions the Newton direction, 
but does not replace the role of the gradient in descent verification, especially for nonconvex cases.
Thus YAND does not eliminate gradient evaluations, but rather uses them
more geometrically and more robustly.
\end{remark}

\paragraph*{Implementation notes}
\begin{itemize}
  \item \textbf{Stability:} When the Hessian degenerates (tangent–tangent block not invertible), switch to slice-centroid, cap-centroid or regularize the Hessian (e.g., trust-region/Levenberg-Marquardt).
  \item \textbf{Higher derivatives:} Use Auto-Differentiation (AD) or finite differences along tangential directions for the third derivatives; prefer strong Wolfe when noise is present.
  \item \textbf{Armijo backtracking:} For Armijo backtracking, we set parameters 
$0<\sigma<1$, $0<\beta<1$, and an initial trial stepsize $\alpha_0>0$ (e.g., $\alpha_0 = 1$). 
Given a descent direction $d_k$, 
choose the smallest integer $m\ge0$ such that
\[
   f(x_k+\alpha_m d_k)
   \;\le\;
   f(x_k)
   + \sigma\,\alpha_m \langle \nabla f(x_k), d_k\rangle,
   \qquad
   \alpha_m = \beta^m \alpha_0 .
\]
This guarantees sufficient decrease for any direction satisfying 
$\langle \nabla f(x_k), d_k\rangle < 0$.

\item \textbf{BB initialization (optional):}
By default we set the initial stepsize to a fixed value, e.g.\ $\alpha_0=1$.
Optionally, for $k\ge1$, a Barzilai--Borwein (BB) estimate may be used to
initialize the line search.  
Let $s_{k-1}=x_k-x_{k-1}$ and 
$y_{k-1}=\nabla f(x_k)-\nabla f(x_{k-1})$.  
The classical BB formulas are
\[
   \alpha_k^{\mathrm{BB1}}
   = \frac{s_{k-1}^\top s_{k-1}}
          {s_{k-1}^\top y_{k-1}},
   \qquad
   \alpha_k^{\mathrm{BB2}}
   = \frac{s_{k-1}^\top y_{k-1}}
          {y_{k-1}^\top y_{k-1}} .
\]
When $s_{k-1}^\top y_{k-1}>0$, we compute a safeguarded BB value
\[
   \alpha_0
   = \min\{\alpha_{\max},
      \max\{\alpha_{\min},\,\alpha_k^{\mathrm{BB}}\}\},
\]
where $\alpha_k^{\mathrm{BB}}$ denotes either $\alpha_k^{\mathrm{BB1}}$ or
$\alpha_k^{\mathrm{BB2}}$,
and $0<\alpha_{\min}<\alpha_{\max}$ are fixed bounds.
If $s_{k-1}^\top y_{k-1}\le0$, or if $k=0$, we simply use the default
$\alpha_0=1$.
\end{itemize}

\paragraph{Quadratic specialization}
Combined with the quadratic equivalence established in Section~\ref{sec:newton-equivalencequadratic},
we conclude that YAND coincides with Newton's method
on strictly convex quadratic objectives
under exact line search.

\section{Convergence analysis of YAND}
\label{sec:YAND-convergence}

In this section we develop a full convergence theory for YAND algorithm.
We proceed in a hierarchy of assumptions, starting from the most restrictive
(strongly convex and smooth), and gradually relaxing to nonconvex settings,
and finally establishing local quadratic convergence under classical nondegeneracy conditions.

\subsection{Preliminaries}
Throughout this section, $L$–smoothness refers to 
\[
\|\nabla f(x)-\nabla f(y)\|\le L\|x-y\|, \qquad\forall x,y.
\]
We impose the following mild geometric assumption on the affine normal direction.
\begin{assumption}[T--boundedness: Uniformly bounded affine-normal direction]
\label{assump:T-bounded}
Let $T_k := \|\tau_k\|$ denote the tangential magnitude of $d_k$ in the normal–aligned frame.
We suppose that there exists a constant $T<\infty$ such that $T_k \le T$ for all iterates $k$. 
In particular,
\begin{equation}
    \label{eq:T-bounded}
\|d_k\|^2 = 1 + T_k^2 \;\leq\; 1+T^2.
\end{equation}
\end{assumption}
\begin{remark}
The boundedness of ${d_k}$ is a mild assumption and is automatic whenever the level sets of $f$ admit uniformly bounded third-order affine geometry (e.g., bounded affine metric, bounded cubic form, and nondegenerate Hessian).
In particular, this condition holds whenever $\nabla^2 f$ is uniformly positive definite and Lipschitz continuous on the relevant level sets, which includes strongly convex functions with Lipschitz continuous Hessian on level sets. Indeed, the analytic affine normal is constructed from the inverse Hessian and third-order derivatives of $f$; hence whenever the level sets avoid degeneracy and their cubic form is uniformly controlled, the resulting affine-normal vector remains bounded along the entire trajectory.
Thus Assumption~\ref{assump:T-bounded} merely excludes pathological degeneracies of the affine metric and does not restrict typical optimization problems.
\end{remark}

\begin{lemma}[Angle bound]
\label{lem:angle-from-T}
Suppose Assumption~\ref{assump:T-bounded} holds.
Let $\theta_k$ denote the angle between $-\,\nabla f(x_k)$ and $d_k$, i.e.,
\begin{equation}\label{eq:theta-def}
\cos\theta_k
:=
\frac{-\langle \nabla f(x_k), d_k\rangle}{\|\nabla f(x_k)\|\,\|d_k\|},
\qquad \theta_k\in[0,\pi].
\end{equation}
Then there exists a constant $c>0$ such that
\[
\cos\theta_k \;\ge\; c
\qquad\text{for all }k.
\]
In particular, we may take
\[
c \;=\; \frac{1}{\sqrt{1+T^2}}.
\]
Consequently, the uniform angle condition holds:
\begin{equation}\label{eq:angle-condition}
-\langle \nabla f(x_k), d_k\rangle 
\;\ge\;
c\,\|\nabla f(x_k)\|\,\|d_k\|, \qquad \forall k.
\end{equation}
\end{lemma}

\begin{proof}
By construction of the normal-aligned frame at $x_k$, we have
\[
e_{n+1} = \frac{\nabla f(x_k)}{\|\nabla f(x_k)\|},
\qquad
d_k = \sum_{i=1}^n (\tau_k)_i\, e_i - e_{n+1},
\]
so the $(n+1)$-st component of $d_k$ is $-1$. Hence
\[
\langle \nabla f(x_k), d_k\rangle
=
\|\nabla f(x_k)\|\langle e_{n+1}, \sum_{i=1}^n (\tau_k)_i\, e_i - e_{n+1}\rangle
=
-\|\nabla f(x_k)\|.
\]
Moreover, by definition of $T_k$,
\[
\|d_k\|^2
=
\|\tau_k\|^2 + 1
=
1 + T_k^2.
\]
Substituting into \eqref{eq:theta-def} yields
\[
\cos\theta_k
=
\frac{-\langle \nabla f(x_k), d_k\rangle}{\|\nabla f(x_k)\|\,\|d_k\|}
=
\frac{1}{\|d_k\|}
=
\frac{1}{\sqrt{1+T_k^2}}.
\]
Assumption~\ref{assump:T-bounded} ensures $T_k\le T$ for all $k$, hence
\[
\cos\theta_k
=
\frac{1}{\sqrt{1+T_k^2}}
\;\ge\;
\frac{1}{\sqrt{1+T^2}}
=: c > 0.
\]
Finally, plugging this lower bound on $\cos\theta_k$ back into
\eqref{eq:theta-def} gives
\[
-\langle \nabla f(x_k), d_k\rangle
=
\|\nabla f(x_k)\|\,\|d_k\|\,\cos\theta_k
\;\ge\;
c\,\|\nabla f(x_k)\|\,\|d_k\|.
\]
\end{proof}

\begin{lemma}[Armijo step lower bound and one--step decrease]\label{lem:Armijo-lower}
Assume that $f$ is $L$--smooth, and let $d_k$ be any descent direction satisfying 
$\langle \nabla f(x_k), d_k\rangle < 0$.
Then with with Armijo parameter $\sigma \in (0,1)$ and backtracking ratio $\beta\in [1/2, 1)$, the Armijo backtracking step $\alpha_k$ satisfies
\[
\alpha_k
\;\ge\;
\frac{1-\sigma}{L}\cdot 
\frac{-\langle \nabla f(x_k), d_k\rangle}{\|d_k\|^2},
\]
and the corresponding iterate obeys the decrease estimate
\[
f(x_{k+1})
\;\le\;
f(x_k)
-
\frac{\sigma(1-\sigma)}{L}\cdot
\frac{\langle \nabla f(x_k), d_k\rangle^2}{\|d_k\|^2}.
\]
\end{lemma}

\begin{proof}
Let $g_k = \nabla f(x_k)$ and consider the univariate function 
$\phi_k(\alpha) := f(x_k+\alpha d_k)$. 
By $L$--smoothness,
\[
\phi_k(\alpha)
\;\le\;
\phi_k(0) + \alpha \phi_k'(0)
+ \tfrac12 L \alpha^2 \|d_k\|^2, 
\qquad \forall \alpha \ge 0.
\]
Since $d_k$ is a descent direction, 
$\phi_k'(0)=\langle g_k,d_k\rangle < 0$.

\noindent\textbf{Step 1: Sufficient condition for Armijo.}
The Armijo condition reads
\[
\phi_k(\alpha)
\;\le\;
\phi_k(0) + \sigma \alpha \phi_k'(0).
\]
Using the smoothness bound, a sufficient condition is
\[
\tfrac12 L \alpha^2 \|d_k\|^2 
+ \alpha \phi_k'(0)
\;\le\;
\sigma \alpha \phi_k'(0),
\]
which rearranges to
\[
\alpha
\;\le\;
\frac{2(1-\sigma)}{L}\cdot 
\frac{-\phi_k'(0)}{\|d_k\|^2}.
\]

\noindent\textbf{Step 2: Lower bound on the accepted step.}
Let $\alpha_k$ be the first step satisfying Armijo in backtracking with 
ratio $\beta\ge 1/2$.
Minimality of $\alpha_k$ implies that $\alpha_k/\beta$ fails Armijo.
Using the sufficient condition above, we obtain
\[
\frac{\alpha_k}{\beta}
\;>\;
\frac{2(1-\sigma)}{L}\cdot 
\frac{-\phi_k'(0)}{\|d_k\|^2}.
\]
Thus
\[
\alpha_k
\;\ge\;
\frac{2\beta(1-\sigma)}{L}\cdot 
\frac{-\phi_k'(0)}{\|d_k\|^2}
\;\ge\;
\frac{1-\sigma}{L}\cdot 
\frac{-\phi_k'(0)}{\|d_k\|^2},
\]
where the last inequality uses $\beta\ge 1/2$.
Recalling $\phi_k'(0)=\langle g_k,d_k\rangle$ proves the first claim:
\[
\alpha_k
\;\ge\;
\frac{1-\sigma}{L}\cdot 
\frac{-\langle \nabla f(x_k), d_k\rangle}{\|d_k\|^2}.
\]

\noindent\textbf{Step 3: Decrease of $f$.}
Applying the Armijo condition at $\alpha_k$:
\[
f(x_{k+1})
=
f(x_k+\alpha_k d_k)
\;\le\;
f(x_k) + \sigma \alpha_k \langle g_k,d_k\rangle.
\]
Using the lower bound on $\alpha_k$ from Step~2,
\[
f(x_{k+1})
\;\le\;
f(x_k)
-
\frac{\sigma(1-\sigma)}{L}
\cdot
\frac{\langle g_k,d_k\rangle^2}{\|d_k\|^2}.
\]
This completes the proof.
\end{proof}

\begin{theorem}[Global convergence under Armijo backtracking]
\label{thm:global-conv-Armijo}
Let $f:\R^{n+1}\to\R$ be continuously differentiable, $L$--smooth, and bounded below on $\R^{n+1}$.
Suppose that $\{x_k\}$ is generated by
\[
x_{k+1} = x_k + \alpha_k d_k,
\]
where the step sizes $\alpha_k$ are obtained by Armijo backtracking with parameters
$\sigma\in(0,1)$ and $\beta\in[1/2,1)$, and the directions $d_k$ satisfy
\[
\langle \nabla f(x_k), d_k\rangle < 0
\quad\text{for all }k.
\]
Assume in addition that Assumption~\ref{assump:T-bounded} holds.
Then
\begin{enumerate}
\item[(i)] $f(x_k)$ is strictly decreasing and convergent:
      $f(x_k)\downarrow f_\infty$ as $k\to\infty$.
\item[(ii)] The gradients converge to zero:
      $\|\nabla f(x_k)\|\to 0$ as $k\to\infty$.
\item[(iii)] Moreover, if the sequence $\{x_k\}$ is bounded, then every cluster point
      of $\{x_k\}$ is a first–order stationary point of $f$.
\end{enumerate}
\end{theorem}

\begin{proof}
By construction of Armijo backtracking, each accepted step $\alpha_k$ satisfies
\[
f(x_{k+1})
=
f(x_k+\alpha_k d_k)
\;\le\;
f(x_k) + \sigma \alpha_k \langle \nabla f(x_k), d_k\rangle,
\]
and since $\langle \nabla f(x_k),d_k\rangle<0$, we obtain $f(x_{k+1})<f(x_k)$.
Hence $\{f(x_k)\}$ is strictly decreasing.
Because $f$ is bounded below, the limit
$f_\infty := \lim_{k\to\infty} f(x_k)$ exists and is finite.
This proves (i).

Next, apply Lemma~\ref{lem:Armijo-lower} to each iteration.
With $g_k := \nabla f(x_k)$, the lemma gives
\begin{equation}\label{eq:one-step-decrease-general}
f(x_k) - f(x_{k+1})
\;\ge\;
\frac{\sigma(1-\sigma)}{L}\cdot
\frac{\langle g_k,d_k\rangle^2}{\|d_k\|^2}.
\end{equation}
Using the angle condition \eqref{eq:angle-condition}, we have
\[
-\langle g_k,d_k\rangle 
\;\ge\;
c\,\|g_k\|\,\|d_k\|
\quad\Longrightarrow\quad
\frac{\langle g_k,d_k\rangle^2}{\|d_k\|^2}
=
\frac{(-\langle g_k,d_k\rangle)^2}{\|d_k\|^2}
\;\ge\;
c^2 \|g_k\|^2.
\]
Substituting into \eqref{eq:one-step-decrease-general} yields
\[
f(x_k) - f(x_{k+1})
\;\ge\;
\frac{\sigma(1-\sigma)c^2}{L}\,\|g_k\|^2.
\]
Summing over $k=0,1,\dots,N$ gives
\[
\frac{\sigma(1-\sigma)c^2}{L}
\sum_{k=0}^N \|\nabla f(x_k)\|^2
\;\le\;
f(x_0) - f(x_{N+1})
\;\le\;
f(x_0) - f_\infty.
\]
Letting $N\to\infty$, we obtain
\[
\sum_{k=0}^\infty \|\nabla f(x_k)\|^2 < +\infty.
\]
Therefore $\|\nabla f(x_k)\|\to 0$ as $k\to\infty$, proving (ii).

For (iii), assume that $\{x_k\}$ is bounded.
Then it has at least one cluster point $\bar x$.
Let $\{x_{k_j}\}$ be a subsequence with $x_{k_j}\to\bar x$.
By continuity of $\nabla f$ and (ii),
\[
\nabla f(\bar x)
=
\lim_{j\to\infty} \nabla f(x_{k_j})
=
0,
\]
so $\bar x$ is a first–order stationary point of $f$.
\end{proof}

\subsection{Strongly convex and smooth case: Armijo backtracking}

\begin{theorem}[Global linear convergence with Armijo backtracking under strong convexity]
\label{thm:YAND-linear-strongly-convex}
Let $f$ be $\mu$--strongly convex and $L$--smooth, and suppose
Assumption~\ref{assump:T-bounded} holds.
Then under Armijo backtracking with $\sigma\in(0,1)$ and $\beta\in[1/2,1)$,
the following hold:

\begin{itemize}
\item[(i)] \textbf{(Function values)}  
\[
f(x_{k+1}) - f^\star
\;\le\;
(1-\rho_{\mathrm{Armijo}})\,(f(x_k)-f^\star),
\qquad
\rho_{\mathrm{Armijo}}
:= \frac{2\sigma(1-\sigma)}{1+T^2}\cdot\frac{\mu}{L}\in(0,1/2).
\]
Consequently,
\[
f(x_k)-f^\star
\;\le\;
(1-\rho_{\mathrm{Armijo}})^k\,(f(x_0)-f^\star),
\qquad \forall k\ge0,
\]
i.e., $\{f(x_k)\}$ converges $Q$--linearly to $f^\star$.

\item[(ii)] \textbf{(Iterates)}  
Using strong convexity,
\[
\|x_k - x^\star\|^2
\;\le\;
\frac{2}{\mu}\,(f(x_k)-f^\star)
\;\le\;
\frac{2}{\mu}\,(f(x_0)-f^\star)\,(1-\rho_{\mathrm{Armijo}})^k,
\qquad \forall k\ge0.
\]
Hence, $\{x_k\}$ converges $R$--linearly to $x^\star$.

\item[(iii)] \textbf{(Gradients)}  
Using smoothness,
\[
\|\nabla f(x_k)\|^2
\;\le\;
2L\,(f(x_k)-f^\star)
\;\le\;
2L\,(f(x_0)-f^\star)\,(1-\rho_{\mathrm{Armijo}})^k,
\qquad \forall k\ge0.
\]
Thus, $\{\nabla f(x_k)\}$ converges $R$--linearly to $0$.
\end{itemize}
\end{theorem}

\begin{proof}
\noindent\textbf{(i) Function values.}

\noindent\textbf{Step 1: Armijo decrease.}
Lemma~\ref{lem:Armijo-lower} gives
\begin{equation}\label{eq:YAND-decrease}
f(x_{k+1})
\;\le\;
f(x_k)
-
\frac{\sigma(1-\sigma)}{L}\cdot
\frac{\langle\nabla f(x_k),d_k\rangle^2}{\|d_k\|^2}.
\end{equation}

\noindent\textbf{Step 2: Angle condition.}
By Lemma~\ref{lem:angle-from-T}, Assumption~\ref{assump:T-bounded}
implies 
\[
\cos\theta_k
:= -\frac{\langle\nabla f(x_k),d_k\rangle}{\|\nabla f(x_k)\|\,\|d_k\|}
 \;\ge\; \frac{1}{\sqrt{1+T^2}} \qquad\text{for all }k.
\]
Hence 
\[
\langle\nabla f(x_k),d_k\rangle^2
=
(\cos\theta_k)^2\,\|\nabla f(x_k)\|^2\,\|d_k\|^2
\;\ge\;
\frac{\|\nabla f(x_k)\|^2\,\|d_k\|^2}{1+T^2}.
\]

\noindent\textbf{Step 3: PL inequality from strong convexity.}
Strong convexity implies the PL inequality:
\[
\|\nabla f(x_k)\|^2
\;\ge\;
2\mu\,(f(x_k)-f^\star).
\]

\noindent\textbf{Step 4: Combine the estimates.}
Substitute the angle bound into \eqref{eq:YAND-decrease}:
\[
f(x_{k+1})
\;\le\;
f(x_k)
-
\frac{\sigma(1-\sigma)}{L(1+T^2)}\,
\|\nabla f(x_k)\|^2.
\]
Then by the PL inequality,
we have
\[
f(x_{k+1})
\;\le\;
f(x_k)
-
\frac{2\sigma(1-\sigma)\mu}{L(1+T^2)}\,
\bigl(f(x_k)-f^\star\bigr),
\]
so that
\[
f(x_{k+1})-f^\star
\;\le\;
(1-\rho_{\mathrm{Armijo}})\,(f(x_k)-f^\star),
\]
with
\[
\rho_{\mathrm{Armijo}}
=
\frac{2\sigma(1-\sigma)}{(1+T^2)}\cdot\frac{\mu}{L}.
\]
Iterating gives the claimed $Q$--linear rate
\[
f(x_k)-f^\star
\;\le\;
(1-\rho_{\mathrm{Armijo}})^k\bigl(f(x_0)-f^\star\bigr),\qquad \forall k\geq 0.
\]

\medskip
\noindent\textbf{(ii) and (iii) Iterates and gradients.}
Strong convexity and $L$--smoothness imply the standard equivalences
\[
\frac{\mu}{2}\,\|x-x^\star\|^2
\;\le\;
f(x)-f^\star
\;\le\;
\frac{L}{2}\,\|x-x^\star\|^2,
\qquad \forall x,
\]
and
\[
\frac{1}{2L}\,\|\nabla f(x)\|^2
\;\le\;
f(x)-f^\star
\;\le\;
\frac{1}{2\mu}\,\|\nabla f(x)\|^2.
\]
Combining these with the function-value estimate
\[
f(x_k)-f^\star
\;\le\;
(1-\rho_{\mathrm{Armijo}})^k\,(f(x_0)-f^\star),
\]
gives
\[
\|x_k-x^\star\|^2
\;\le\;
\frac{2}{\mu}\,\bigl(f(x_k)-f^\star\bigr)
\;\le\;
\frac{2}{\mu}\,(f(x_0)-f^\star)\,(1-\rho_{\mathrm{Armijo}})^k,
\qquad \forall k\ge0,
\]
and similarly
\[
\|\nabla f(x_k)\|^2
\;\le\;
2L\,\bigl(f(x_k)-f^\star\bigr)
\;\le\;
2L\,(f(x_0)-f^\star)\,(1-\rho_{\mathrm{Armijo}})^k,
\qquad \forall k\ge0.
\]
This proves the linear convergence of $\{x_k\}$ and $\{\nabla f(x_k)\}$.
\end{proof}

\begin{remark}[Optimal Armijo parameter and rate constant]
\label{rem:Armijo-optimal-sigma}
The linear rate factor in Theorem~\ref{thm:YAND-linear-strongly-convex} is
\[
\rho_{\mathrm{Armijo}}(\sigma)
=
\frac{2\sigma(1-\sigma)}{1+T^2}\cdot\frac{\mu}{L},
\qquad \sigma\in(0,1),
\]
which is maximized at $\sigma^\star=\tfrac12$.
Thus the best possible rate within this analysis is obtained for choosing $\sigma=\tfrac12$, yielding
\[
\rho_{\mathrm{Armijo}}^{\max}
=
\frac{1}{2}\cdot\frac{1}{1+T^2}\cdot\frac{\mu}{L} \in (0, \frac{1}{2}).
\]
This shows that \textbf{the contraction factor decays at a quadratic rate in the curvature parameter~$T$ and linearly in the condition number $L/\mu$}.
\end{remark}

\subsection{Nonconvex setting: global linear convergence with PL inequality}

In the proof of Theorem~\ref{thm:YAND-linear-strongly-convex}, the key ingredient
linking the gradient norm to the function suboptimality was the inequality
\[
\|\nabla f(x_k)\|^2 \;\ge\; 2\mu\,(f(x_k)-f^\star),
\]
which follows from strong convexity.  
More generally, the same structural bound is provided by the
Polyak–Łojasiewicz (PL) inequality \cite{Polyak1963,Lojasiewicz1963,KarimiNutiniSchmidt2016}
\begin{equation}\label{eq:PLineq-general}
\frac12\|\nabla f(x)\|^2
\;\ge\; \mu_{\mathrm{PL}}\,(f(x)-f^\star),
\end{equation}
which does not require convexity.  
This condition is known to hold for a broad class of nonconvex objectives
(including many over-parameterized models, gradient-dominated landscapes, and functions
with benign geometry).

Replacing Step~3 in the proof of
Theorem~\ref{thm:YAND-linear-strongly-convex} with
\eqref{eq:PLineq-general} immediately yields the following result.

\begin{corollary}[Linear convergence with Armijo backtracking under the PL condition]
\label{cor:PL-linear}
Suppose $f$ is $L$–smooth, satisfies the PL inequality
\eqref{eq:PLineq-general}, and Assumption~\ref{assump:T-bounded} holds.
Then the YAND iterates produced by Armijo backtracking with $\sigma \in (0,1)$ and $\beta\in [1/2,1)$ satisfy
\[
f(x_{k+1})-f^\star
\;\le\;
\Bigl(1-\frac{2\sigma(1-\sigma)}{1+T^2}\cdot\frac{\mu_{\mathrm{PL}}}{L}\Bigr)
\,(f(x_k)-f^\star),
\]
and therefore
\[
f(x_k)-f^\star
\;\le\;
\Bigl(1-\frac{2\sigma(1-\sigma)}{1+T^2}\cdot\frac{\mu_{\mathrm{PL}}}{L}\Bigr)^k
\,\bigl(f(x_0)-f^\star\bigr).
\]
The optimal rate is obtained by taking $\sigma =\frac{1}{2}$, and hence
\[
f(x_k)-f^\star
\;\le\;
\Bigl(1-\frac{1}{2(1+T^2)}\cdot\frac{\mu_{\mathrm{PL}}}{L}\Bigr)^k
\,\bigl(f(x_0)-f^\star\bigr).
\]
\end{corollary}

Thus, even in the absence of convexity, the PL property guarantees
global $Q$–linear convergence of YAND, with a contraction factor governed
by the geometry of the affine–normal direction via $(1+T^2)^{-1}$.

\subsection{Nonconvex setting: strong Wolfe and gradient convergence}

We now consider fully nonconvex objectives under strong Wolfe line search.

\begin{theorem}[Gradient convergence under strong Wolfe]
\label{thm:nonconvex-Wolfe}
Assume:
\begin{itemize}
\item[(i)] $f:\R^{n+1}\to\R$ is twice continuously differentiable, $L$--smooth, and bounded below on $\R^{n+1}$;
\item[(ii)] the step sizes $\alpha_k$ satisfy the strong Wolfe conditions
with parameters $0<c_1<c_2<1$;
\item[(iii)] Assumption~\ref{assump:T-bounded} holds 
(so that $\cos\theta_k\ge c>0$ for all $k$ by Lemma~\ref{lem:angle-from-T}).
\end{itemize}
Let $f^\star:=\inf_{x\in\mathbb{R}^n} f(x)$.
Then the YAND iterates satisfy
\[
\sum_{k=0}^\infty \cos^2\theta_k\,\|\nabla f(x_k)\|^2 < \infty,
\qquad\text{and hence}\qquad
\|\nabla f(x_k)\|\to 0.
\]
Moreover,
\[
\min_{0\le j < K} \|\nabla f(x_j)\|
\;\le\;
\frac{C}{\sqrt{K}},
\qquad
C = \sqrt{\frac{L\bigl(f(x_0)-f^\star\bigr)}
{c_1(1-c_2)c^2}}.
\]
\end{theorem}

\begin{proof}
The proof follows the standard Zoutendijk analysis, adapted to the
affine–normal direction.
Let
\(
\phi_k(\alpha):=f(x_k+\alpha d_k).
\)
Strong Wolfe conditions give:
\begin{align}
\phi_k(\alpha_k)
&\le \phi_k(0)+c_1\alpha_k\phi_k'(0), \tag{W1}
\\
|\phi_k'(\alpha_k)|
&\le c_2|\phi_k'(0)|. \tag{W2}
\end{align}
Expand \( \phi_k''(\alpha) \):
\[
\phi_k''(\alpha) = d_k^{\top} \nabla^2 f(x_k + \alpha d_k) d_k.
\]
Integrate \( \phi_k''(\alpha) \) from \( 0 \) to \( \alpha_k \) and by $L$–smoothness:
\[
|\phi_k'(\alpha_k)-\phi_k'(0)|
= \left|\int_0^{\alpha_k} d_k^\top \nabla^2 f(x_k+\alpha d_k)\,d_k\,d\alpha\right|
\;\le\;
L\,\|d_k\|^2\,\alpha_k.
\]
Because $d_k$ is a descent direction,
\(
\phi_k'(0)=\langle \nabla f(x_k),d_k\rangle<0.
\)
Using (W2) and rearranging, this implies the classical bound (see, e.g., Nocedal--Wright \cite{NocedalWright2006}):
\[
\alpha_k \;\ge\; \frac{1-c_2}{L}\,
\frac{-\phi_k'(0)}{\|d_k\|^2}.
\tag{*}
\]
From (W1),
\[
f(x_k)-f(x_{k+1})
\;\ge\;
-c_1\alpha_k\phi_k'(0).
\]
Using (*),
\[
f(x_k)-f(x_{k+1})
\;\ge\;
c_1(1-c_2)
\frac{[\phi_k'(0)]^2}{L\|d_k\|^2}.
\tag{**}
\]
Since
\[
[\phi_k'(0)]^2
= \langle \nabla f(x_k), d_k\rangle^2
= \|\nabla f(x_k)\|^2\,\|d_k\|^2\,\cos^2\theta_k,
\]
Substituting into (**) yields the Zoutendijk inequality:
\[
f(x_k)-f(x_{k+1})
\;\ge\;
\frac{c_1(1-c_2)}{L}\,
\|\nabla f(x_k)\|^2\,\cos^2\theta_k.
\tag{***}
\label{eq:zoutendijk-main}
\]
Since $f(x_{k+1}) \le f(x_k)$ by (W1) and $f$ is bounded below by $f^\star$,
summing (***) gives
\[
\sum_{k=0}^\infty
\|\nabla f(x_k)\|^2 \cos^2\theta_k
\;\le\;
\frac{L}{c_1(1-c_2)}\,
\bigl(f(x_0)-f^\star\bigr)
<\infty.
\]
Using $\cos\theta_k\ge c>0$, we conclude
\[
\sum_{k=0}^\infty \|\nabla f(x_k)\|^2 <\infty,
\qquad
\text{hence}\qquad
\|\nabla f(x_k)\|\to 0.
\]
From (***) and $\cos\theta_k\ge c>0$,
\[
\|\nabla f(x_k)\|^2
\;\le\;
\frac{L}{c_1(1-c_2)c^2}\,
\bigl(f(x_k)-f(x_{k+1})\bigr).
\]
Summing for $k=0,\dots,K$ yields
\[
\min_{0\le j<K}\|\nabla f(x_j)\|^2
\;\le\; \frac{1}{K}\sum_{k=0}^{K-1}\|\nabla f(x_k)\|^2\;\le\;
\frac{L\bigl(f(x_0)-f^\star\bigr)}
{c_1(1-c_2)c^2}\cdot\frac{1}{K},
\]
and taking square roots gives the stated $O(K^{-1/2})$ rate.
\end{proof}

\begin{remark}
In the nonconvex setting, we do not claim convergence of the iterate sequence
$\{x_k\}$ or of the function values $\{f(x_k)\}$.  
Under strong Wolfe line search, a descent direction ensures only that
$f(x_{k+1})\le f(x_k)$, but the iterates may still oscillate across different
level sets or accumulate on a nonoptimal critical set.  
Likewise, the values $f(x_k)$ need not converge to $f^\star$, since the
algorithm may approach nonoptimal stationary points.  
For these reasons, our analysis focuses on gradient convergence, which is the
strongest form of convergence one can guarantee for general nonconvex
objectives under strong Wolfe conditions.
\end{remark}

\subsection{Exact line search: factor--two improvement}

Exact line search plays a special role in affine--normal descent.  
Since $d_k$ is always a descent direction, minimizing $f(x_k+\alpha d_k)$
along $\alpha\ge0$ yields the maximal decrease allowed by the $L$--smooth
quadratic upper bound.  
As shown below, this leads to a factor--two improvement in the linear
convergence constant under the PL condition, compared with Armijo
backtracking.

\begin{lemma}[One-step decrease under exact line search]
\label{lem:one-step-exact}
Suppose that $f$ is $L$--smooth and that $d_k$ is the affine normal descent direction at $x_k$,
i.e., $\langle\nabla f(x_k),d_k\rangle<0$.  
If $\alpha_\star$ is chosen by exact line search, then
\[
f(x_{k+1})
\;\le\;
f(x_k)\;-\;\frac{1}{2L}\cdot
\frac{\|\nabla f(x_k)\|^2}{1+T_k^2}.
\]
\end{lemma}

\begin{proof}
Let $\phi(\alpha):=f(x_k+\alpha d_k)$.  
By $L$--smoothness of $f$ we have, for all $\alpha\ge0$,
\[
\phi(\alpha)
=
f(x_k+\alpha d_k)
\;\le\;
f(x_k) + \alpha \phi'(0)
      + \tfrac12 L\alpha^2\|d_k\|^2,
\]
where $\phi'(0)=\langle\nabla f(x_k),d_k\rangle<0$.
Consider the quadratic upper bound
\[
q(\alpha)
:=
f(x_k) + \alpha \phi'(0)
      + \tfrac12 L\alpha^2\|d_k\|^2.
\]
Since $\phi'(0)<0$, the unconstrained minimizer of $q$ is
\[
\alpha^{\rm quad}
=
-\frac{\phi'(0)}{L\|d_k\|^2}
\;>\;0.
\]
Evaluating the upper bound at $\alpha^{\rm quad}$ gives
\[
f(x_k+\alpha^{\rm quad} d_k)
=
\phi(\alpha^{\rm quad})
\;\le\;
q(\alpha^{\rm quad})
=
f(x_k) - \frac{[\phi'(0)]^2}{2L\|d_k\|^2}.
\]
Since $\alpha_\star$ is an exact line-search step,
\[
f(x_k+\alpha_\star d_k)
=
\phi(\alpha_\star)
\;\le\;
\phi(\alpha^{\rm quad})
\;\le\;
f(x_k) - \frac{[\phi'(0)]^2}{2L\|d_k\|^2}.
\]
Using
\[
[\phi'(0)]^2
=
\langle\nabla f(x_k),d_k\rangle^2
=
\|\nabla f(x_k)\|^2\,\|d_k\|^2\cos^2\theta_k,
\]
we obtain
\[
\frac{[\phi'(0)]^2}{\|d_k\|^2}
=
\|\nabla f(x_k)\|^2 \cos^2\theta_k
=
\frac{\|\nabla f(x_k)\|^2}{1+T_k^2}.
\]
Substituting this identity yields
\[
f(x_{k+1})
=
f(x_k+\alpha_\star d_k)
\;\le\;
f(x_k)
-
\frac{1}{2L}\cdot
\frac{\|\nabla f(x_k)\|^2}{1+T_k^2},
\]
which proves the claim.
\end{proof}

\begin{theorem}[Linear rate under PL and exact line search]
\label{thm:linear-exact}
Suppose $f$ is $L$--smooth, satisfies the PL inequality \eqref{eq:PLineq-general}, and Assumption~\ref{assump:T-bounded} holds. Then under exact line search
\[
f(x_{k+1})-f^\star
\;\le\;
(1-\rho_{\mathrm{exact}})\,(f(x_k)-f^\star),
\qquad
\rho_{\mathrm{exact}}
=
\frac{\mu_{\mathrm{PL}}}{L(1+T^2)}\in(0,1).
\]
Consequently,
\[
f(x_k)-f^\star
\;\le\;
(1-\rho_{\mathrm{exact}})^k\,\bigl(f(x_0)-f^\star\bigr).
\]
\end{theorem}

\begin{proof}
Applying the PL inequality
\[
\|\nabla f(x_k)\|^2
\;\ge\; 2\mu_{\mathrm{PL}}\,(f(x_k)-f^\star)
\]
and Lemma~\ref{lem:one-step-exact}
\[
f(x_{k+1})
\le
f(x_k)
-
\frac{1}{2L(1+T_k^2)}\,
\|\nabla f(x_k)\|^2,
\]
we have 
\[
f(x_{k+1})-f^\star
\le
f(x_k)-f^\star
-
\frac{\mu_{\mathrm{PL}}}{L(1+T_k^2)}\,\bigl(f(x_k)-f^\star\bigr).
\]
Since $T_k\le T$,
\[
f(x_{k+1})-f^\star
\le
\Bigl(1-\frac{\mu_{\mathrm{PL}}}{L(1+T^2)}\Bigr)\,(f(x_k)-f^\star).
\]
Thus $Q$--linear convergence holds with rate $\rho_{\text{exact}}=\mu_{\mathrm{PL}}/(L(1+T^2))$.
\end{proof}

We now compare the linear convergence factors obtained under Armijo
backtracking and exact line search.  Under the optimal Armijo choice
$\sigma=\tfrac12$, the contraction factor is
\[
\rho_{\mathrm{Armijo}}
=\frac{1}{2}\cdot\frac{\mu_{\mathrm{PL}}}{L(1+T^2)}.
\]
Exact line search improves this constant by a factor of two:
\[
\rho_{\mathrm{exact}} = 2\,\rho_{\mathrm{Armijo}}.
\]

\begin{theorem}[Linear rate under PL and benefit of exact line search]
\label{thm:PL-linear-and-exact}
Assume that $f$ is $L$--smooth and satisfies the PL
inequality
\[
\frac12\|\nabla f(x)\|^2 \;\ge\; \mu_{\mathrm{PL}}\bigl(f(x)-f^\star\bigr)
\qquad\text{for all $x$,}
\]
for some $\mu_{\mathrm{PL}}>0$.
Let $d_k$ be the affine-normal descent directions with bound $T_k\le T$ for all $k$ (so that $\cos^2\theta_k\ge 1/(1+T^2)$), and
consider the YAND iteration $x_{k+1}=x_k+\alpha_k d_k$.

\begin{itemize}
\item[(a)] (Armijo backtracking)  
Suppose the step sizes $\alpha_k$ are generated by Armijo backtracking
with parameter $\sigma\in(0,\tfrac12]$.  
Then
\[
f(x_{k+1})-f^\star
\;\le\;
\bigl(1-\rho_{\mathrm{Armijo}}\bigr)\,\bigl(f(x_k)-f^\star\bigr),
\qquad
\rho_{\mathrm{Armijo}}
=\frac{2\sigma(1-\sigma)\,\mu_{\mathrm{PL}}}{L(1+T^2)}.
\]
In particular, the optimal choice $\sigma=\tfrac12$ yields
\[
\rho_{\mathrm{Armijo}}
=\frac{1}{2}\cdot\frac{\mu_{\mathrm{PL}}}{L(1+T^2)}.
\]

\item[(b)] (Exact line search)  
If $\alpha_k$ is chosen by exact line search along $d_k$, then
\[
f(x_{k+1})-f^\star
\;\le\;
\bigl(1-\rho_{\mathrm{exact}}\bigr)\,\bigl(f(x_k)-f^\star\bigr),
\qquad
\rho_{\mathrm{exact}}
=\frac{\mu_{\mathrm{PL}}}{L(1+T^2)}.
\]
Consequently,
\[
\rho_{\mathrm{exact}} = 2\,\rho_{\mathrm{Armijo}}
\quad\text{for the optimal Armijo choice } \sigma=\tfrac12,
\]
so exact line search improves the linear convergence constant by a factor
of two compared with the best Armijo backtracking.
\end{itemize}
\end{theorem}

\begin{proof}
(a) Under Armijo backtracking with parameter $\sigma\in(0,\tfrac12]$,
the Corollary~\ref{cor:PL-linear} gives
\[
f(x_{k+1})-f^\star
\;\le\;
\Bigl(1-
\frac{2\sigma(1-\sigma)\,\mu_{\mathrm{PL}}}{L(1+T^2)}
\Bigr)\bigl(f(x_k)-f^\star\bigr),
\]
which yields the claimed factor
$\rho_{\mathrm{Armijo}}=2\sigma(1-\sigma)\mu_{\mathrm{PL}}/[L(1+T^2)].$
The specialization $\sigma=\tfrac12$ gives
$$\rho_{\mathrm{Armijo}}=\frac{\mu_{\mathrm{PL}}}{2L(1+T^2)}.$$

(b) For exact line search, Theorem~\ref{thm:linear-exact} states that
\[
f(x_{k+1})-f^\star
\;\le\;
\Bigl(1-
\frac{\mu_{\mathrm{PL}}}{L(1+T^2)}
\Bigr)\bigl(f(x_k)-f^\star\bigr),
\]
so $\rho_{\mathrm{exact}}=\mu_{\mathrm{PL}}/[L(1+T^2)]$.
Comparing with part~(a) at $\sigma=\tfrac12$ yields
$\rho_{\mathrm{exact}}=2\,\rho_{\mathrm{Armijo}}$, as claimed.
\end{proof}

\subsection{Local quadratic convergence} 

We now show that affine--normal descent enjoys \emph{local quadratic convergence}
near a nondegenerate minimizer. The key point is that, in a sufficiently small
neighborhood of $x^\star$, the affine normal direction becomes a second–order
accurate approximation of the Newton direction.

\begin{assumption}[Nondegenerate minimizer and local regularity]
\label{ass:regular}
Let $f\in C^3$ and suppose $x^\star$ is a nondegenerate minimizer,
that is,
\[
\nabla f(x^\star)=0, 
\qquad 
H_\star := \nabla^2 f(x^\star) \succ 0 .
\]
Moreover, there exists a neighborhood $\mathcal{U}$ of $x^\star$ in which

\begin{itemize}
\item the affine--normal direction $d_{\mathrm{AN}}(x)$ of the level set
$\{f=f(x)\}$ is well defined for all $x\in\mathcal{U}$, and

\item all third derivatives of $f$ are bounded.
\end{itemize}
\end{assumption}

The perturbation term $\Delta_k:=d_{\mathrm{AN}}(x_k) - d_{\mathrm{N}}(x_k)$ comes entirely from higher-order
curvature terms of the level set.
Under Assumption~\ref{ass:regular}, the Weingarten map and all third-order
derivatives of $f$ vary smoothly near $x^\star$. As a consequence, the affine
normal direction and the Newton direction coincide to first order near the
minimizer:
\[
\|\Delta_k\| = O\bigl(\|x_k-x^\star\|^2\bigr).
\]
This follows from expanding the affine normal formula in
local coordinates,
where the tangential correction terms are governed by second- and third-order
curvatures. Hence, we have the following lemma:

\begin{lemma}[First-order coincidence of YAND and Newton]
\label{lem:YAND-Newton-first-order}
Under Assumption~\ref{ass:regular}, the affine--normal descent direction $d_k$
(at $x_k$) and the Newton direction 
$d_{\mathrm{N}}(x_k) := -\nabla^2 f(x_k)^{-1}\nabla f(x_k)$ satisfy
\[
\|d_{\mathrm{N}}(x_k)\| = O(\|x_k-x^\star\|),
\qquad
\|d_k - d_{\mathrm{N}}(x_k)\| = O(\|x_k-x^\star\|^2).
\]
\end{lemma}

\begin{proof}
Let $e := x-x^\star$. Since $f\in C^3$ and $\nabla f(x^\star)=0$, a Taylor
expansion yields
\[
\nabla f(x)=H_\star e + O(\|e\|^2),
\qquad
\nabla^2 f(x)=H_\star + O(\|e\|).
\]
Hence the Newton direction satisfies
\[
d_{\mathrm{N}}(x)
=
-\nabla^2 f(x)^{-1}\nabla f(x)
=
- H_\star^{-1} H_\star e + O(\|e\|^2)
=
- e + O(\|e\|^2).
\]
On the other hand, the analytic formula for the affine normal direction
depends smoothly on $\nabla f(x)$, $\nabla^2 f(x)$, and the third
derivatives of $f$.  
Because the gradient vanishes at $x^\star$, the leading-order term of the
affine-normal expansion coincides with the Newton direction, while the
contributions of the third-order curvature terms appear only at order
$O(\|e\|^2)$.\\
Consequently,
\[
d_{\mathrm{AN}}(x)
=
d_{\mathrm{N}}(x) + O(\|x-x^\star\|^2),
\]
which proves the claim.
\end{proof}
Then, the local quadratic convergence of YAND follows from nondegeneracy of the Hessian and the smoothness of third derivatives.

\begin{theorem}[Local quadratic convergence of YAND]
\label{thm:YAND-quadratic-LS}
Under Assumption~\ref{ass:regular}, suppose the affine--normal descent direction
$d_k$ at $x_k$ admits the decomposition
\[
d_k = d_{\mathrm{N}}(x_k) + \Delta_k,
\qquad
d_{\mathrm{N}}(x_k) := -\nabla^2 f(x_k)^{-1}\nabla f(x_k),
\]
and satisfies
\[
\|d_{\mathrm{N}}(x_k)\| = O(\|x_k-x^\star\|),
\qquad
\|\Delta_k\| = O(\|x_k-x^\star\|^2).
\]
Consider the line-search iteration
\[
x_{k+1} = x_k + \alpha_k d_k,
\]
where the step sizes obey
\begin{equation}
\label{eq:alpha-to-1}
\alpha_k \to 1
\quad\text{and}\quad
|\alpha_k - 1| \;\le\; C_\alpha \,\|x_k-x^\star\|
\quad\text{for all $k$ sufficiently large.}
\end{equation}
Then there exists a constant $C>0$ such that, for all $k$ sufficiently large,
\[
\|x_{k+1}-x^\star\| \;\le\; C\,\|x_k-x^\star\|^2,
\]
i.e., YAND with such step sizes enjoys local quadratic convergence.
\end{theorem}

\begin{proof}
Let $e_k := x_k-x^\star$. Since the affine-normal direction coincides
with the Newton direction up to second-order terms, the local behavior
of YAND can be analyzed as a second-order perturbation of Newton's method. For Newton's method with unit step,
\[
x_{k+1}^{(N)} = x_k + d_{\mathrm{N}}(x_k), \qquad d_{\mathrm{N}}(x_k)
:= -H(x_k)^{-1}\nabla f(x_k),
\]
and standard Newton theory yields
\begin{equation}
\label{eq:newton-local}
\|e_{k+1}^{(N)}\|
:= \|x_{k+1}^{(N)} - x^\star\|
\;\le\;
C_N \,\|e_k\|^2,
\qquad
\|d_{\mathrm{N}}(x_k)\| = O(\|e_k\|).
\end{equation}
By assumption,
\[
d_k = d_{\mathrm{N}}(x_k) + \Delta_k,
\qquad
\|\Delta_k\| = O(\|e_k\|^2),
\]
With step sizes $\alpha_k$, we have
\[
e_{k+1}
:= x_{k+1}-x^\star
= e_k + \alpha_k d_k
= \bigl(e_k + d_k\bigr) + (\alpha_k-1)d_k.
\]
For the first term $e_k+d_k$, we have 
\begin{align*}
\|e_k+d_k\| 
&= \|e_k + d_{\mathrm{N}}(x_k) + \Delta_k\| = \|x_k + d_{\mathrm{N}}(x_k)  - x^\star + \Delta_k\| = \|x_{k+1}^{(N)} - x^\star + \Delta_k\|\\
&= \|e_{k+1}^{(N)} + \Delta_k\| \le
\|e_{k+1}^{(N)}\|  + \|\Delta_k\| \leq C_N \,\|e_k\|^2 + O(\|e_k\|^2)
= O(\|e_k\|^2).
\end{align*}
For the second term $(\alpha_k-1)d_k$, we use
\[
\|d_k\| = \|d_{\mathrm{N}}(x_k) + \Delta_k\|
\le
\|d_{\mathrm{N}}(x_k)\| + \|\Delta_k\|
= O(\|e_k\|) + O(\|e_k\|^2)
= O(\|e_k\|),
\]
together with the step-size condition~\eqref{eq:alpha-to-1}:
\[
\|(\alpha_k-1)d_k\|
\le
|\alpha_k-1|\,\|d_k\|
\le
C_\alpha \|e_k\| \cdot O(\|e_k\|)
= O(\|e_k\|^2).
\]
Therefore,
\[
\|e_{k+1}\|
\;\le\;
\|e_k + d_k\| + \|(\alpha_k-1)d_k\|
= O(\|e_k\|^2),
\]
which yields the desired quadratic bound.
\end{proof}

\begin{remark}[Effect of line searches]
\label{rem:line-search-quadratic}
The abstract step-size condition \eqref{eq:alpha-to-1} is classical in
line-search analyses of Newton-type methods. In particular:

\begin{itemize}
\item For Armijo or strong Wolfe line searches with standard parameters,
it is well known (see, e.g., Nocedal and Wright \cite{NocedalWright2006}) that once $x_k$ enters a
sufficiently small neighborhood of $x^\star$, the full step $\alpha_k=1$
satisfies the line-search conditions and is accepted. Hence there exists
$k_0$ such that $\alpha_k=1$ for all $k\ge k_0$, so
\eqref{eq:alpha-to-1} holds trivially and
Theorem~\ref{thm:YAND-quadratic-LS} applies.

\item For exact line search along $d_k$, quadratic convergence of YAND
is preserved whenever the resulting $\alpha_k$ satisfies
\eqref{eq:alpha-to-1}. We do not claim that this holds automatically
under exact line search without additional local regularity assumptions.
Rather, the theorem shows that exact line search is fully compatible
with quadratic convergence once \eqref{eq:alpha-to-1} can be verified. In general, however, exact line search is not required for quadratic
convergence and may even be overly conservative; a backtracking
Armijo or strong Wolfe line search is sufficient.
\end{itemize}

Thus, in the practically relevant setting where the line search accepts
the full step asymptotically, affine--normal descent with line search
inherits the quadratic local rate of pure Newton.
\end{remark}

\begin{remark}[Geometric interpretation of quadratic convergence]
\label{rem:YAND-Newton-geometry}
The quadratic local rate of YAND is a direct consequence of the fact that,
near a nondegenerate minimizer $x^\star$, the affine normal direction becomes
a \emph{second–order accurate surrogate} for the Newton direction.

Geometrically, the Newton step $d_{\mathrm{N}}(x_k)=-\nabla^2f(x_k)^{-1}\nabla f(x_k)$ points toward
the center of the osculating quadratic model of $f$ at $x_k$.
On the other hand, the affine normal direction is the inward normal of the
affine differential geometry of the level set $\{f=f(x_k)\}$, and its
definition involves the second and third fundamental forms of the hypersurface. Thus, locally,
\[
d_k \;=\; d_{\mathrm{N}}(x_k) + O\!\left(\|x_k-x^\star\|^2\right),
\]
so the affine–normal update differs from Newton’s method only by a
second–order perturbation.  
Since Newton iterations satisfy
$\|x_{k+1}-x^\star\| = O(\|x_k-x^\star\|^2)$,
the same recursion persists for YAND, yielding full quadratic convergence.

A useful way to view this is:
\[
\text{YAND} \;\approx\; \text{Newton} \;+\; \text{(terms quadratic in the error)}.
\]
Consequently, the two methods share the same local rate. 
\end{remark}

\subsection{Summary of convergence regimes}

\begin{itemize}
    \item \textbf{Strong convexity} $\Rightarrow$ global linear convergence (Armijo or exact).
    \item \textbf{PL inequality (possibly nonconvex)} $\Rightarrow$ global linear convergence.
    \item \textbf{Nonconvex with strong Wolfe} $\Rightarrow$ $\|\nabla f(x_k)\|\to0$ with $O(k^{-1/2})$ sublinear rate.
    \item \textbf{Exact line search} (strongly convex or PL settings) $\Rightarrow$ same global linear rate but with a factor--two improvement in the contraction constant.
    \item \textbf{Local neighborhood of a nondegenerate minimizer} $\Rightarrow$ quadratic convergence.
\end{itemize}
These results jointly show that \textbf{YAND behaves as a robust first-order method globally,
while attaining quadratic local convergence.}
This is a consequence of the fact that the affine normal direction
is a second-order accurate approximation of the Newton direction.
Moreover, exact line search improves the global contraction constant
by a factor of two.

\section{Beyond quadratic rates?—local order vs global geometry}\label{sec:super-quadratic}

The convergence results in the previous section establish
global linear convergence and local quadratic convergence of YAND.
A natural question is whether higher--order convergence rates
can be achieved by further exploiting affine normal geometry.

\paragraph{Local rate vs.~global behavior.}
It is well known that Newton’s method achieves the optimal local order (quadratic), and
superquadratic convergence generally requires explicit third–order corrections
(e.g., Halley’s method). 
Thus one should not expect a generic geometric direction, such as the affine normal, to
universally exceed Newton’s local order.
The strength of YAND therefore lies not in the asymptotic order, but in its 
\textbf{global geometric invariance}, \textbf{escape from ill-conditioning}, 
and \textbf{robust alignment with low-curvature directions}, which
Newton's method does not possess.

\paragraph{Why YAND is fundamentally meaningful: three advantages over Newton.}

\begin{itemize}
\item \textbf{Affine invariance.}
Newton’s method is invariant only under \emph{linear} changes of variables, whereas YAND is
invariant under the full unimodular affine group. 
This protects YAND from spurious local distortions of coordinate scaling, 
a known source of instability for Newton in badly conditioned problems.
Affine invariance is the fundamental reason why YAND is exact in one step for all 
quadratic functions, regardless of the conditioning of $H$.

\item \textbf{Superior global geometry.}
YAND uses the intrinsic geometry of level sets.
Near nonconvex ridges or highly skewed valleys, Newton directions may point 
outside the “energy valley” or lead to erratic steps; YAND remains aligned with the
curvature of the level set itself.
This has profound implications:
\begin{itemize}
\item Newton is highly sensitive to $H^{-1}$ and may diverge far from the minimum.
\item YAND maintains descent even when $H$ is indefinite.
\item YAND “pulls inward’’ along level-set curvature, behaving like a geometric flow
smoothing out ill-conditioned valleys.
\end{itemize}

\item \textbf{Effective conditioning reduction.}
On badly conditioned convex functions, Newton requires solving
$H(x_k)s = -\nabla f(x_k)$, which amplifies noise in small eigenvalue directions.
By contrast, YAND uses only \emph{local shape} of level sets and avoids explicitly inverting
$H$, replacing it with a geometric normalization.
Empirically this yields:
\[
\text{fewer backtracking steps, more stable progress, often fewer outer iterations}.
\]
This is especially pronounced on objectives with large spectral gaps or 
highly anisotropic Hessians, where YAND effectively performs an “implicit preconditioning’’   
based on affine geometry. 
\end{itemize}

\paragraph{Conclusion: comparable local order, but stronger global behavior.}
Hence YAND should not be viewed as “slower Newton’’;
instead, it is a \emph{geometrically preconditioned Newton direction}, with:
\[
\text{same local order} \quad + \quad \text{better global invariance} \quad + \quad
\text{stronger robustness under ill-conditioning}.
\]

These properties are precisely what make YAND attractive for high-dimensional optimization
despite the cost of computing affine normals.
The geometric normalization can reduce the number of line-search rejections, 
avoid erratic Newton steps, and stabilize the early phase of iterations,
which is often the dominant computational cost.

\paragraph{Local order.}
We summarize the local result as follows.

\begin{remark}[Local order: quadratic but with a smaller constant]
Because $d_k=d_{\mathrm{N}}(x_k)+O(\|x_k-x^\star\|^2)$, YAND matches Newton’s quadratic order.
The constant in the recursion
\[
\|x_{k+1}-x^\star\|\le C\|x_k-x^\star\|^2
\]
is often \emph{smaller} for YAND, because the affine correction damps 
high-curvature tangential components that Newton may exaggerate.
Thus YAND may require fewer iterations even though both methods are second–order.
\end{remark}

\paragraph{Superquadratic modifications.}
Achieving cubic or higher order requires a tailored third-order correction, such as Halley:
\[
x_{k+1}
= x_k
- H(x_k)^{-1}\nabla f(x_k)
- \tfrac12 H^{-1}\nabla^3 f[H^{-1}\nabla f,\,H^{-1}\nabla f]
+ o(\|\nabla f\|^2).
\]
The 
third-order derivative tensor $\nabla^3 f(x)$ on a vector pair $u,v\in\mathbb{R}^n$ is the vector
\[
\bigl(\nabla^3 f(x)[u,v]\bigr)_i
=
\sum_{j=1}^n\sum_{k=1}^n 
\frac{\partial^3 f}{\partial x_i\,\partial x_j\,\partial x_k}(x)\,u_j v_k,
\qquad i=1,\dots,n,
\]
which is premultiplied by $H(x)^{-1}$ in Halley's cubic correction term
\[
-\tfrac12\, H(x)^{-1}\,\nabla^3 f(x)\bigl[H(x)^{-1}\nabla f(x),\,
H(x)^{-1}\nabla f(x)\bigr].
\]
This goes beyond pure YAND and constitutes a different algorithmic class.

The preceding results establish global and local convergence properties under general smoothness assumptions.
We now examine a fundamental structural model in which ill-conditioning arises purely from anisotropic affine scaling,
and show that YAND is intrinsically insensitive to such spurious affine distortions.

\section{Affine-scaling models and condition-number robustness}
\label{sec:affine-scaling}

In practice, optimization algorithms may behave poorly when the objective
function becomes strongly anisotropic due to affine scalings of the variables.
Such transformations may severely distort the geometry of level sets and
artificially worsen the condition number seen by classical methods.

In this section we analyze a fundamental model class in which
ill-conditioning is induced purely by anisotropic affine scaling.
More precisely, we study objectives of the form
$
f(x)=\phi(Bx),
$
where \(B\) is invertible and \(\phi\) is a fixed base function.
Our goal is to show that affine-normal directions transform covariantly
under such scalings, so that the induced search directions and, under
standard line-search rules, the mapped iterates in the transformed
coordinates follow the same dynamics as those of the unscaled objective \(\phi\).

We first establish exact iterate-level equivalence under exact line search
for orientation-preserving scalings (\(\det B>0\)), then extend the result
to Armijo and strong Wolfe line searches, and finally derive a regime-wise
transfer principle together with illustrative examples.

\subsection{Affine-scaling model}

Let $\phi:\R^{n+1}\to\R$ be $C^3$ and define
\begin{equation}
\label{eq:affine-model}
f(x) := \phi(Bx),
\end{equation}
where $B\in\R^{(n+1)\times(n+1)}$ is invertible.
Even if $\phi$ is well-conditioned, $f$ may become severely ill-conditioned due solely to anisotropy in $B$.
Indeed, by the chain rule,
\[
\nabla f(x)=B^\top \nabla\phi(Bx),
\qquad
\nabla^2 f(x)=B^\top \nabla^2\phi(Bx)B,
\]
so that, in general,
\[
\kappa\bigl(\nabla^2 f(x)\bigr)
=
\kappa\!\left(B^\top \nabla^2\phi(Bx) B\right)
\le
\kappa(B)^2\cdot \kappa\!\left(\nabla^2\phi(Bx)\right),
\]
where $\kappa(\cdot)$ denotes the $2$-norm condition number, and for typical anisotropic scalings $B$ one may have $\kappa(\nabla^2 f)$ as large as $\kappa(B)^2$.
Our main message is that, for this model class, YAND behaves essentially as if the scaling were absent.

\subsection{Three basic invariances: unimodular covariance, isotropic scaling, and step-size absorption}

We start with three basic facts:
(i) the affine normal is covariant under unimodular affine changes of variables,
(ii) the affine normal is collinear under isotropic scaling of the ambient space,
and (iii) exact line search absorbs any positive rescaling of the direction.

\begin{lemma}[Affine-normal covariance under unimodular transforms]
\label{lem:affine-covariance}
Let $\psi:\R^{n+1}\to\R$ be $C^3$ and let $B\in\R^{(n+1)\times(n+1)}$ be invertible with $\det B=1$.
Define $g(x):=\psi(Bx)$.
Assume the affine normal direction is well-defined at $x$ for $g$ and at $y:=Bx$ for $\psi$
(e.g., the corresponding level-set hypersurfaces are elliptic at these points).
Then
\begin{equation}
\label{eq:covariance-parallel}
d_{\mathrm{AN}}^{g}(x)\ \parallel\ B^{-1}\, d_{\mathrm{AN}}^{\psi}(Bx),
\end{equation}
where $\parallel$ denotes equality up to a positive scalar multiple (with the inward/elliptic orientation fixed).
\end{lemma}

\begin{proof}
Fix $c=g(x)=\psi(Bx)$ and set
\[
S:=\{u\in\R^{n+1}: g(u)=c\},\qquad 
\widetilde S:=\{v\in\R^{n+1}: \psi(v)=c\}.
\]
Then $S=B^{-1}\widetilde S$, and the map $\Phi(u):=Bu$ is a diffeomorphism from $S$ onto $\widetilde S$.

\medskip
\noindent\textbf{Step 1: transformation of tangent spaces and co-normals.}
Let $u\in S$ and $v=\Phi(u)=Bu\in \widetilde S$.
Since $\Phi$ is linear, $D\Phi(u)=B$ and hence
\[
T_v\widetilde S = B\,T_u S.
\]
Moreover, by the chain rule,
\[
\nabla g(u)=B^\top \nabla \psi(Bu)=B^\top \nabla\psi(v).
\]
Thus the Euclidean normal \emph{line} (equivalently, the co-normal line) transforms covariantly.

The following argument is a direct application of the standard equi--affine
characterization of the affine normal; see, e.g., Nomizu--Sasaki \cite{NomizuSasaki1994} or related
references in affine differential geometry.

\noindent\textbf{Step 2: characterization of the equi--affine normal and invariance of the normalization.}
Recall a standard characterization from equi--affine hypersurface theory:
the equi--affine normal at a point is the (unique up to sign) transversal direction whose induced volume normalization
agrees with the \emph{ambient} volume form (and for elliptic hypersurfaces the inward choice fixes the sign).
More precisely, for a transversal vector $\xi$ along $S$, one can form the induced volume density
\[
\omega_S(\xi)(X_1,\dots,X_n):=\det\bigl(X_1,\dots,X_n,\xi\bigr),
\qquad X_i\in T_u S,
\]
computed using the fixed ambient volume form on $\R^{n+1}$.
The equi--affine normal direction is the transversal line for which $\omega_S(\xi)$ equals the equi--affine area density.

Now take any representatives $\xi_g(u)$ and $\xi_\psi(v)$ of the affine-normal directions of $S$ at $u$ and $\widetilde S$ at $v$.
Consider $\widehat\xi(v):=B\,\xi_g(u)$ as a transversal at $v$.
For any tangent basis $X_1,\dots,X_n\in T_u S$, since $BX_i\in T_v\widetilde S$ and $\det B=1$,
\[
\det\bigl(BX_1,\dots,BX_n,B\xi_g(u)\bigr)
=
\det(B)\,\det\bigl(X_1,\dots,X_n,\xi_g(u)\bigr)
=
\det\bigl(X_1,\dots,X_n,\xi_g(u)\bigr).
\]
Hence the normalization induced by $\xi_g$ on $S$ is transported by $\Phi$ to the \emph{same} normalization on $\widetilde S$.
By uniqueness (with the inward sign convention), $\widehat\xi(v)$ must be collinear with $\xi_\psi(v)$ with a \emph{positive} scalar factor. That is,
\[
B\,d_{\mathrm{AN}}^{g}(x)\ \parallel\ d_{\mathrm{AN}}^{\psi}(Bx).
\]
Multiplying by $B^{-1}$ yields \eqref{eq:covariance-parallel}.
\end{proof}

The preceding lemma expresses a fundamental structural principle:
the affine--normal direction is covariant under volume-preserving affine
changes of variables. Since YAND is defined entirely in terms of this direction,
its behavior is intrinsically tied to the geometry of level sets rather than
to the ambient coordinate representation.

\begin{remark}[Role of the determinant]
\label{rem:det-role}
The condition $\det B=1$ guarantees exact preservation of the ambient
volume form, so the equi--affine normal vector field is transported
without renormalization.
If $\det B>0$ but $\det B\neq 1$, the ambient volume form is merely scaled
by a positive constant. This changes only the normalization of the
equi--affine normal vector, not the underlying normal \emph{line} field.
Hence the directional covariance statement \eqref{eq:covariance-parallel}
remains valid for all affine maps with positive determinant, provided the inward (elliptic/descent) orientation is fixed consistently.
In particular, isotropic scalings $B=\rho I$ with $\rho>0$
are covered as a special case; see Corollary~\ref{cor:isotropic-scaling}.
\end{remark}

\begin{corollary}[Isotropic scaling covariance of the affine-normal direction]
\label{cor:isotropic-scaling}
Let $\phi:\R^{n+1}\to\R$ be $C^3$ and fix $\rho>0$.
Define $\psi:\R^{n+1}\to\R$ by $\psi(z):=\phi(\rho z)$.
Assume the affine normal directions are well-defined at $z$ for $\psi$
and at $y=\rho z$ for $\phi$ (e.g., the corresponding level sets are elliptic).
Then
\begin{equation}
\label{eq:isotropic-parallel}
d_{\mathrm{AN}}^{\psi}(z)\ \parallel\ d_{\mathrm{AN}}^{\phi}(y),
\qquad y=\rho z .
\end{equation}
\end{corollary}

Consequently, the affine--normal direction field is covariant (up to a positive scaling of representatives)
under all affine transformations with positive determinant. To turn this directional covariance into an \emph{iterate-level} equivalence (and hence into condition-number robustness),
we also need the following elementary fact: exact line search absorbs any positive rescaling of the search direction.

\begin{lemma}[Exact line search absorbs positive rescaling]
\label{lem:exact-absorb}
Let $h:\R^{n+1}\to\R$ and fix $x\in\R^{n+1}$ and a direction $d\neq 0$.
For any scalar $\tau>0$,
\[
\alpha^\star \in \arg\min_{\alpha>0} h(x+\alpha d)
\quad\Longleftrightarrow\quad
\frac{\alpha^\star}{\tau} \in \arg\min_{\alpha>0} h(x+\alpha\, \tau d).
\]
In particular, the \emph{step} $x+\alpha^\star d$ is invariant under rescaling $d\mapsto \tau d$.
\end{lemma}

\begin{proof}
Let $\tau>0$ and define the change of variable $\beta=\tau\alpha$.
Then
\[
\min_{\alpha>0} h(x+\alpha\,\tau d)
=
\min_{\beta>0} h(x+\beta d),
\]
and the minimizers satisfy $\beta^\star=\tau\alpha^\star$.
Hence $x+\alpha^\star \tau d = x+\beta^\star d$, proving the claim.
\end{proof}

\subsection{Exact invariance under general affine scaling}

We now prove the main invariance property for the affine-scaling model \eqref{eq:affine-model}.
The key is to reduce general $B$ to a unimodular map times an isotropic scaling,
and then use Lemma~\ref{lem:affine-covariance}, Corollary~\ref{cor:isotropic-scaling} and Lemma~\ref{lem:exact-absorb}.

\begin{theorem}[Affine-scaling equivalence under exact line search]
\label{thm:affine-scaling-invariance}
Let $f(x)=\phi(Bx)$ with $B$ invertible and $\det B>0$, and let $\{x_k\}$ be generated by YAND with exact line search
\begin{equation}
\label{eq:and-exact}
x_{k+1}
=
x_k+\alpha_k\, d_{\mathrm{AN}}^{f}(x_k),
\qquad
\alpha_k \in
\arg\min_{\alpha>0}
f\bigl(x_k+\alpha d_{\mathrm{AN}}^{f}(x_k)\bigr).
\end{equation}
Define $y_k:=Bx_k$.
Assume along the iterates the affine normal directions are well-defined for $f$ at $x_k$
and for $\phi$ at $y_k$ (e.g., the relevant level sets are elliptic).
Then $\{y_k\}$ coincides with the YAND iterates (with exact line search) applied directly to $\phi$:
\begin{equation}
\label{eq:and-exact-y}
y_{k+1}
=
y_k+\beta_k\, d_{\mathrm{AN}}^{\phi}(y_k),
\qquad
\beta_k \in
\arg\min_{\beta>0}
\phi\bigl(y_k+\beta d_{\mathrm{AN}}^{\phi}(y_k)\bigr).
\end{equation}
Consequently, after the change of variables \(y=Bx\), the mapped YAND iterates
for \(f\) coincide with the YAND iterates for \(\phi\). In this sense, the
behavior of the method is unaffected by affine scalings arising solely from \(B\),
and the corresponding convergence statements in \(y\)-space do not depend on
\(\kappa(B)\).
\end{theorem}

\begin{proof}
\noindent\textbf{Step 1: unimodular--scaling factorization.}
Let
\[
\rho := (\det B)^{1/(n+1)} >0,
\qquad
A := \rho^{-1}B,
\]
so that $B=\rho A$ and $\det A=1$.

\noindent\textbf{Step 2: rewrite $f$ through a unimodular transform.}
Define $\psi:\R^{n+1}\to\R$ by $\psi(z):=\phi(\rho z)$.
Then for every $x$,
\[
f(x)=\phi(Bx)=\phi(\rho Ax)=\psi(Ax).
\]

\noindent\textbf{Step 3: relate affine normals by unimodular covariance.}
Applying Lemma~\ref{lem:affine-covariance} to $g=f$ and $\psi$ yields
\[
d_{\mathrm{AN}}^{f}(x)\ \parallel\ A^{-1} d_{\mathrm{AN}}^{\psi}(Ax).
\]
Thus there exists a scalar $\eta(x)>0$ such that
\begin{equation}
\label{eq:d-relate-final}
d_{\mathrm{AN}}^{f}(x)=\eta(x)\, A^{-1} d_{\mathrm{AN}}^{\psi}(Ax).
\end{equation}

\noindent\textbf{Step 4: map the $x$-update into $y$-space.}
Let $y=Bx=\rho Ax$ and write $x^+=x+\alpha d_{\mathrm{AN}}^{f}(x)$.
Then
\[
y^+ := Bx^+ = y + \alpha\, B d_{\mathrm{AN}}^{f}(x)
= y + \alpha\, \rho A d_{\mathrm{AN}}^{f}(x).
\]
Using \eqref{eq:d-relate-final},
\[
\rho A d_{\mathrm{AN}}^{f}(x)
=
\rho A\bigl(\eta(x)A^{-1} d_{\mathrm{AN}}^{\psi}(Ax)\bigr)
=
(\rho\eta(x))\, d_{\mathrm{AN}}^{\psi}(Ax).
\]
Since $Ax=\rho^{-1}y$, we obtain
\begin{equation}
\label{eq:y-step-psi-final}
y^+ = y + \tilde\alpha\, d_{\mathrm{AN}}^{\psi}(\rho^{-1}y),
\qquad
\tilde\alpha := \alpha\,\rho\eta(x) >0.
\end{equation}

\noindent\textbf{Step 5: identify the direction as the affine normal for $\phi$.}
By Corollary~\ref{cor:isotropic-scaling}, with $\psi(z)=\phi(\rho z)$ and $y=\rho z$, we have
\[
d_{\mathrm{AN}}^{\psi}(\rho^{-1}y)\ \parallel\ d_{\mathrm{AN}}^{\phi}(y).
\]
Hence there exists $\xi(y)>0$ such that
\[
d_{\mathrm{AN}}^{\psi}(\rho^{-1}y)=\xi(y)\, d_{\mathrm{AN}}^{\phi}(y).
\]
Substituting into \eqref{eq:y-step-psi-final} yields
\begin{equation}
\label{eq:y-step-phi-final}
y^+ = y + \beta\, d_{\mathrm{AN}}^{\phi}(y),
\qquad
\beta := \tilde\alpha\, \xi(y) >0.
\end{equation}

\noindent\textbf{Step 6: exact line search yields identical $y$-steps.}
Since $f(x)=\phi(Bx)$ and $y=Bx$, we have
\[
\alpha^\star \in \arg\min_{\alpha>0} f(x+\alpha d_{\mathrm{AN}}^{f}(x))
\quad\Longleftrightarrow\quad
\alpha^\star \in \arg\min_{\alpha>0}\phi\bigl(y+\alpha\,B d_{\mathrm{AN}}^{f}(x)\bigr).
\]
By \eqref{eq:y-step-phi-final}, there exists $p>0$ such that
$B d_{\mathrm{AN}}^{f}(x)=p\, d_{\mathrm{AN}}^{\phi}(y)$, hence
\[
\arg\min_{\alpha>0}\phi\bigl(y+\alpha\,B d_{\mathrm{AN}}^{f}(x)\bigr)
=
\arg\min_{\alpha>0}\phi\bigl(y+\alpha p\, d_{\mathrm{AN}}^{\phi}(y)\bigr).
\]
Applying Lemma~\ref{lem:exact-absorb} shows that the resulting \emph{step} in $y$-space
coincides with that produced by exact line search along $d_{\mathrm{AN}}^{\phi}(y)$,
i.e., by $\beta^\star \in \arg\min_{\beta>0}\phi(y+\beta d_{\mathrm{AN}}^{\phi}(y))$.
Therefore, $y_{k+1}=Bx_{k+1}$ coincides with the YAND update \eqref{eq:and-exact-y} at each $k$.

Finally, any convergence property of YAND on $\phi$ transfers verbatim to $f$ under the change of variables $y=Bx$,
and no dependence on $\kappa(B)$ can enter at the iterate level in $y$-space.
\end{proof}

Theorem~\ref{thm:affine-scaling-invariance} formalizes this invariance: the
change of variables $y=Bx$ removes the spurious conditioning induced by $B$,
so the resulting convergence guarantees and constants are inherited from
$\phi$ and do not depend on $\kappa(B)$. The analysis shows that affine-normal directions transform covariantly under affine scalings of the coordinates. As a consequence, the geometric search direction is intrinsically insensitive to anisotropic affine distortions of the objective landscape.

\begin{remark}[Where the assumptions enter]
\label{rem:where-assumptions-enter}
The only nontrivial assumption in Theorem~\ref{thm:affine-scaling-invariance}
is that the affine normal direction is defined along the iterates.
In our framework this is ensured, for instance, when the visited level sets are locally elliptic
(see Theorem~\ref{thm:ANDescent}).
\end{remark}

\subsection{Extension to Armijo line search}

We now extend the invariance principle from exact line search to the standard inexact line search
based on Armijo (sufficient decrease) conditions.
Let
\[
d_k := d_{\mathrm{AN}}^{f}(x_k),
\]
and assume the direction is chosen with the \emph{descent orientation}, i.e.,
\begin{equation}
\label{eq:descent-orientation}
\nabla f(x_k)^\top d_k < 0.
\end{equation}
Let $\alpha_k>0$ satisfy the Armijo condition
\begin{equation}
\label{eq:armijo}
f(x_k+\alpha_k d_k)
\le
f(x_k)
+
c_1 \alpha_k \nabla f(x_k)^\top d_k,
\qquad
0<c_1<1.
\end{equation}

\begin{theorem}[Armijo invariance under orientation-preserving affine scaling]
\label{thm:armijo-scaling}
Let $f(x)=\phi(Bx)$, where $B\in GL(n+1)$ satisfies $\det B>0$.
Define $y_k:=Bx_k$.
Assume the affine normal directions are well-defined at $x_k$ for $f$ and at $y_k$ for $\phi$.
Then there exists $\tau_k>0$ such that
\[
B d_{\mathrm{AN}}^{f}(x_k)
=
\tau_k\, d_{\mathrm{AN}}^{\phi}(y_k).
\]
Let $\beta_k:=\alpha_k \tau_k$.
Then
\[
y_{k+1}
=
y_k+\beta_k d_{\mathrm{AN}}^{\phi}(y_k),
\]
and $\beta_k$ satisfies the Armijo condition for $\phi$ with the same constant $c_1$:
\[
\phi(y_k+\beta_k d_{\mathrm{AN}}^{\phi}(y_k))
\le
\phi(y_k)
+
c_1 \beta_k
\nabla\phi(y_k)^\top d_{\mathrm{AN}}^{\phi}(y_k).
\]
Consequently, Armijo-based YAND is affine invariant
under orientation-preserving affine scalings (after the change of variables $y=Bx$).
\end{theorem}

\begin{proof}
By Remark~\ref{rem:det-role}, the affine-normal directions satisfy
$B d_{\mathrm{AN}}^{f}(x_k)\parallel d_{\mathrm{AN}}^{\phi}(y_k)$; hence the stated identity holds for some $\tau_k>0$.
Using
\[
\nabla f(x)=B^\top\nabla\phi(Bx),
\]
we obtain
\[
\nabla f(x_k)^\top d_k
=
\nabla\phi(y_k)^\top (B d_k)
=
\tau_k
\nabla\phi(y_k)^\top d_{\mathrm{AN}}^{\phi}(y_k).
\]
Moreover,
\[
f(x_k+\alpha_k d_k)
=
\phi(y_k+\beta_k d_{\mathrm{AN}}^{\phi}(y_k)).
\]
Substituting these identities into \eqref{eq:armijo}
yields exactly the Armijo condition for $\phi$
with step size $\beta_k$.
\end{proof}

\subsection{Extension to strong Wolfe line search}

We now strengthen Theorem~\ref{thm:armijo-scaling}
by incorporating the curvature condition, i.e., strong Wolfe conditions.
Let
\[
d_k := d_{\mathrm{AN}}^{f}(x_k),
\]
and assume \eqref{eq:descent-orientation}.
Let $\alpha_k>0$ satisfy the strong Wolfe conditions
\begin{align}
\label{eq:wolfe1-final}
f(x_k+\alpha_k d_k)
&\le f(x_k)+c_1 \alpha_k \nabla f(x_k)^\top d_k,
\\
\label{eq:wolfe2-final}
\bigl|\nabla f(x_k+\alpha_k d_k)^\top d_k\bigr|
&\le c_2 \bigl|\nabla f(x_k)^\top d_k\bigr|,
\end{align}
with $0<c_1<c_2<1$.

\begin{theorem}[Strong-Wolfe invariance under orientation-preserving affine scaling]
\label{thm:wolfe-scaling-final}
Let $f(x)=\phi(Bx)$, where $B\in GL(n+1)$ satisfies $\det B>0$.
Assume YAND is implemented with strong Wolfe line search
\eqref{eq:wolfe1-final}--\eqref{eq:wolfe2-final},
and that directions are chosen with the descent orientation
\eqref{eq:descent-orientation}. Define $y_k:=Bx_k$.
Assume the affine normal directions are well-defined at $x_k$ for $f$ and at $y_k$ for $\phi$.
Then for each $k$ there exists a scalar $\tau_k>0$ such that
\begin{equation}
\label{eq:tau-def}
B d_{\mathrm{AN}}^{f}(x_k)
=
\tau_k\, d_{\mathrm{AN}}^{\phi}(y_k).
\end{equation}
Let $\beta_k := \alpha_k \tau_k$.
Then the mapped iterate satisfies
\[
y_{k+1}
=
y_k + \beta_k\, d_{\mathrm{AN}}^{\phi}(y_k),
\]
and $\beta_k$ satisfies the \emph{same} strong Wolfe conditions
(with parameters $(c_1,c_2)$) for $\phi$
along the direction $d_{\mathrm{AN}}^{\phi}(y_k)$.
Consequently, any convergence guarantee proved for YAND on \(\phi\)
under strong Wolfe conditions transfers verbatim
(after the change of variables \(y=Bx\)),
and the corresponding convergence statements in \(y\)-space do not explicitly depend on \(\kappa(B)\).
\end{theorem}

\begin{proof}
\noindent\textbf{Step 1: Mapping of directions and positivity of $\tau_k$.}
By Remark~\ref{rem:det-role}, $B d_{\mathrm{AN}}^{f}(x_k)\parallel d_{\mathrm{AN}}^{\phi}(y_k)$, hence \eqref{eq:tau-def} holds for some $\tau_k>0$.

\noindent\textbf{Step 2: Mapping of search curves.}
For any $\alpha>0$,
\[
Bx_k + \alpha\, B d_k
=
y_k + (\alpha\tau_k)\, d_{\mathrm{AN}}^{\phi}(y_k).
\]
Let $\beta := \alpha \tau_k$.

\noindent\textbf{Step 3: Sufficient decrease condition.}
Using $\nabla f(x)=B^\top \nabla\phi(Bx)$,
\[
\nabla f(x_k)^\top d_k
=
\nabla\phi(y_k)^\top (B d_k)
=
\tau_k\, \nabla\phi(y_k)^\top d_{\mathrm{AN}}^{\phi}(y_k),
\]
and
\[
f(x_k+\alpha d_k)
=
\phi(y_k+\beta d_{\mathrm{AN}}^{\phi}(y_k)).
\]
Substituting into \eqref{eq:wolfe1-final} gives
\[
\phi(y_k+\beta d_{\mathrm{AN}}^{\phi}(y_k))
\le
\phi(y_k)
+
c_1 \beta \nabla\phi(y_k)^\top d_{\mathrm{AN}}^{\phi}(y_k),
\]
which is exactly the strong Wolfe sufficient-decrease condition for $\phi$.

\noindent\textbf{Step 4: Curvature condition.}
Similarly,
\[
\nabla f(x_k+\alpha d_k)^\top d_k
=
\tau_k
\nabla\phi(y_k+\beta d_{\mathrm{AN}}^{\phi}(y_k))^\top
d_{\mathrm{AN}}^{\phi}(y_k),
\]
so \eqref{eq:wolfe2-final} becomes
\[
\bigl|
\nabla\phi(y_k+\beta d_{\mathrm{AN}}^{\phi}(y_k))^\top
d_{\mathrm{AN}}^{\phi}(y_k)
\bigr|
\le
c_2
\bigl|
\nabla\phi(y_k)^\top
d_{\mathrm{AN}}^{\phi}(y_k)
\bigr|,
\]
which is precisely the strong Wolfe curvature condition for $\phi$.
\end{proof}

\begin{remark}[Natural notion of invariance for inexact line search]
\label{rem:wolfe-invariance}
Theorem~\ref{thm:wolfe-scaling-final} does not assert that the numerical values of the step sizes $\alpha_k$
in the $x$-variables remain unchanged under affine scalings.
Rather, after the change of variables $y=Bx$ (with $\det B>0$) and the reparameterization
$\beta_k=\alpha_k\tau_k$, the accepted step in $x$-space corresponds exactly to an accepted
strong-Wolfe step in $y$-space along $d_{\mathrm{AN}}^{\phi}(y_k)$.
Consequently, the induced sequence \(\{y_k\}\) is identical to the sequence obtained by applying
the same line-search rule directly to \(\phi\), and the associated convergence properties in \(y\)-space
do not explicitly depend on \(\kappa(B)\).
\end{remark}

\subsection{Unified affine invariance of monotone line-search rules}

The previous results show that both Armijo and strong Wolfe conditions are preserved under affine scaling.
We now state a unified formulation.

\begin{theorem}[Unified invariance of first-order monotone line search]
\label{thm:unified-line-search}
Let $f(x)=\phi(Bx)$ with $B\in GL(n+1)$ satisfying $\det B>0$.
Assume the search direction $d_k=d_{\mathrm{AN}}^{f}(x_k)$ is chosen with descent orientation \eqref{eq:descent-orientation}.
Let $y_k=Bx_k$.
Assume the affine normal directions are well-defined at $x_k$ for $f$ and at $y_k$ for $\phi$. Suppose a step size $\alpha_k>0$ is accepted according to any line-search rule that depends only on
\begin{itemize}
\item function values $f(x_k+\alpha d_k)$,
\item and first-order directional derivatives $\nabla f(x_k+\alpha d_k)^\top d_k$,
\end{itemize}
through inequalities that are homogeneous of degree one with respect to the directional-derivative term.
Then there exists $\tau_k>0$ such that
\[
B d_k = \tau_k d_{\mathrm{AN}}^{\phi}(y_k),
\]
and, defining $\beta_k=\alpha_k \tau_k$, the step $\beta_k$ is accepted by the \emph{same} rule applied to $\phi$
along $d_{\mathrm{AN}}^{\phi}(y_k)$.
Consequently, the induced sequence $\{y_k\}$ is identical to the sequence obtained by applying
the same line-search rule directly to $\phi$, and the iteration complexity in $y$-space is independent of $\kappa(B)$.
(This covers, in particular, Armijo, Wolfe/strong Wolfe, and related monotone first-order rules.)
\end{theorem}

\begin{remark}[Structural origin of invariance]
\label{rem:structural-origin}
The invariance of Armijo and strong Wolfe conditions is not accidental.
It follows from two structural facts:
\begin{enumerate}
\item Function values transform by composition: $f(x)=\phi(Bx)$.
\item Directional derivatives transform linearly:
      $\nabla f(x)^\top d
      =
      \nabla\phi(Bx)^\top (B d)$.
\end{enumerate}
Since monotone line-search rules are expressed purely in terms of these two quantities,
and the affine normal direction transforms by a positive scalar factor (for $\det B>0$ with consistent orientation),
their acceptance mechanisms are preserved under orientation-preserving affine scalings.
This provides a complete line-search-level affine invariance theory for YAND.
\end{remark}

\subsection{Regime-wise transfer and explicit rates}
The affine covariance established above implies a transfer principle:
any convergence property of YAND proved for a base objective $\phi$
is inherited verbatim by all anisotropically scaled objectives
$f(x)=\phi(Bx)$.
Crucially, the rate constants depend only on the intrinsic geometry
of $\phi$ and not on the conditioning of $B$. The following transfer principle should be understood under the same
line-search regime and local regularity assumptions under which the
corresponding convergence result is established for the base objective \(\phi\).

\begin{corollary}[Regime-wise invariance under affine scaling]
\label{cor:regime-transfer-final}
Let $f(x)=\phi(Bx)$ with $B$ invertible and let $y_k=Bx_k$.
Under the standing assumption that the affine normal is well-defined along the iterates:
\begin{itemize}
\item[(i)] If YAND applied to $\phi$ is globally convergent (under the chosen line search),
then YAND applied to $f$ is globally convergent.
\item[(ii)] If YAND applied to $\phi$ enjoys a global linear rate under some condition
(e.g., strong convexity, PL, etc.), then YAND applied to $f$ enjoys the same linear rate with identical constants.
\item[(iii)] If YAND applied to $\phi$ is locally quadratically convergent near a nondegenerate minimizer,
then the same local quadratic convergence holds for YAND applied to $f$.
\end{itemize}
In all cases, the convergence rates and associated constants
coincide with those for \(\phi\) and therefore do not explicitly depend on \(\kappa(B)\)
within the affine-scaling model \(f(x)=\phi(Bx)\).
\end{corollary}

\begin{proof}
Under exact line search, Theorem~\ref{thm:affine-scaling-invariance} shows that $\{y_k\}$ coincides with the YAND iterates on $\phi$.
Under strong Wolfe, Theorem~\ref{thm:wolfe-scaling-final} shows that $\{y_k\}$ follows the same accepted steps as YAND on $\phi$
after reparameterization.
Therefore, any convergence statement for YAND on $\phi$ transfers directly to YAND on $f$ under $y=Bx$,
with the same constants.
No dependence on $\kappa(B)$ can enter because $B$ is eliminated by the change of variables.
\end{proof}

As a concrete instantiation of Corollary~\ref{cor:regime-transfer-final},
we record an explicit linear-rate bound (and thus an iteration-complexity bound)
by invoking the exact-line-search rate established in Section~\ref{sec:YAND-convergence}.

\begin{corollary}[Function-value complexity transfer under affine scaling (exact line search)]
\label{cor:kappa-free-final}
Assume that YAND with exact line search applied to $\phi$ enjoys a global linear rate
\[
\phi(y_k)-\phi^\star \le (1-\theta)^k\bigl(\phi(y_0)-\phi^\star\bigr)
\qquad\text{for some }\theta\in(0,1),
\]
under certain regularity conditions on $\phi$.
Let $f(x)=\phi(Bx)$ with $B$ invertible and $\det B>0$, and let $\{x_k\}$ be generated by YAND with exact line search on $f$.
Then
\[
f(x_k)-f^\star \le (1-\theta)^k\bigl(f(x_0)-f^\star\bigr),
\]
and to reach $f(x_k)-f^\star\le \varepsilon$ it suffices that
\[
k \ge 
\left\lceil
\frac{1}{\theta}\log\!\left(\frac{f(x_0)-f^\star}{\varepsilon}\right)
\right\rceil.
\]
All constants are inherited from the base objective \(\phi\) and therefore
do not explicitly depend on \(\kappa(B)\) within this affine-scaling model.
\end{corollary}

\begin{proof}
By Theorem~\ref{thm:affine-scaling-invariance}, $y_k:=Bx_k$ coincides with the YAND iterates on $\phi$,
and $f(x_k)=\phi(y_k)$ with $f^\star=\phi^\star$. The claim follows immediately.
\end{proof}

This shows that, within the affine-scaling model \(f(x)=\phi(Bx)\),
ill-conditioning arising purely from anisotropic affine scaling does not
affect the mapped YAND dynamics: the relevant convergence constants are
inherited from the base function \(\phi\) (e.g., through \(\mu\), \(L\),
and the geometric bound \(T\) in the exact-line-search analysis) rather
than from \(\kappa(B)\).

\subsection{Illustrative examples}

\paragraph{Example 1: anisotropic quadratic scaling.}
Let
\[
\phi(y)=\tfrac12 \|y\|^2,
\qquad
f(x)=\tfrac12 \|Bx\|^2,
\]
with $B=\diag(1,\dots,1,\gamma)$ for some $\gamma\ge 1$.
Then $\nabla^2 f = B^\top B$ and $\kappa(\nabla^2 f)=\kappa(B)^2=\gamma^2$.
Gradient descent requires $O(\gamma^2\log(1/\varepsilon))$ iterations.
In contrast, by the quadratic equivalence established earlier,
YAND coincides with Newton's method on strictly convex quadratics
and reaches the minimizer in one step, independently of $\gamma$.

\paragraph{Example 2: feature scaling in $\ell_2$-regularized logistic regression.}
Let
\[
\phi(y)
=
\sum_{i=1}^m
\log\bigl(1+\exp(-b_i a_i^\top y)\bigr)
+\tfrac{\lambda}{2}\|y\|^2,
\qquad \lambda>0,
\]
and consider $f(x)=\phi(Bx)$ where $B$ is diagonal and highly anisotropic (feature scaling).
While the smoothness constants and Hessian conditioning of $f$ may deteriorate with $\kappa(B)$,
Theorems~\ref{thm:affine-scaling-invariance} and~\ref{thm:wolfe-scaling-final}
show that YAND behaves as if the scaling were absent in the transformed coordinates \(y=Bx\):
its accepted steps and progress are governed by the intrinsic geometry of \(\phi\)
rather than by the spurious anisotropy induced by \(B\).

\medskip

Overall, for objective functions whose ill-conditioning is induced purely
by affine scaling, YAND inherits the convergence behavior of the underlying
unscaled objective after the change of variables \(y=Bx\).
The results should not be interpreted as asserting that the global
complexity of YAND is condition-number-independent for arbitrary problems.
Rather, they show that under the affine-scaling model, the search direction
and the induced mapped iterates are invariant with respect to anisotropic
affine distortions of the coordinate system.

\section{Numerical experiments}
\label{sec:numerics}

In this section we present a series of numerical experiments designed to
illustrate the geometric behavior of the proposed YAND. The goal is not large-scale benchmarking, but rather to
verify the main theoretical predictions of the paper and to examine how
the method behaves across representative geometric regimes.

The experiments are organized in three stages. We first study convex
quadratic problems, where the theory predicts that the affine-normal
direction coincides with the Newton direction and exhibits affine-scaling
robustness. We then turn to smooth nonquadratic convex objectives to show
that the favorable behavior of YAND is not restricted to the quadratic
setting. Finally, we consider smooth nonconvex problems with curved
valleys, saddle regions, and multi-well landscapes in order to assess the
stability of the method beyond convexity. The following subsections examine these behaviors on representative
examples of increasing geometric complexity.

\subsection{Experimental setup}

All experiments were conducted on a Windows\,11 laptop, 
MATLAB~R2025b with an Intel(R)~Core(TM)~Ultra\,9\,275HX CPU. The implementation is in MATLAB and uses analytic derivatives (via automatic differentiation). No external optimization libraries were used. Unless otherwise specified, the following default parameters are used.

\paragraph{Algorithmic parameters.}
\begin{itemize}
\item maximum number of iterations: $\texttt{maxIter}=200$,
\item gradient-norm stopping tolerance: \(\texttt{tolGrad} = 10^{-4}\),
  \item initial step length for inexact line searches: \(\texttt{alpha0} = 1\),
  \item upper bound for exact line search: \(\texttt{alpha\_max} = 10\),
  \item Armijo backtracking parameter: \(\rho = 0.5\),
\item strong Wolfe parameters: $c_1 = 10^{-4}$ and $c_2 = 0.9$.
\end{itemize}

\paragraph{Line-search strategies.}
For each test problem we employ three standard step-size rules:
\begin{enumerate}
  \item \textbf{Exact line search:}  
        one-dimensional minimization of $f(x_k+\alpha d_k)$ over
      $\alpha\in[0,\alpha_{\max}]$.
  \item \textbf{Armijo backtracking:}  
        sufficient-decrease condition with parameter $\rho$.
  \item \textbf{Strong Wolfe:}  
        the standard Armijo and curvature conditions with parameters
      $(c_1,c_2)$.
\end{enumerate}
In all cases the search direction is the affine-normal direction;
only the step-size selection differs.

\paragraph{Quantities reported.}
For every run we display three diagnostic plots:
\begin{enumerate}
  \item the YAND trajectory overlaid on level sets of the objective function \(f\),
  \item the semilog plot of the objective gap \(f(x_k) - f^\star\), 
        where the optimal objective value \(f^\star\) is known, 
  \item the semilog plot of the gradient norm \(\|\nabla f(x_k)\|_2\).
\end{enumerate}
This combination visualizes both the global path geometry and the local convergence rate.

\paragraph{Classes of test problems.}
To evaluate YAND under different curvature and conditioning regimes, the experiments are grouped into three categories:
\begin{itemize}
  \item \textbf{Convex quadratic problems:}  
        including well-conditioned and severely ill-conditioned instances;
  \item \textbf{Smooth convex nonquadratic problems:}
        nonlinear objectives with  tunable curvature and conditioning;
  \item \textbf{Smooth nonconvex problems:}  examples containing saddle regions, curved valleys, and
      multi-well landscapes.
\end{itemize}
All problems are posed in two dimensions to enable clear visualization
of level sets and optimization trajectories and to highlight the
connection between the numerical behavior and the theoretical
convergence results.

\subsection{Convex quadratic problems}
\label{subsec:quad-convex}

We begin with convex quadratic objectives, for which the theory predicts
that the affine-normal direction coincides with the Newton direction.
These examples serve as a baseline and provide the cleanest setting in
which to visualize both quadratic exactness and affine-scaling robustness.

\subsubsection{Well-conditioned quadratic}

We first consider the simple quadratic
\begin{equation}
\label{eq:quad-well}
f(x) \;=\; \tfrac12 x^\top A x + b^\top x,
\qquad
A = \begin{bmatrix} 2 & 0 \\[0.2em] 0 & 8 \end{bmatrix},
\quad
b = \begin{bmatrix} 0.1 \\[0.2em] 0.2 \end{bmatrix}.
\end{equation}
The unique minimizer is $x^\star = -A^{-1}b = (-0.05,-0.025)^\top$,
and the experiment is initialized at $x_0=(1,1)^\top$. Since $A$ is diagonal with positive eigenvalues, the level sets of~\eqref{eq:quad-well} are ellipses. For this strictly convex quadratic, the affine normal direction coincides (up to scaling) with the Newton direction, and the theory predicts essentially one-step convergence with exact line search.
The numerical results confirm this behavior.

Figure~\ref{fig:quad-well-AND} shows that YAND reaches the minimizer
in one iteration under exact line search, while the Wolfe and Armijo
variants require only a few additional steps due to their inexact
step sizes. In all three cases, the semilog plots of $f(x_k)-f^\star$ and
$\|\nabla f(x_k)\|_2$ show the rapid local convergence predicted by the
theoretical analysis.

\begin{figure}[h!]
  \centering
  \begin{subfigure}[t]{0.9\linewidth}
    \centering
    \includegraphics[width=\linewidth]{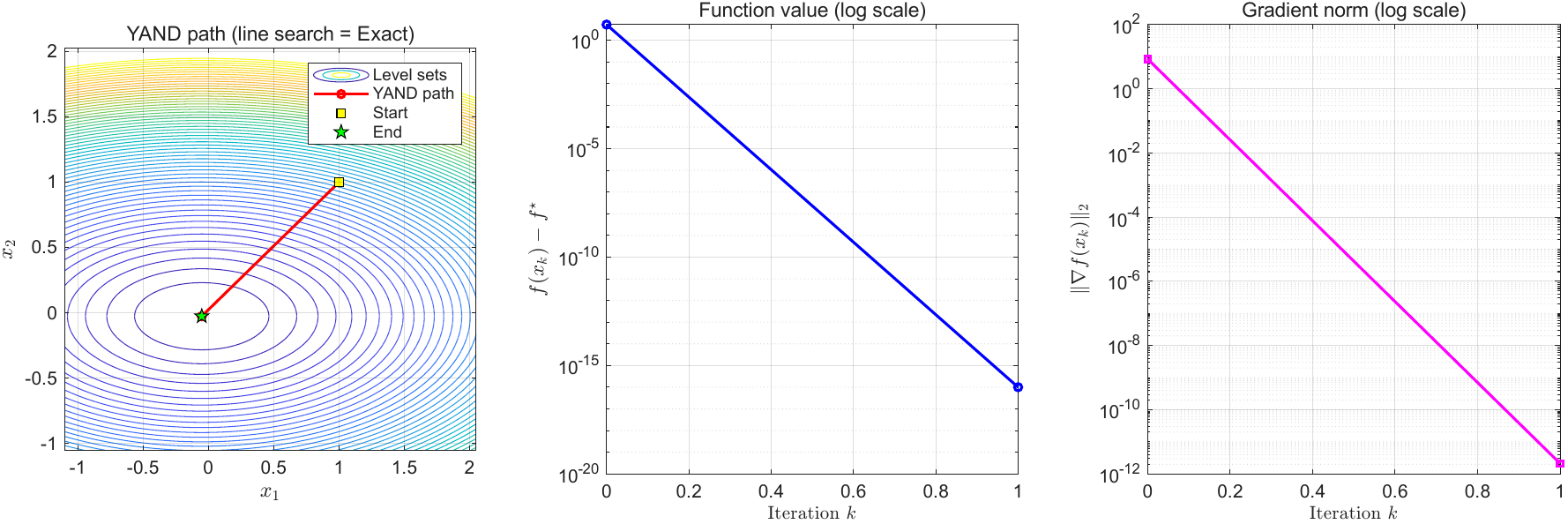}
    \caption{Exact line search.}
    \label{fig:quad-well-exact}
  \end{subfigure}
  \hfill
  \begin{subfigure}[t]{0.9\linewidth}
    \centering
    \includegraphics[width=\linewidth]{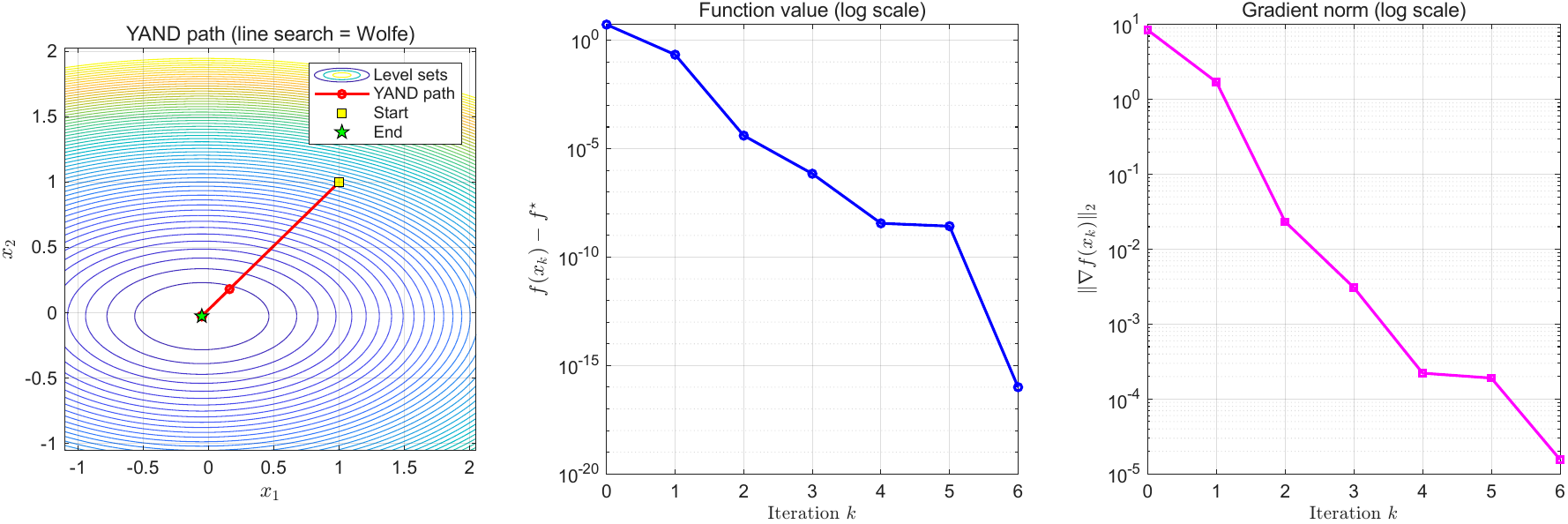}
    \caption{Strong Wolfe line search.}
    \label{fig:quad-well-wolfe}
  \end{subfigure}
  \hfill
  \begin{subfigure}[t]{0.9\linewidth}
    \centering
    \includegraphics[width=\linewidth]{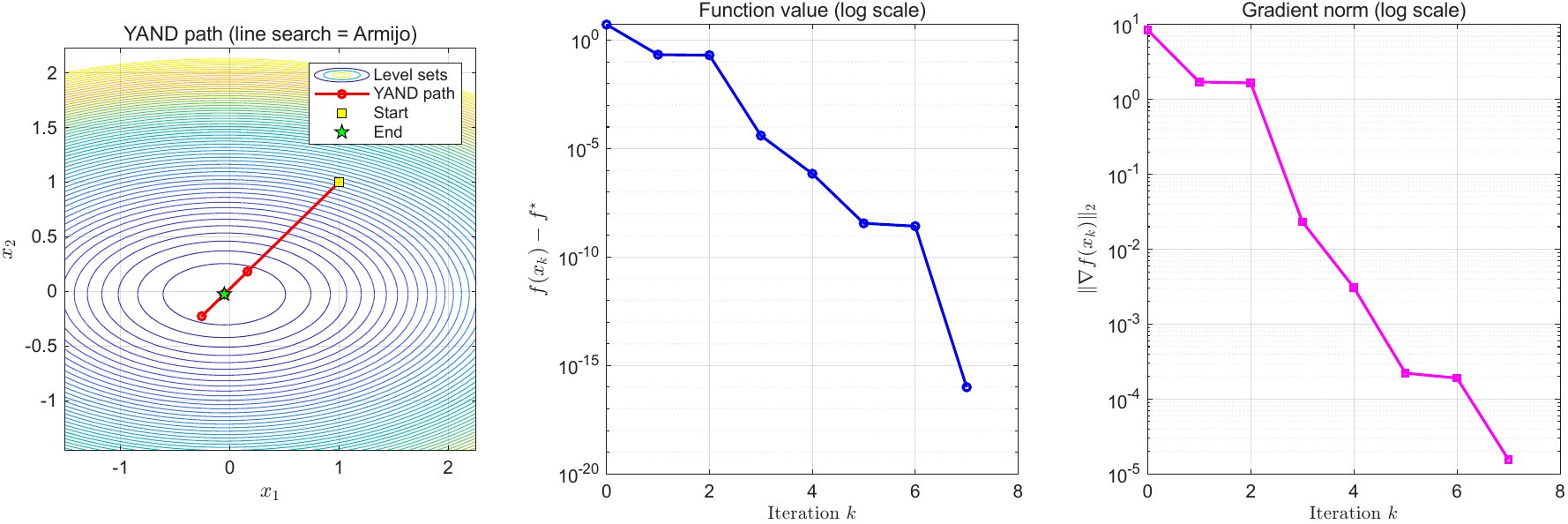}
    \caption{Armijo backtracking.}
    \label{fig:quad-well-armijo}
  \end{subfigure}
  \caption{YAND on the well-conditioned quadratic~\eqref{eq:quad-well} with three different line-search strategies. Each panel shows (from left to right) the YAND trajectory on level sets, the function value $f(x_k)-f^\star$ (log scale), 
and the gradient norm $\|\nabla f(x_k)\|_2$ (log scale).}
  \label{fig:quad-well-AND}
\end{figure}

\subsubsection{Ill-conditioned and affine-scaled quadratics}
\label{sec:affine-scaling-test}

To isolate the effect of affine scaling, we consider the base quadratic
\[
\phi(y)=\tfrac12\|y\|^2=\tfrac12(y_1^2+y_2^2)
\]
and construct a family of functions obtained through the affine change of
variables
\[
f_\gamma(x)=\phi(B_\gamma x), \qquad 
B_\gamma=\mathrm{diag}(1,\gamma),
\]
which yields
\[
f_\gamma(x)=\tfrac12(x_1^2+\gamma^2 x_2^2).
\]
In this model, the ill-conditioning is induced entirely by the affine
transformation $y=B_\gamma x$.
In particular,
\[
\kappa(B_\gamma)=\gamma,
\qquad
\kappa(\nabla^2 f_\gamma)=\kappa(B_\gamma^\top B_\gamma)=\gamma^2.
\]
We test the values
\[
\gamma \in \{1,10,10^2,10^3,10^4\},
\]
starting from the common initial point $x_0=(1,1)^\top$.
The following methods are compared:
YAND with exact line search (YAND-Exact),
YAND with strong Wolfe line search (YAND-Wolfe),
YAND with Armijo backtracking (YAND-Armijo),
gradient descent with exact line search (GD-Exact),
gradient descent with fixed step size $\alpha=1/\gamma^2$ (GD-Fixed),
and Newton's method (Newton).

\begin{figure}[t]
\centering
\includegraphics[width=\textwidth]{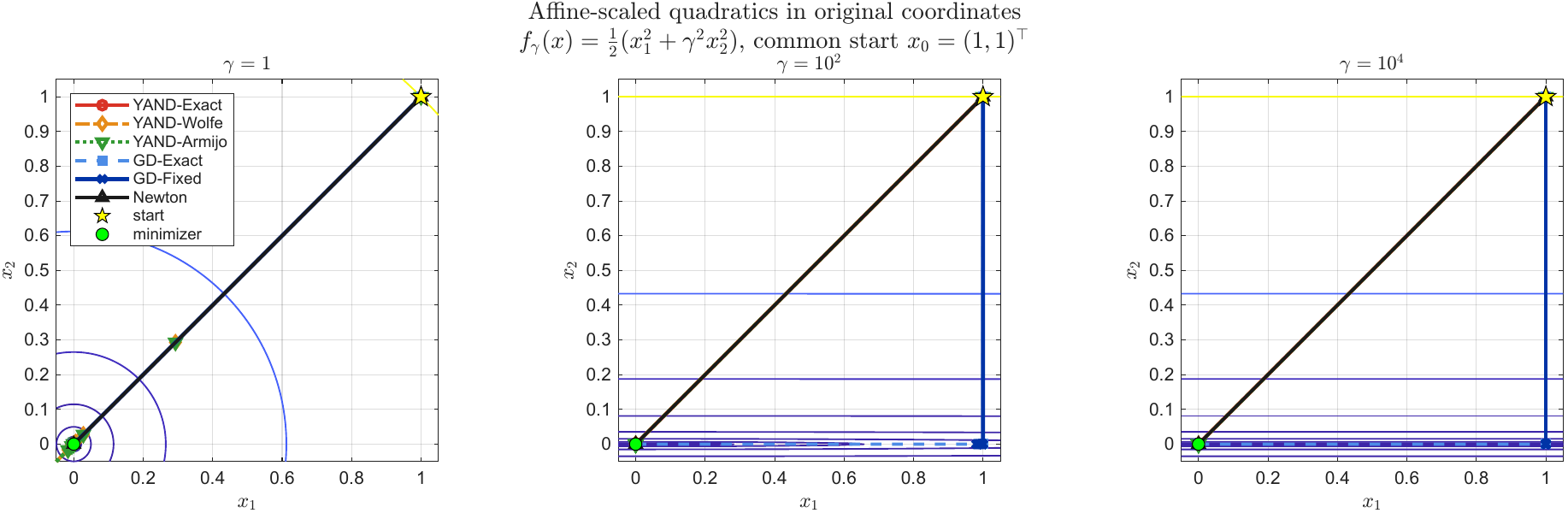}
\caption{Optimization trajectories in the original $x$-coordinates for
$f_\gamma(x)=\tfrac12(x_1^2+\gamma^2 x_2^2)$ with $\gamma=1,10^2,10^4$.
As $\gamma$ increases, the level sets become increasingly elongated.
YAND-Exact and Newton remain essentially one-step methods, while the
behavior of gradient descent depends more strongly on the step-size rule.
In particular, GD-Fixed becomes substantially slower as the anisotropy
increases, whereas GD-Exact remains convergent on this diagonal quadratic.}
\label{fig:affine-xspace}
\end{figure}

Figure~\ref{fig:affine-xspace} shows the optimization trajectories in the
original $x$-coordinates for representative values
$\gamma=1,10^2,10^4$.
As $\gamma$ increases, the level sets become increasingly elongated,
resulting in a progressively more anisotropic landscape.
For this axis-aligned quadratic, GD-Exact remains convergent, whereas
GD-Fixed becomes much more sensitive to the conditioning.
By contrast, both Newton's method and YAND-Exact reach the minimizer in
essentially one step for all tested values of $\gamma$.

To highlight the intrinsic affine invariance of the affine-normal
direction, we map the iterates to the normalized coordinates
$
y = B_\gamma x .
$
Figure~\ref{fig:affine-yspace} displays the corresponding YAND-Exact
trajectories in the $y$-coordinates.
After this normalization, the trajectories nearly collapse onto the same
path, indicating that the convergence behavior of YAND is governed
primarily by the geometry of the base function $\phi$ rather than by the
artificial conditioning introduced through the affine scaling.

\begin{figure}[t]
\centering
\includegraphics[width=0.4\textwidth]{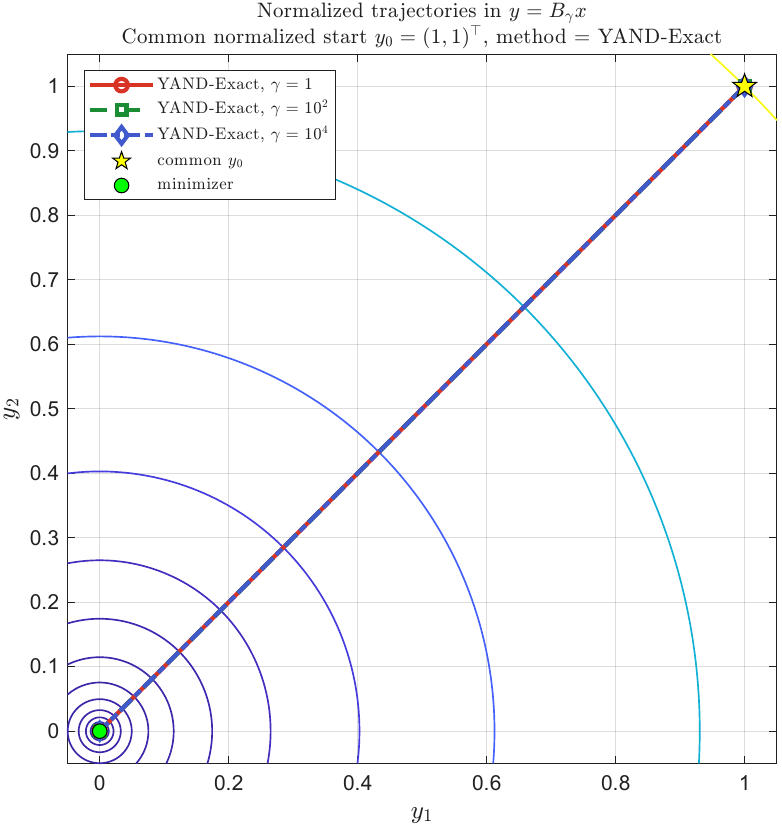}
\caption{
YAND trajectories after normalization $y=B_\gamma x$ for
$\gamma=1,10^2,10^4$.
After mapping to the intrinsic coordinates of
$\phi(y)=\tfrac12\|y\|^2$, the trajectories collapse onto nearly
identical paths, illustrating the affine invariance predicted by the
theory.
}
\label{fig:affine-yspace}
\end{figure}

\begin{table}[t]
\centering
\caption{Effect of affine scaling on algorithm performance for $f_\gamma(x)$.}
\label{tab:affine-scaling}

\footnotesize
\setlength{\tabcolsep}{3pt}
\renewcommand{\arraystretch}{1.12}

\begin{adjustbox}{max width=\textwidth,center}
\begin{tabular}{
>{\centering\arraybackslash}p{0.8cm}
>{\centering\arraybackslash}p{1.0cm}
>{\centering\arraybackslash}p{1.5cm}
>{\centering\arraybackslash}p{1.9cm}
>{\centering\arraybackslash}p{1.9cm}
>{\centering\arraybackslash}p{1.99cm}
>{\centering\arraybackslash}p{1.9cm}
>{\centering\arraybackslash}p{1.55cm}
>{\centering\arraybackslash}p{1.1cm}
}
\toprule
$\gamma$ & $\kappa(B_\gamma)$ & $\kappa(\nabla^2 f_\gamma)$
& YAND-Exact & YAND-Wolfe & YAND-Armijo & GD-Exact & GD-Fixed & Newton \\
\midrule
$1$      & $1$      & $1$      & 1 & 10 & 11 & 1 & 1 & 1 \\
$10^{1}$ & $10^{1}$ & $10^{2}$ & 1 & 8  & 10 & 7 & $200^{\ast}$ & 1 \\
$10^{2}$ & $10^{2}$ & $10^{4}$ & 1 & 7  & 13 & 5 & $200^{\ast}$ & 1 \\
$10^{3}$ & $10^{3}$ & $10^{6}$ & 1 & 3  & 10 & 5 & $200^{\ast}$ & 1 \\
$10^{4}$ & $10^{4}$ & $10^{8}$ & 1 & 2  & 12 & 5 & $200^{\ast}$ & 1 \\
\bottomrule
\end{tabular}
\end{adjustbox}

\vspace{2pt}
\footnotesize
Entries marked with $^\ast$ reached the iteration cap before satisfying the stopping tolerance.
\end{table}

Table~\ref{tab:affine-scaling} reports the iteration counts for the
tested values of $\gamma$.
As predicted by the quadratic theory, both YAND-Exact and Newton converge
in a single step for all $\gamma$.
The Armijo variant of YAND remains stable, and the Wolfe variant also
shows robust practical behavior in this experiment.
Among the two gradient-descent baselines, GD-Exact remains convergent on
this diagonal quadratic, while GD-Fixed deteriorates much more severely
and reaches the iteration cap for all cases with $\gamma \ge 10$.
Overall, the results confirm the affine-scaling robustness predicted by
the theoretical analysis in Section~\ref{sec:affine-scaling} and show
that the exact affine-normal step is essentially insensitive to the
artificial conditioning induced by $B_\gamma$.

Taken together, the quadratic experiments confirm two central features of
YAND: exact agreement with Newton's method on strictly convex quadratics
and strong robustness with respect to affine scaling. We next examine
whether similarly favorable behavior persists beyond the quadratic setting.

\subsection{Smooth nonquadratic convex problems}
\label{subsec:general-convex}

We next consider smooth convex functions that are not quadratic.
In this regime the affine-normal direction no longer coincides with
either the gradient direction or the Newton direction.
These experiments therefore probe the genuine behavior of YAND beyond
the quadratic setting and test whether its geometry-adaptive character
persists for nonlinear convex objectives with strongly varying curvature.

\subsubsection{Sixth-degree anisotropic polynomial}

We first consider the sixth-degree convex polynomial
\begin{equation}\label{eq:convex-poly6-well}
    f(x)
    \;=\;
    \bigl(x_1^2 + 4x_2^2\bigr)^3
    \;+\;
    0.1\,(x_1^2 + x_2^2)
    \;+\;
    0.01\,(x_1 + 2x_2),
\end{equation}
We use the initial point $x_0 = (0.5,-0.5)^\top$.
The sixth-degree term produces a steep convex bowl with strongly
anisotropic curvature, while the small linear perturbation breaks
the radial symmetry and ensures that the affine-normal direction
does not coincide with either the gradient or the Newton direction.
Consequently, exact line search no longer terminates in a single iteration,
allowing us to observe the characteristic curvature-adaptive behavior of YAND.

\begin{figure}[h!]
  \centering
  \begin{subfigure}[t]{0.9\linewidth}
    \centering
    \includegraphics[width=\linewidth]{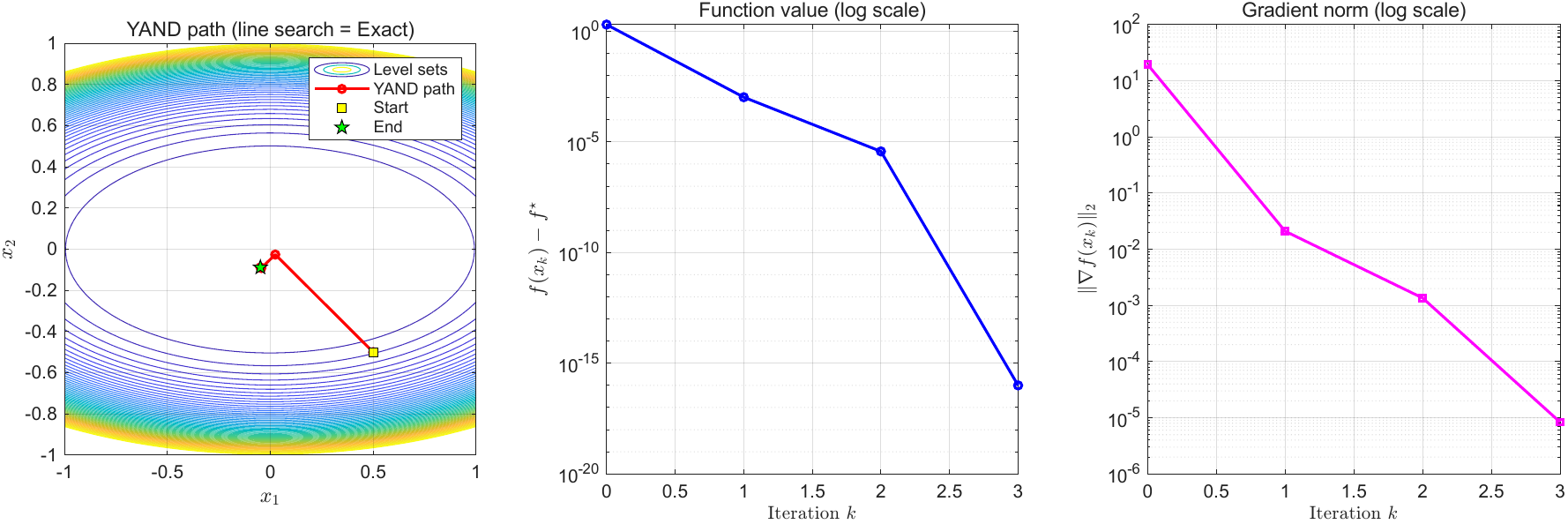}
    \caption{Exact line search.}
    \label{fig:convex-poly6-well-exact}
  \end{subfigure}
  \hfill
  \begin{subfigure}[t]{0.9\linewidth}
    \centering
    \includegraphics[width=\linewidth]{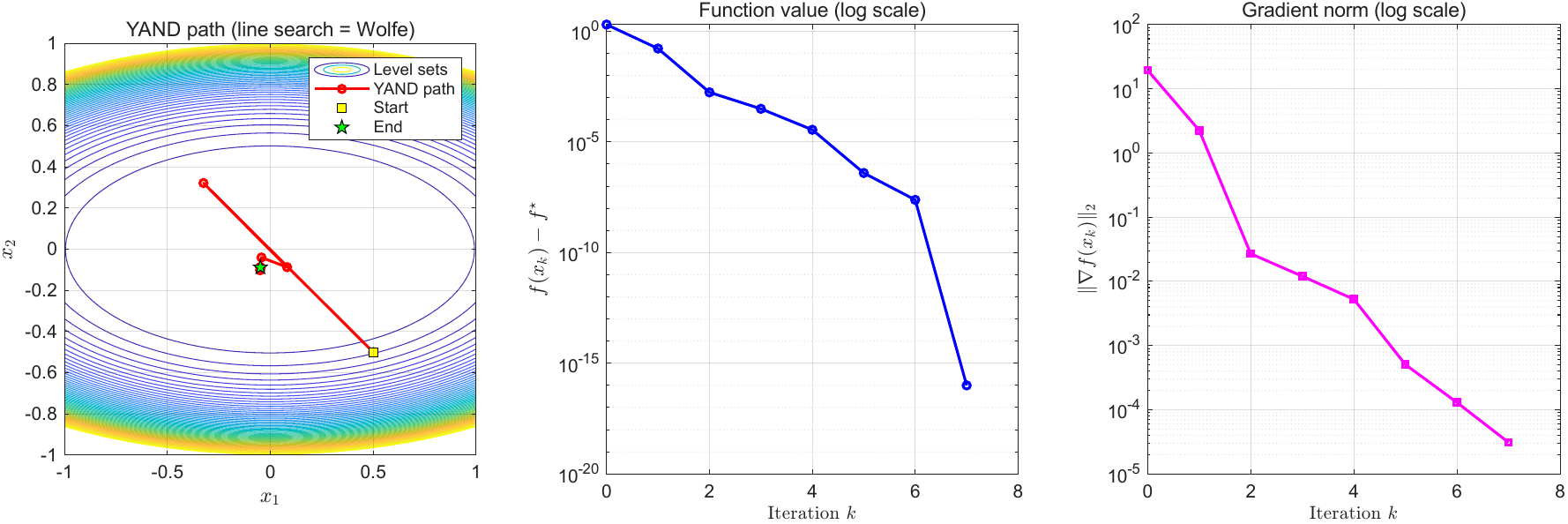}
    \caption{Strong Wolfe line search.}
    \label{fig:convex-poly6-well-wolfe}
  \end{subfigure}
  \hfill
  \begin{subfigure}[t]{0.9\linewidth}
    \centering
    \includegraphics[width=\linewidth]{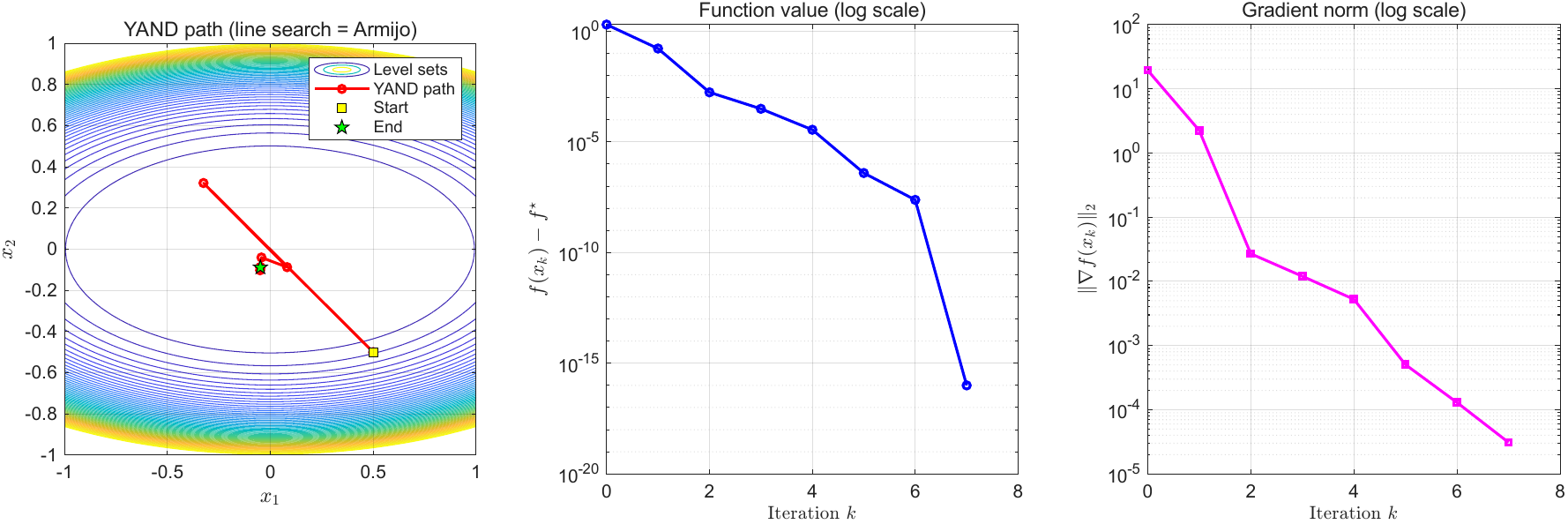}
    \caption{Armijo backtracking.}
    \label{fig:convex-poly6-well-armijo}
  \end{subfigure}
  \caption{YAND on the smooth convex polynomial
\eqref{eq:convex-poly6-well}. Each panel shows the optimization
trajectory together with the semilog plots of $f(x_k)-f^\star$
and $\|\nabla f(x_k)\|_2$.}
  \label{fig:convex-poly6-well-AND}
\end{figure}

Figure~\ref{fig:convex-poly6-well-AND} shows the trajectories produced by the three
line-search schemes.  
All variants converge rapidly. The exact line search produces the most direct trajectory
and reaches the minimizer in only three iterations,
illustrating the curvature-adaptive nature of the affine-normal step, with the path bending in a manner
that faithfully reflects the anisotropic curvature of the objective.  
The strong Wolfe and Armijo rules exhibit more conservative step sizes, as expected from
their inexact step conditions, but still maintain fast and stable convergence.

The semilog plots of $f(x_k)-f^\star$ and $\|\nabla f(x_k)\|_2$ display smooth,
monotone decay for all three methods, fully consistent with the theoretical guarantees for
smooth strongly convex functions.  
Overall, this example demonstrates that YAND remains robust and curvature-aware even for
high-order, nonquadratic convex objectives with pronounced anisotropy.

\subsubsection{Ill-conditioned convex inverse-barrier problem}

We next consider a smooth yet highly ill-conditioned convex objective
obtained by adding an inverse barrier to a quadratic bowl:
\begin{equation}\label{eq:convex-inversebarrier}
    f(x)
    = \frac12\,(x_1^2 + x_2^2)
      \;+\;
      \frac{\mu}{\,d - x_1 - x_2\,},
    \qquad
    \mu = 1,\; d = 1,
\end{equation}
defined on the open half-space $\{\,x \in \mathbb{R}^2 : x_1 + x_2 < d\,\}$.
The barrier term induces rapidly increasing curvature as the iterate
approaches the affine boundary $x_1 + x_2 = d$, resulting in a strongly
convex problem with extreme anisotropy and a highly ill-conditioned local
Hessian:
\[
\nabla^2 f(x)
\;=\;
I
\;+\;
\frac{\mu}{(d - x_1 - x_2)^3}
\begin{bmatrix}
1 & 1\\[0.2em] 1 & 1
\end{bmatrix},
\]
whose dominant eigenvalue grows on the order of
$(d-x_1-x_2)^{-3}$ as the boundary is approached.  
To expose this ill-conditioning, we initialize at a point extremely close to the
feasible boundary:
\[
x_0 = (0.01,\; 0.98)^\top,
\qquad x_1 + x_2 = 0.99 < 1,
\]
where the local condition number is already of order $10^6$. The unique minimizer for this problem lies along the symmetry line
$x_1 = x_2 = t$, where the first-order condition reduces to solving the cubic
\[
s(1-s)^2 + 2 = 0,\qquad s = x_1 + x_2.
\]
Its closed-form solution is
\[
s^\star
= \frac{2}{3}
   - \frac{1}{3}\!\left(\,(3\sqrt{87}+28)^{1/3}
                       + (3\sqrt{87}+28)^{-1/3}\right),
\qquad
x_1^\star = x_2^\star = \tfrac12 s^\star .
\]
Numerically,
\[
x^\star \approx (-0.3478103848,\,-0.3478103848)^\top,
\qquad
f^\star \approx 0.7107265761.
\]

\begin{figure}[h!]
\centering
\begin{subfigure}{0.9\textwidth}
    \centering
    \includegraphics[width=\linewidth]{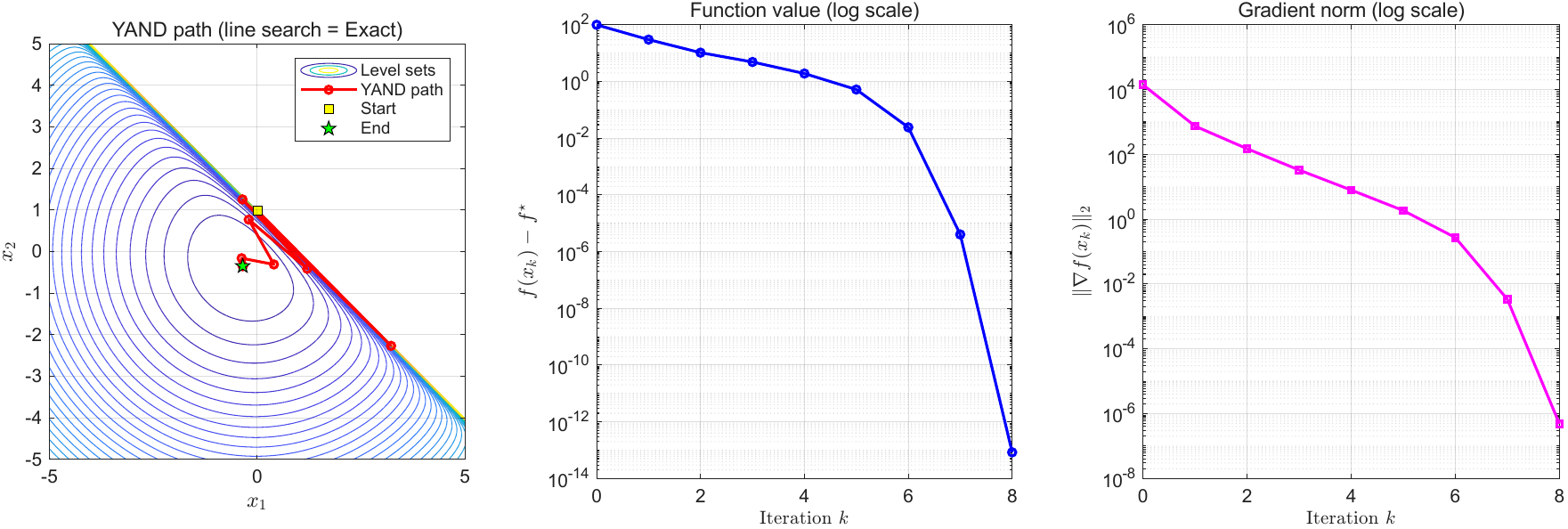}
    \caption{Exact line search}
    \label{fig:logbarrier-exact}
\end{subfigure}
\hfill
\begin{subfigure}{0.9\textwidth}
    \centering
    \includegraphics[width=\linewidth]{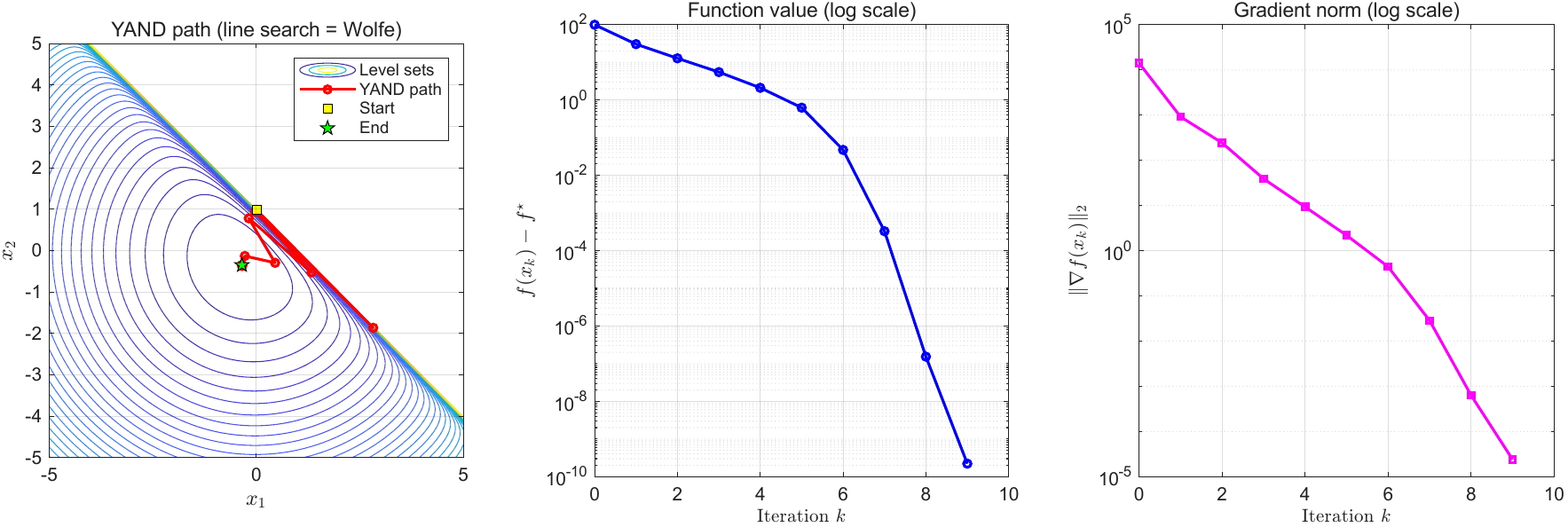}
    \caption{Strong Wolfe line search}
    \label{fig:logbarrier-wolfe}
\end{subfigure}
\hfill
\begin{subfigure}{0.9\textwidth}
    \centering
    \includegraphics[width=\linewidth]{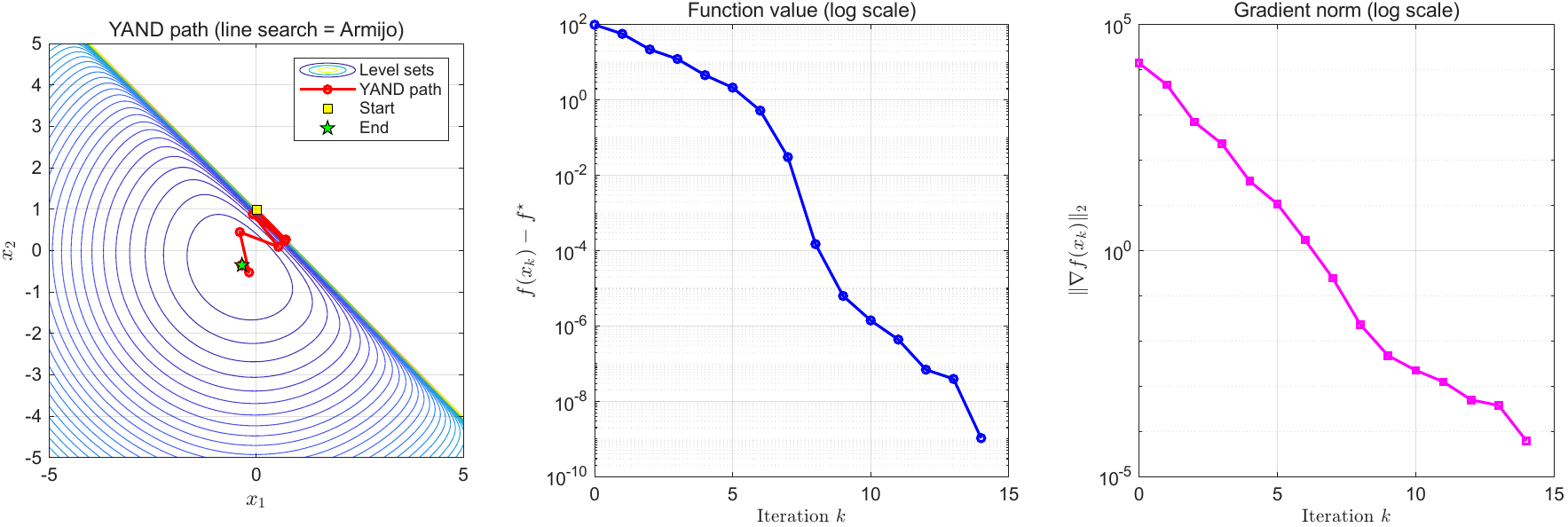}
    \caption{Armijo backtracking}
    \label{fig:logbarrier-armijo}
\end{subfigure}

\caption{
YAND on the ill-conditioned convex inverse barrier problem
\eqref{eq:convex-inversebarrier}.  
The initial point $x_0=(0.01,0.98)^\top$ lies extremely close to the affine
boundary $x_1+x_2=1$.
}
\label{fig:convex-inversebarrier-AND}
\end{figure}

Figure~\ref{fig:convex-inversebarrier-AND} reports the performance of the three
line-search variants of YAND.  
All variants remain stable and converge rapidly despite the extremely
large curvature variations near the boundary.
This example illustrates that the affine-normal direction adapts naturally
to strong anisotropy in the Hessian and remains effective even when the
local conditioning becomes extremely poor.

Taken together, the above examples demonstrate that the favorable
behavior of YAND is not restricted to quadratic objectives.
Even for highly nonlinear convex functions with strongly varying
curvature, the affine-normal direction adapts to the local geometry
and yields stable and efficient convergence.

\subsection{Smooth nonconvex problems}
\label{subsec:nonconvex-experiments}

We finally turn to smooth nonconvex objectives, where the geometry may
include curved valleys, saddle regions, and multiple basins of attraction.
These experiments investigate whether the favorable geometric behavior
observed in the convex setting persists beyond convexity, and whether the
empirical performance of YAND remains consistent with the convergence
theory developed earlier under standard line-search conditions.

\subsubsection{Rosenbrock function}

We first consider the classical Rosenbrock function
\begin{equation}\label{eq:rosenbrock}
f(x)=100(x_2-x_1^2)^2+(1-x_1)^2,
\end{equation}
whose global minimizer is $x^\star=(1,1)^\top$.
This example isolates the behavior of YAND in a narrow curved valley,
which is the most classical source of difficulty for smooth nonconvex
optimization. We initialize the method at the standard starting point
$
x_0=(-1.2,1.0)^\top .
$

\begin{figure}[h!]
\centering
\begin{subfigure}{0.9\textwidth}
    \centering
    \includegraphics[width=\linewidth]{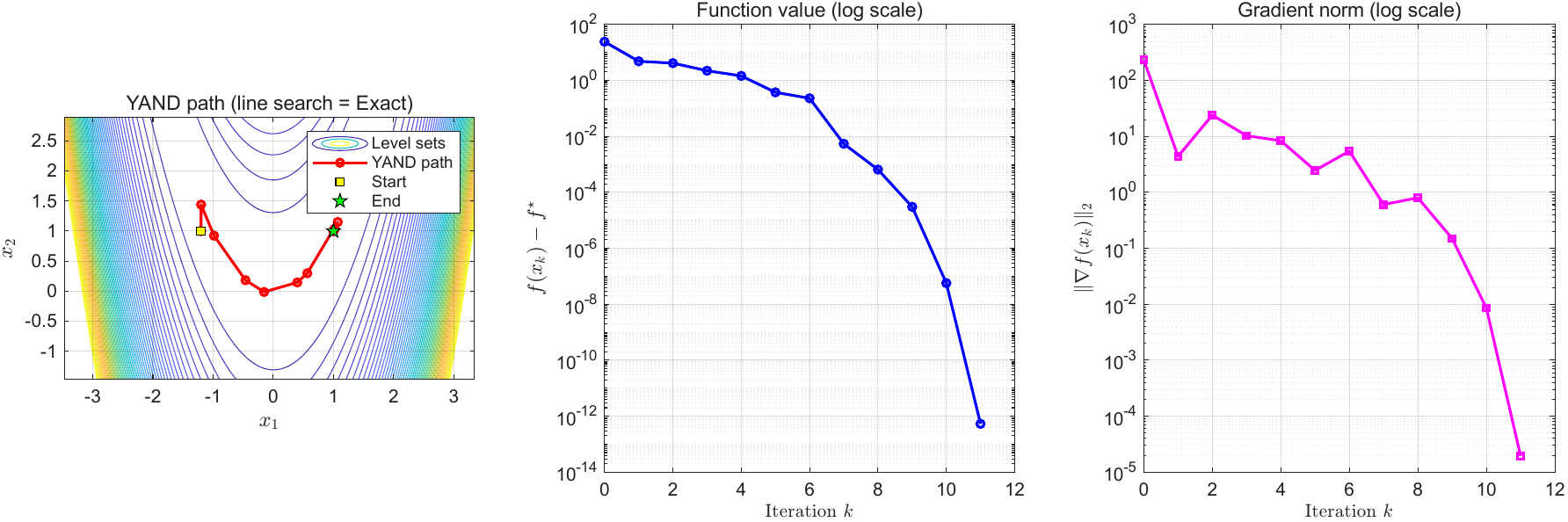}
    \caption{Exact line search}
    \label{fig:rosenbrock-exact}
\end{subfigure}
\hfill
\begin{subfigure}{0.9\textwidth}
    \centering
    \includegraphics[width=\linewidth]{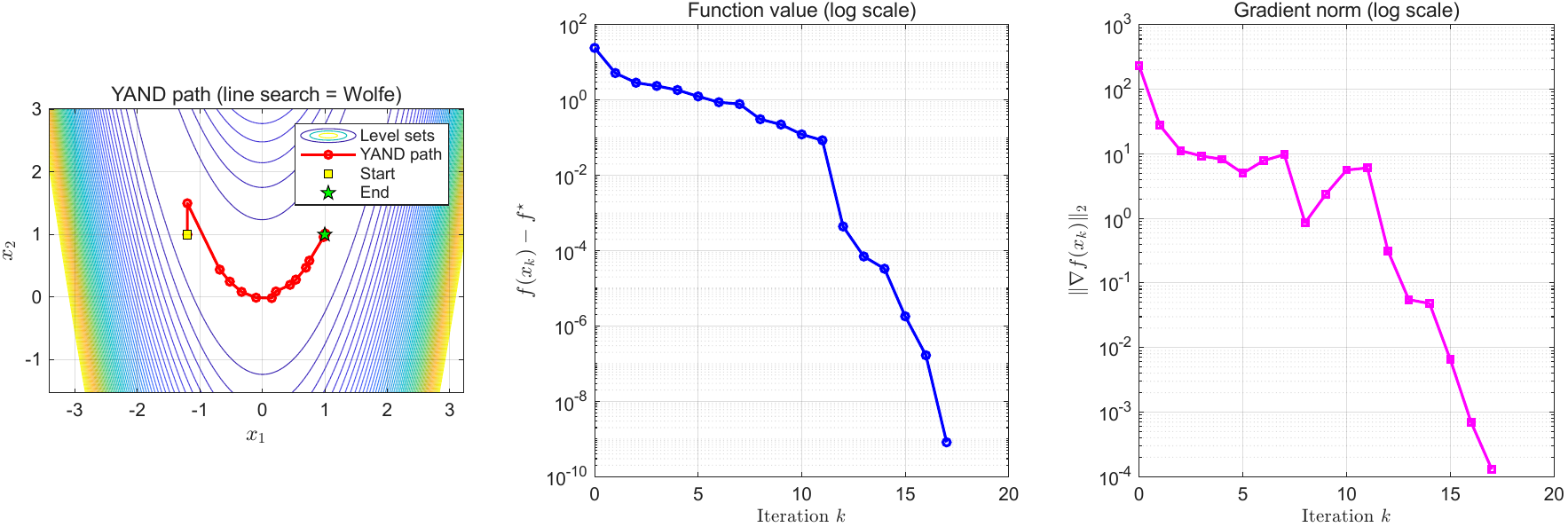}
    \caption{Strong Wolfe}
    \label{fig:rosenbrock-wolfe}
\end{subfigure}
\hfill
\begin{subfigure}{0.9\textwidth}
    \centering
    \includegraphics[width=\linewidth]{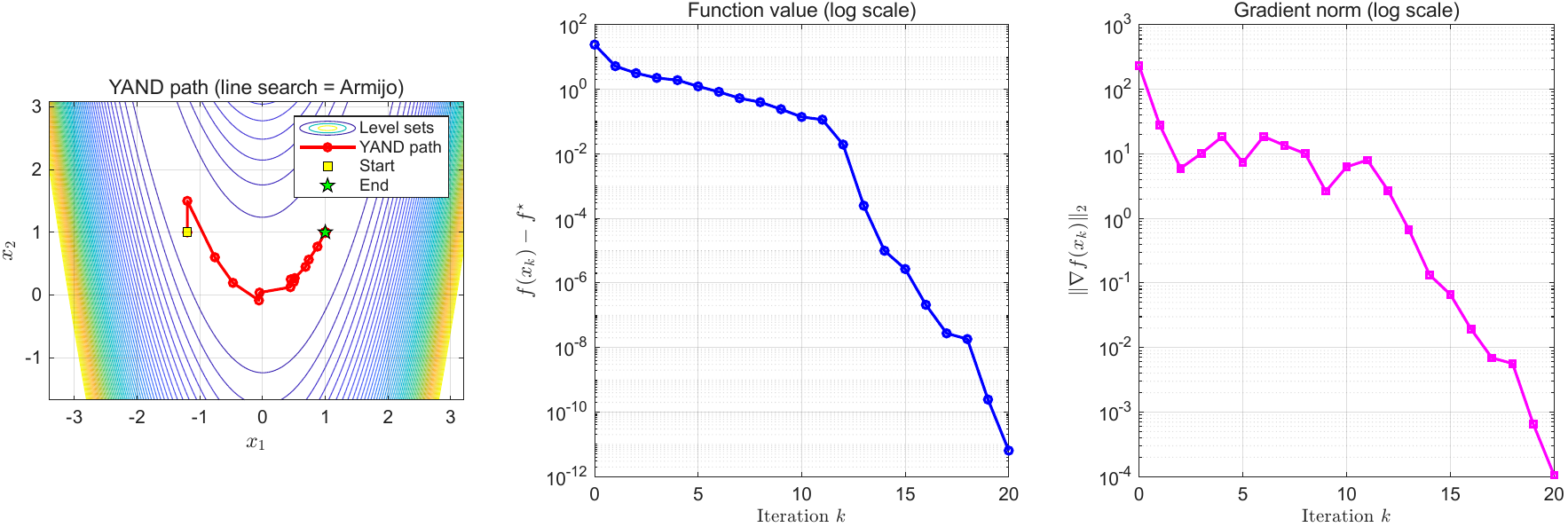}
    \caption{Armijo backtracking}
    \label{fig:rosenbrock-armijo}
\end{subfigure}
\caption{
YAND on the Rosenbrock function \eqref{eq:rosenbrock}
starting from $x_0=(-1.2,1.0)^\top$.
}
\label{fig:rosenbrock-AND}
\end{figure}

Figure~\ref{fig:rosenbrock-AND} reports the trajectories and convergence
profiles of YAND under the three line-search strategies.
All variants successfully follow the curved valley and converge
to the global minimizer.
The exact line search produces the most direct trajectory along the
valley, while the strong Wolfe and Armijo rules take more conservative
steps in regions of high curvature.

For comparison, Figures~\ref{fig:rosenbrock-gd} and
\ref{fig:rosenbrock-newton} illustrate the behavior of gradient descent
and damped Newton's method under the same Wolfe conditions.
Gradient descent exhibits the well-known zigzagging behavior when
traversing the narrow valley, resulting in slow progress.
Damped Newton accelerates once the iterates approach the minimizer,
but requires substantial damping in the early phase to maintain
stability.

\begin{figure}[h!]
\centering
\includegraphics[width=0.9\textwidth]{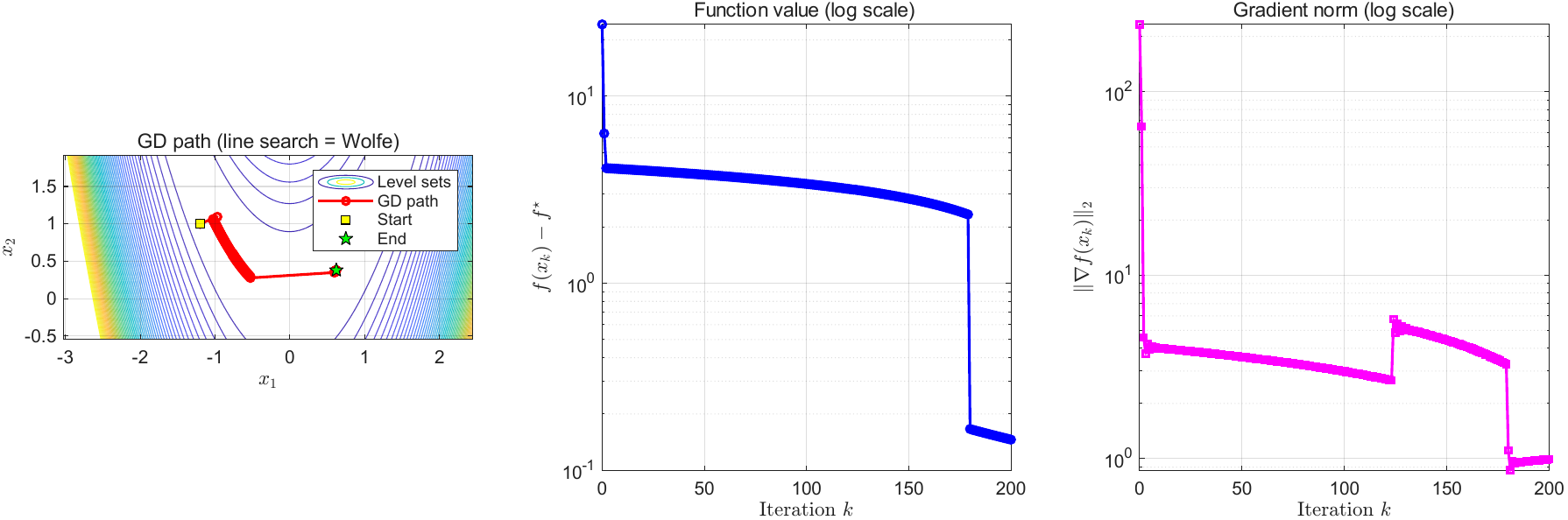}
\caption{Gradient descent on the Rosenbrock function
\eqref{eq:rosenbrock} with strong Wolfe line search.}
\label{fig:rosenbrock-gd}
\end{figure}

\begin{figure}[h!]
\centering
\includegraphics[width=0.9\textwidth]{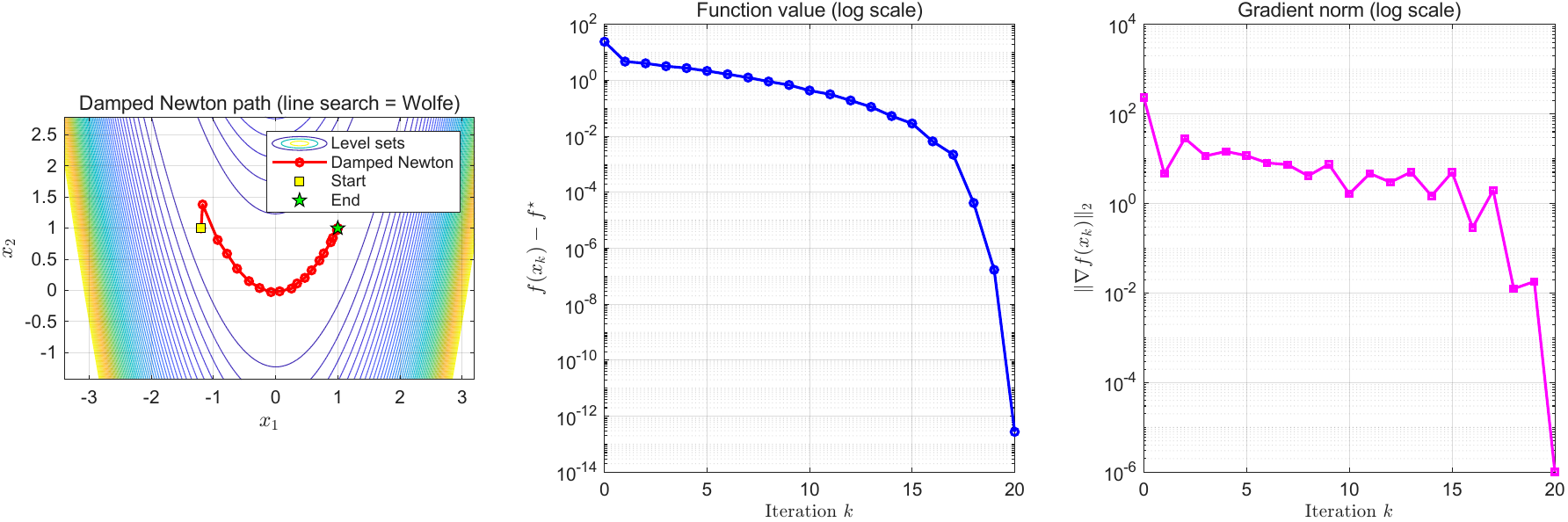}
\caption{Damped Newton on the Rosenbrock function
\eqref{eq:rosenbrock} with strong Wolfe line search.}
\label{fig:rosenbrock-newton}
\end{figure}

In contrast, YAND consistently follows the geometry of the level sets
and maintains stable progress throughout the optimization process.

\subsubsection{Tilted ring-shaped valley}

We next consider the nonconvex objective
\begin{equation}\label{eq:ring-tilted}
f(x)=(x_1^2+x_2^2-1)^2+0.1x_1,
\end{equation}
which forms a nearly circular valley with a small linear tilt. This example complements Rosenbrock by considering a nonconvex landscape
whose minimizer lies along a highly curved valley with a nontrivial
angular component.
The tilt breaks rotational symmetry and induces a unique global
minimizer along the ring. The starting point is chosen as
$
x_0=(0,1.5)^\top .
$

\begin{figure}[h!]
\centering
\begin{subfigure}{0.9\textwidth}
\centering
\includegraphics[width=\linewidth]{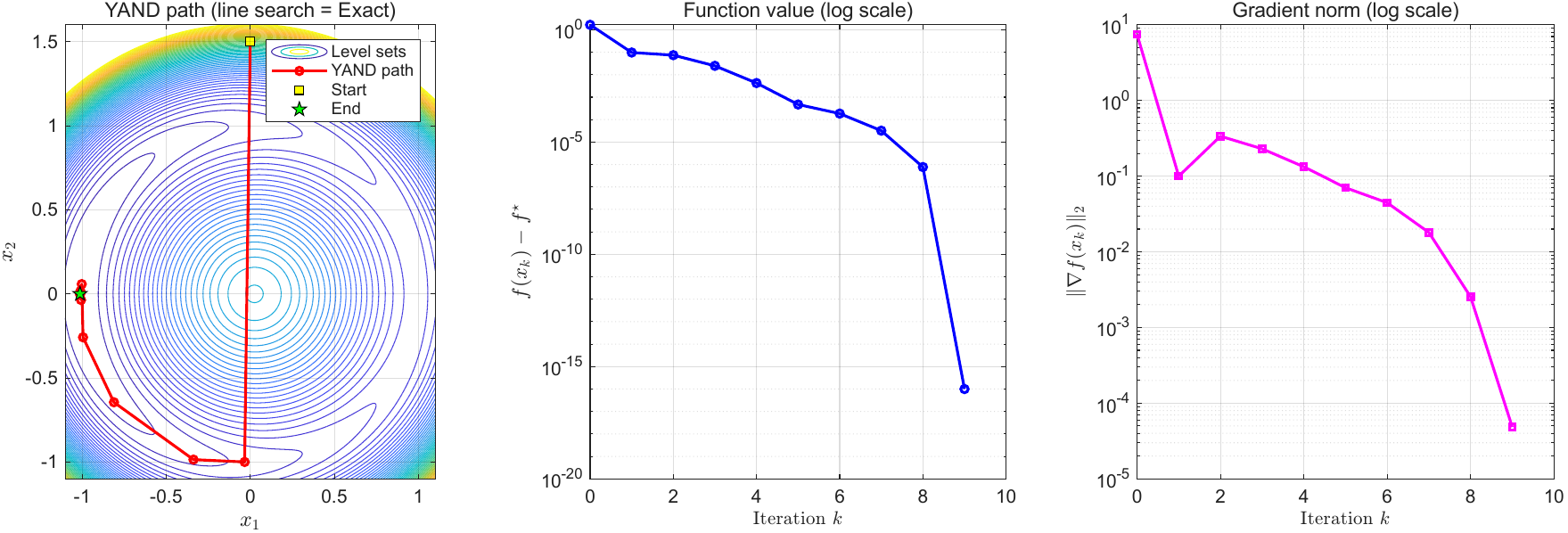}
\caption{Exact line search}
\end{subfigure}

\begin{subfigure}{0.9\textwidth}
\centering
\includegraphics[width=\linewidth]{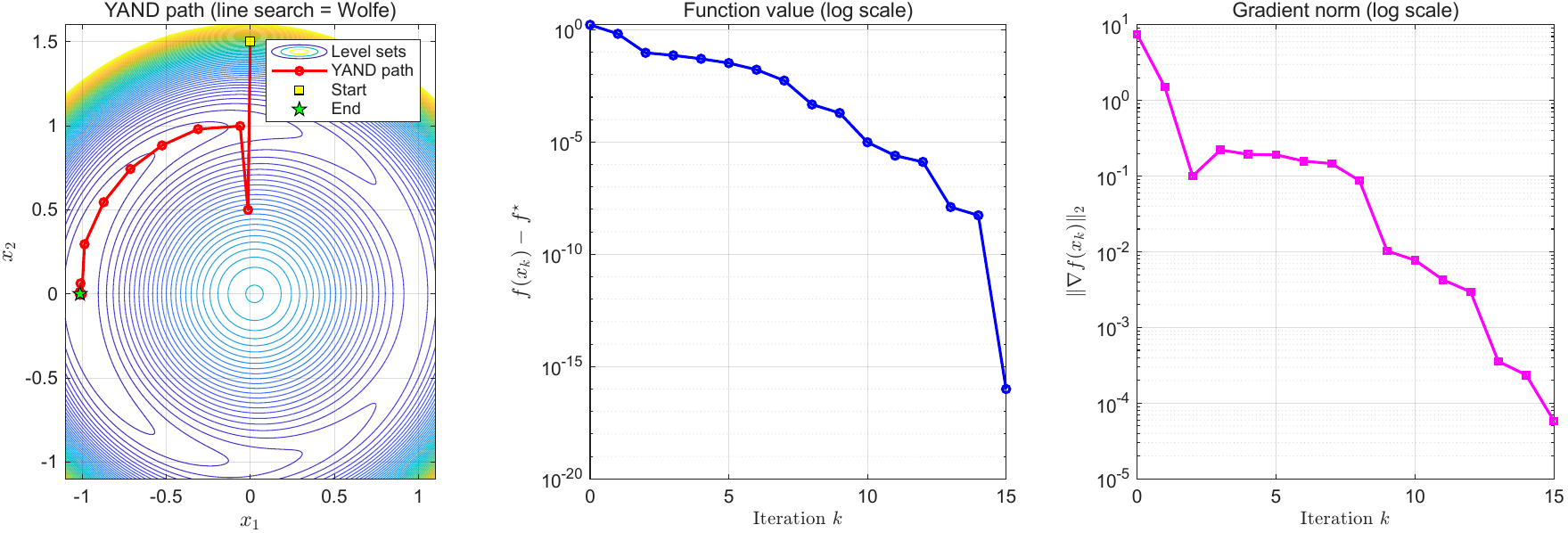}
\caption{Strong Wolfe}
\end{subfigure}

\begin{subfigure}{0.9\textwidth}
\centering
\includegraphics[width=\linewidth]{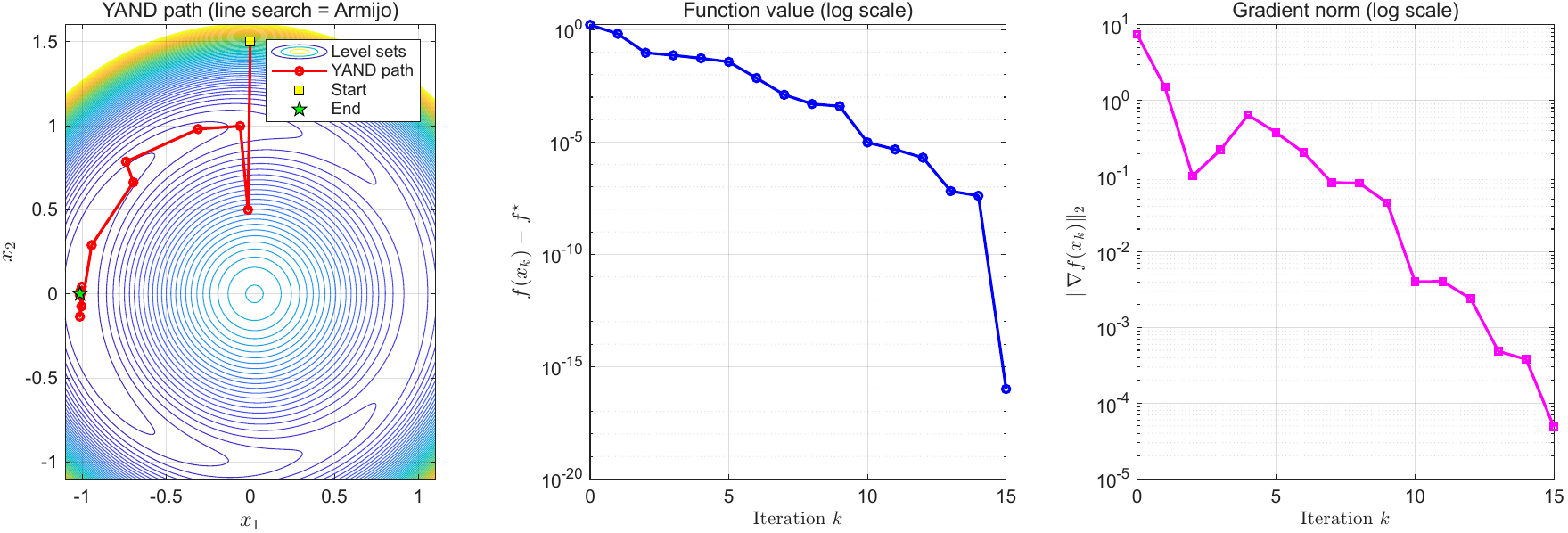}
\caption{Armijo backtracking}
\end{subfigure}

\caption{
YAND on the tilted ring-shaped valley
\eqref{eq:ring-tilted}.
}
\label{fig:ring-tilt-AND}
\end{figure}

Figure~\ref{fig:ring-tilt-AND} shows that all three line-search
variants descend toward the ring and subsequently follow its curved
structure toward the minimizer.
The exact line search produces the most direct trajectory, while
the Wolfe and Armijo rules take smaller steps but maintain stable
convergence.

\subsubsection{Saddle-type polynomial}
This example probes the behavior of the affine-normal direction near a
strict saddle, where the Hessian is indefinite. Consider the nonconvex polynomial
\begin{equation}\label{eq:saddle-poly}
f(x)=x_1^4-x_1^2+x_2^2.
\end{equation}
The origin is a strict saddle point with
$\nabla^2 f(0,0)=\mathrm{diag}(-2,2)$,
while local minima occur near $(\pm2^{-1/2},0)^\top$.
Starting from
$
x_0=(0.1,0.2)^\top,
$
the iterates must escape the saddle region before converging
to one of the wells.

\begin{figure}[h!]
\centering
\begin{subfigure}{0.9\textwidth}
\centering
\includegraphics[width=\linewidth]{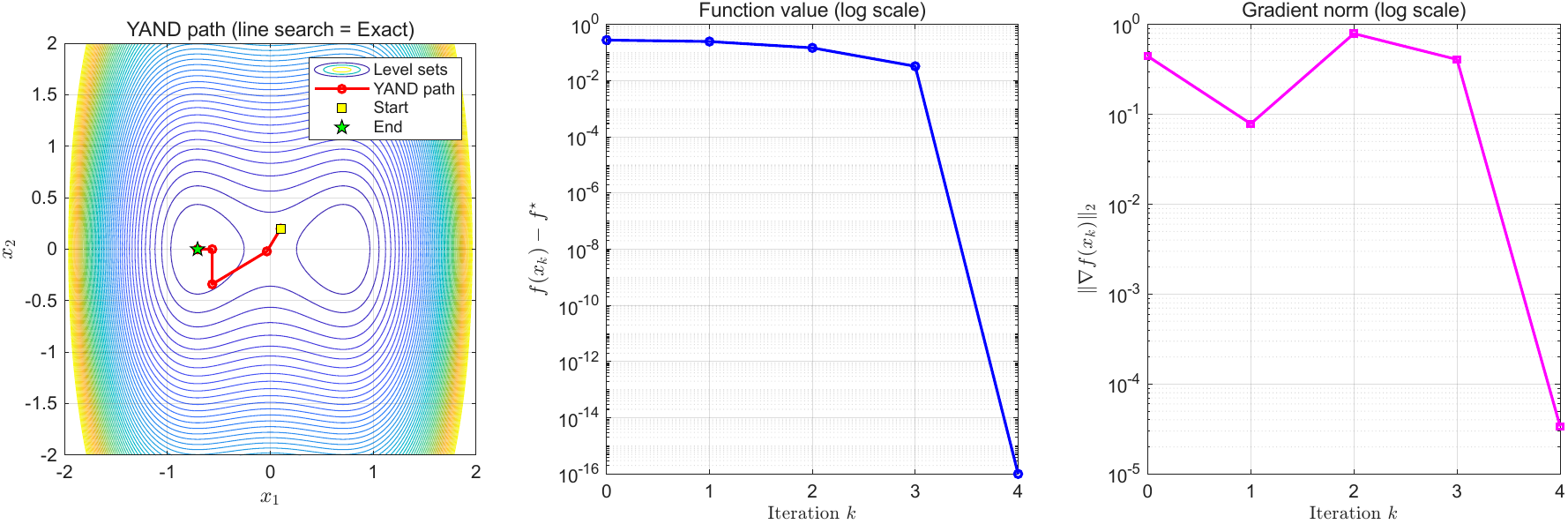}
\caption{Exact line search}
\end{subfigure}

\begin{subfigure}{0.9\textwidth}
\centering
\includegraphics[width=\linewidth]{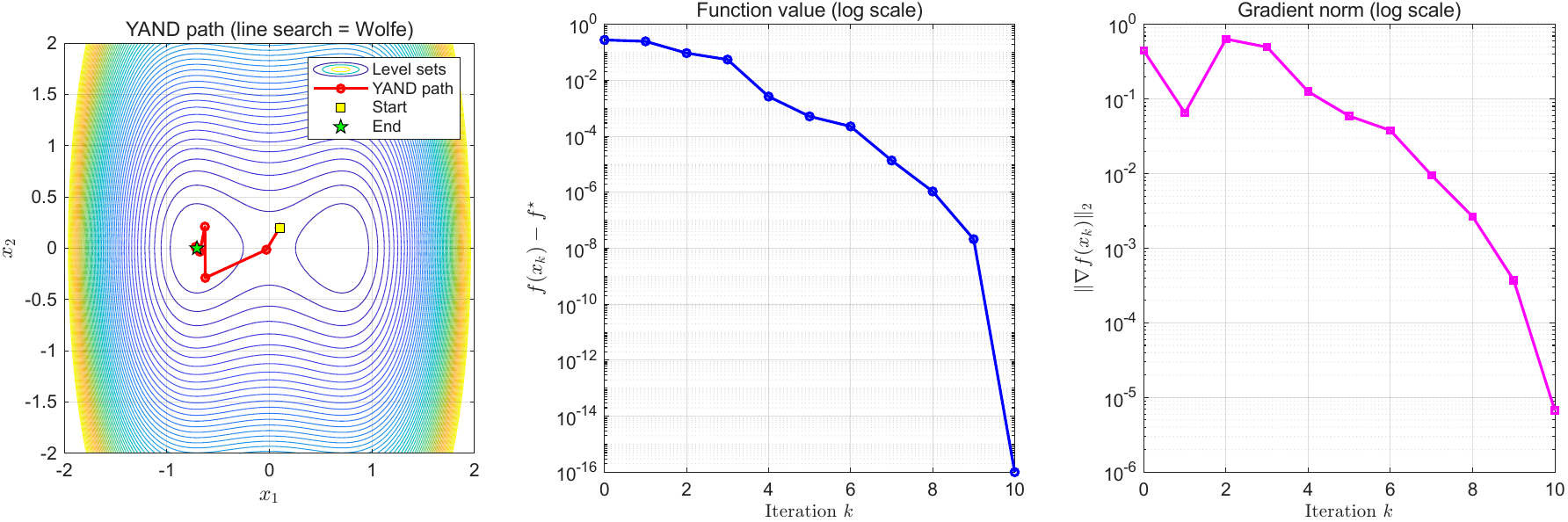}
\caption{Strong Wolfe}
\end{subfigure}

\begin{subfigure}{0.9\textwidth}
\centering
\includegraphics[width=\linewidth]{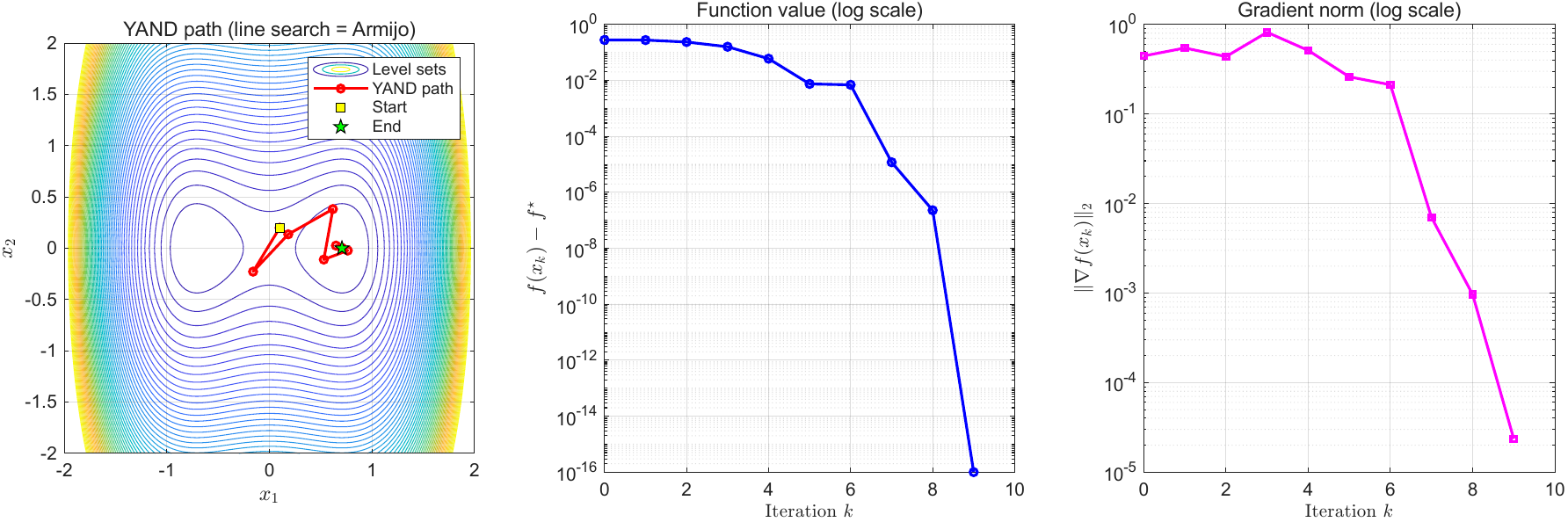}
\caption{Armijo backtracking}
\end{subfigure}

\caption{
YAND on the saddle-type polynomial
\eqref{eq:saddle-poly}.
}
\label{fig:saddle-poly-AND}
\end{figure}

Figure~\ref{fig:saddle-poly-AND} shows that all variants of YAND
escape the saddle and converge to a nearby minimizer.
The exact line search produces the fastest decrease,
while the Wolfe and Armijo rules take shorter steps
in regions of strongly negative curvature.

\subsubsection{Four-well quartic potential}
This example illustrates the basin-selection behavior of YAND in a
multi-well nonconvex landscape.
Consider the quartic function
\begin{equation}\label{eq:quartic-4well}
f(x)=(x_1^2-1)^2+(x_2^2-1)^2,
\end{equation}
which has four equivalent global minimizers at
$(\pm1,\pm1)^\top$.
The origin is a local maximum and saddle points occur
along the coordinate axes.

Starting from
$
x_0=(0.1,-1.5)^\top,
$
the trajectory must navigate the saddle geometry before
entering one of the wells.

\begin{figure}[h!]
\centering
\begin{subfigure}{0.9\textwidth}
\centering
\includegraphics[width=\linewidth]{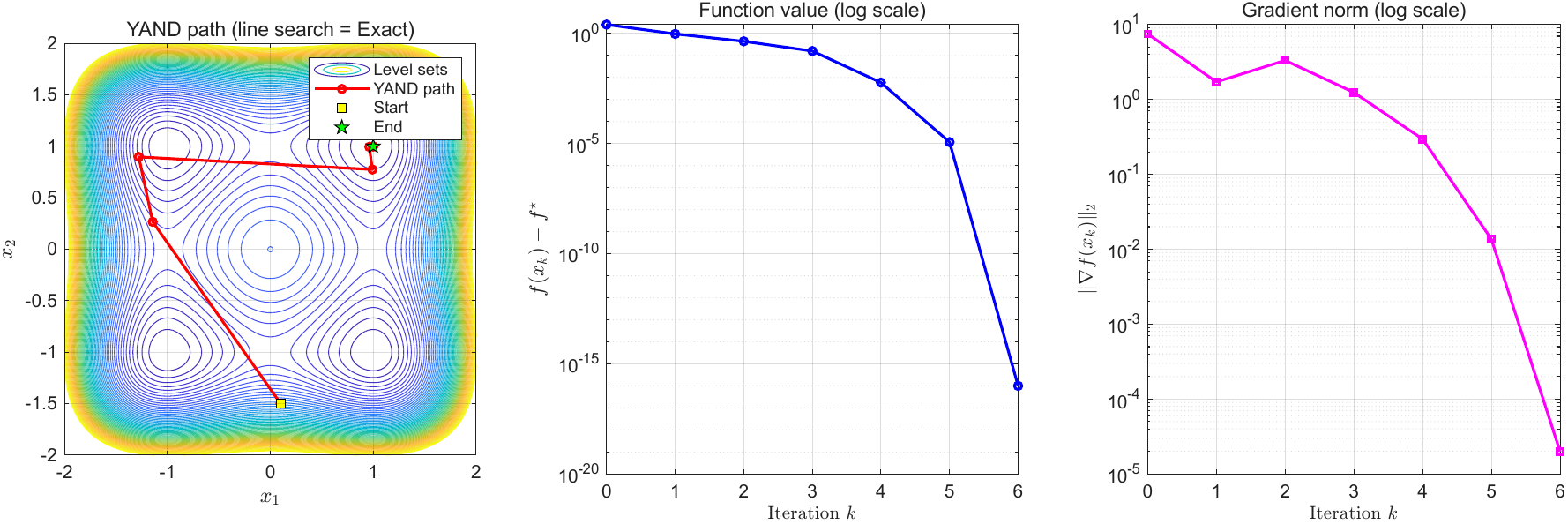}
\caption{Exact line search}
\end{subfigure}

\begin{subfigure}{0.9\textwidth}
\centering
\includegraphics[width=\linewidth]{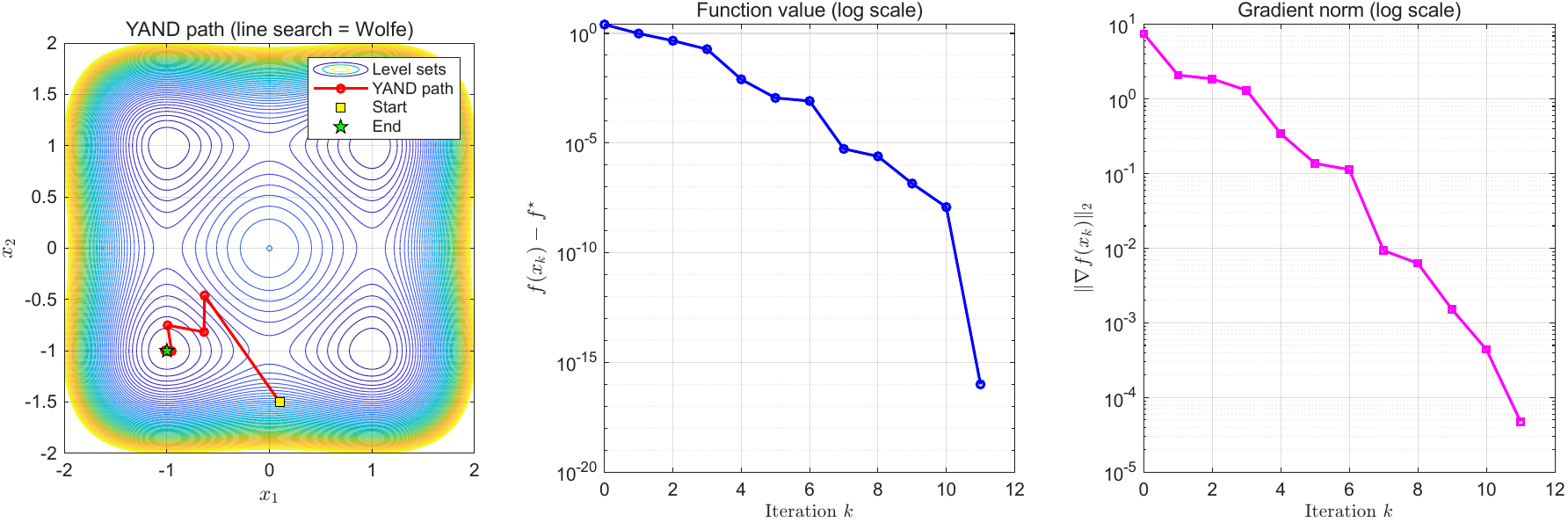}
\caption{Strong Wolfe}
\end{subfigure}

\begin{subfigure}{0.9\textwidth}
\centering
\includegraphics[width=\linewidth]{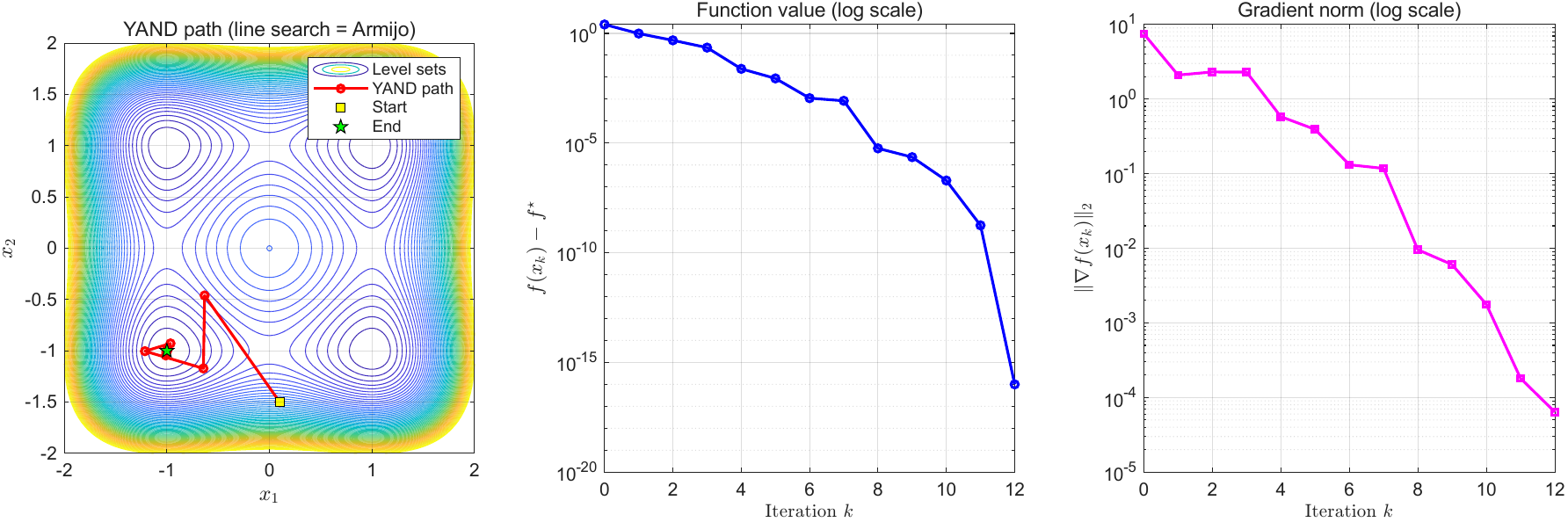}
\caption{Armijo backtracking}
\end{subfigure}

\caption{
YAND on the four-well quartic
\eqref{eq:quartic-4well}.
}
\label{fig:quartic-4well-AND}
\end{figure}

All three line-search strategies eventually converge to a global
minimizer.
The exact line search produces the shortest trajectory,
while Wolfe and Armijo take more conservative steps but maintain
stable descent.

\paragraph{Symmetry-induced convergence to a saddle.}

We also tested the same problem from the symmetric starting point
$
x_0=(0,-1.5)^\top .
$
Due to the symmetry of the objective, both the gradient and the
affine-normal direction remain confined to the invariant manifold
$x_1=0$.
Consequently the iteration reduces to the one-dimensional function
\[
g(t)=f(0,t)=1+(t^2-1)^2,
\]
whose minima occur at $t=\pm1$.
These correspond to the saddle points $(0,\pm1)^\top$ of the full
two-dimensional problem.

This example highlights a limitation typical of descent-type methods:
the iteration may converge rapidly to a saddle when symmetry confines
the trajectory to a lower-dimensional invariant manifold.
A small perturbation of the initial point breaks the symmetry and
steers the trajectory toward one of the true wells.

\paragraph{Summary of line-search strategies}

Across all test problems, the qualitative behavior of YAND with
different line-search rules can be summarized as follows.
Exact line search most closely matches the ideal affine-normal step
and typically yields the fastest convergence.
Strong Wolfe provides a good balance between robustness and
efficiency, while Armijo backtracking offers a simple and stable
alternative with slightly more conservative steps.

Taken together, these examples show that the affine-normal direction
remains effective across a broad range of smooth nonconvex geometries,
including curved valleys, saddle regions, and multi-well landscapes.
While, as expected, convergence in the nonconvex setting is only to
first-order stationary points in general, the observed numerical behavior
is stable and aligns well with the theoretical picture developed earlier.

\subsection{Summary of numerical results}

To summarize the behavior of the considered methods across the different
geometric regimes, Table~\ref{tab:summary-numerics} reports a qualitative
comparison of the representative experiments.

\begin{table}[t]
\centering
\caption{Qualitative summary of the representative numerical experiments.}
\label{tab:summary-numerics}

\footnotesize
\setlength{\tabcolsep}{3pt}
\renewcommand{\arraystretch}{1.14}

\begin{adjustbox}{max width=\textwidth,center}
\begin{tabular}{L{2.55cm} L{2.95cm} L{3.05cm} L{3.45cm} L{3.45cm}}
\toprule
Problem
& Geometry
& Main message
& YAND
& Classical baselines \\
\midrule

Well-conditioned quadratic
& Elliptical bowl
& Quadratic exactness
& One-step convergence with exact line search
& Matches Newton; faster than GD \\

\addlinespace[2pt]

Affine-scaled quadratic
& Affine-scaled anisotropy
& Robustness under affine scaling
& Essentially invariant after normalization
& GD is conditioning-sensitive; Newton remains invariant \\

\addlinespace[2pt]

Sixth-degree polynomial
& Nonlinear convex anisotropy
& Behavior beyond quadratics
& Rapid and stable convergence
& No exact quadratic equivalence \\

\addlinespace[2pt]

Inverse-barrier problem
& Barrier-induced curvature growth
& Robustness under extreme local conditioning
& Stable despite severe anisotropy
& Euclidean directions become more sensitive \\

\addlinespace[2pt]

Rosenbrock function
& Narrow curved valley
& Nonconvex valley tracking
& Stable progress along the valley
& GD zigzags; Newton requires damping \\

\addlinespace[2pt]

Tilted ring-shaped valley
& Curved ring-shaped basin
& Adaptation to curved nonconvex geometry
& Successfully follows the valley to the minimizer
& More sensitive to direction misalignment \\

\addlinespace[2pt]

Saddle-type polynomial
& Strict saddle with double well
& Behavior near indefinite Hessian
& Escapes the saddle and enters a nearby well
& More sensitive to saddle geometry \\

\addlinespace[2pt]

Four-well quartic potential
& Multi-basin landscape
& Basin selection and stationary-point limitation
& Stable basin convergence; symmetry may trap iterates
& Similar limitations occur for descent-type methods \\

\bottomrule
\end{tabular}
\end{adjustbox}
\end{table}

\section{Conclusion}

We have introduced Yau's \emph{affine normal descent} (YAND), a geometric
framework for smooth unconstrained optimization in which search directions
are defined intrinsically by the equi--affine normal of level-set hypersurfaces.
This perspective departs from classical approaches based on Euclidean or
quadratic models, and instead derives optimization directions directly from
the intrinsic geometry of level sets.

We established several fundamental properties of affine-normal directions.
In particular, we connected their analytic representation with the classical
slice--centroid construction, thereby linking computational formulas with
their geometric origin. We characterized precisely when affine-normal directions
yield strict descent and showed that, for strictly convex quadratic objectives,
YAND recovers Newton's method under exact line search.
We further developed a convergence theory establishing global convergence
under standard smoothness assumptions, linear convergence under strong convexity
or the Polyak--\L{}ojasiewicz condition, and quadratic local convergence near
nondegenerate minimizers.

We also analyzed the behavior of affine-normal directions under affine scalings,
showing that the method is inherently robust to arbitrarily ill-conditioned
transformations. This provides a geometric explanation for the stability of YAND
under severe anisotropic distortions of the objective.

Numerical experiments illustrate the geometric behavior of the method and its
robustness across a range of convex and nonconvex problems. Together, these
results suggest that affine differential geometry provides a natural and
powerful framework for designing curvature-aware optimization algorithms.

Several directions remain for future investigation. A central challenge is the
efficient computation or approximation of affine-normal directions in
high-dimensional settings, where analytic formulas involve higher-order
derivatives. Geometric constructions based on local moment information of
level sets, such as slice--centroid formulations, suggest a promising route
toward scalable implementations without explicit higher-order derivatives.
Extensions to constrained and stochastic optimization settings also constitute
natural directions for further development.

\section*{Acknowledgements}

Y.-S.\ N.\ was supported by the National Natural Science Foundation of China
(Grant No.\ 42450242) and China's National Program of Overseas High-Level Talent.
A.\ S.\ would like to acknowledge support from the Beijing Natural Science
Foundation (Grant No.\ BJNSF--IS24005) and the NSFC--RFIS Program
(Grant No.\ W2432008).
He also thanks the NSF AI Institute for Artificial Intelligence and Fundamental
Interactions at the Massachusetts Institute of Technology (MIT), funded by the
U.S.\ National Science Foundation under Cooperative Agreement PHY--2019786,
as well as China's National Program of Overseas High-Level Talent for generous
support.
All three authors gratefully acknowledge institutional support from the Beijing
Institute of Mathematical Sciences and Applications (BIMSA). The authors would also like to thank Prof.\ Liping Zhang of Tsinghua University
for helpful discussions.

\appendix
\section{The affine normal}

\subsection{Foundational concepts}

\subsubsection{Hypersurfaces and immersions}

\begin{definition}[Hypersurface Immersion]
Let $M^n$ be a smooth $n$-dimensional manifold. A \textbf{smooth immersion} $X: M^n \to \mathbb{R}^{n+1}$ is a $C^\infty$ map such that the differential $dX_p: T_pM \to T_{X(p)}\mathbb{R}^{n+1}$ is injective for all $p \in M$. This ensures that $X(M)$ is an embedded submanifold locally, though it may have self-intersections globally.

The \textbf{tangent space} at $p$ is $T_pX(M) = dX_p(T_pM)$, and the \textbf{normal space} is its orthogonal complement $N_pX(M) = (T_pX(M))^\perp$.
\end{definition}

\subsubsection{Connections and covariant derivatives}

\begin{definition}[Affine Connection]
An \textbf{affine connection} $\nabla$ on a manifold $M$ is a bilinear map $\nabla: \Gamma(TM) \times \Gamma(TM) \to \Gamma(TM)$ satisfying:
\begin{enumerate}
\item $\nabla_{fX}Y = f\nabla_XY$ ($C^\infty$-linear in the first argument)
\item $\nabla_X(fY) = X(f)Y + f\nabla_XY$ (Leibniz rule in the second argument)
\end{enumerate}
for all $X,Y \in \Gamma(TM)$, $f \in C^\infty(M)$.
\end{definition}

\begin{definition}[Euclidean Connection]
The \textbf{Euclidean flat connection} $D$ on $\mathbb{R}^{n+1}$ is defined for vector fields $U = \sum u^i\partial_i$, $V = \sum v^j\partial_j$ by:
\[
D_UV = \sum_{i,j} u^i\frac{\partial v^j}{\partial x^i}\partial_j.
\]
This connection is flat (zero curvature) and torsion-free.
\end{definition}

\subsubsection{Transversal vector fields}

\begin{definition}[Transversal Vector Field]
A smooth vector field $\xi$ along $X(M)$ is called \textbf{transversal} if for every $p \in M$:
\[
\xi(p) \notin T_{X(p)}X(M).
\]
Equivalently, $\{\partial_1X(p), \ldots, \partial_nX(p), \xi(p)\}$ forms a basis of $\mathbb{R}^{n+1}$.
\end{definition}

\subsection{Gauss formula and induced structures}

\begin{theorem}[Gauss Formula for Hypersurfaces]
Given a hypersurface immersion $X: M^n \to \mathbb{R}^{n+1}$ and a transversal vector field $\xi$, there exists a unique decomposition:
\[
D_{dX(X)}dX(Y) = dX(\nabla_XY) + h(X,Y)\xi \quad \text{for all } X,Y \in \Gamma(TM),
\]
where:
\begin{itemize}
\item $\nabla$ is an affine connection on $M$ (the \textbf{induced connection})
\item $h: \Gamma(TM) \times \Gamma(TM) \to C^\infty(M)$ is a symmetric bilinear form (the \textbf{affine fundamental form})
\item The decomposition is orthogonal with respect to the transversal direction
\end{itemize}
\end{theorem}

\begin{proof}
Since $\xi$ is transversal, we can write any vector in $\mathbb{R}^{n+1}$ uniquely as a tangent part plus a multiple of $\xi$. The tangential projection defines $\nabla$, while the coefficient of $\xi$ defines $h$.
\end{proof}

\subsubsection{Weingarten formula}

\begin{theorem}[Weingarten Formula]
For the transversal field $\xi$ and any $X \in \Gamma(TM)$, we have:
\[
D_{dX(X)}\xi = -dX(S(X)) + \tau(X)\xi,
\]
where:
\begin{itemize}
\item $S: \Gamma(TM) \to \Gamma(TM)$ is the \textbf{shape operator} or \textbf{Weingarten map}
\item $\tau: \Gamma(TM) \to C^\infty(M)$ is the \textbf{transversal connection form}
\end{itemize}
\end{theorem}

\subsection{Volume forms and equi--affine theory}

\subsubsection{Induced volume forms}

\begin{definition}[Induced Volume Form]
Given a transversal field $\xi$, the \textbf{induced volume form} $\theta_\xi$ is defined by:
\[
\theta_\xi(X_1, \ldots, X_n) = \det\left(dX(X_1), \ldots, dX(X_n), \xi\right).
\]
This is a nonvanishing $n$-form on $M$.
\end{definition}

\begin{definition}[Affine Metric Volume]
For a nondegenerate hypersurface, the \textbf{affine metric volume form} is:
\[
\omega_h(X_1, \ldots, X_n) = |\det(h(X_i, X_j))|^{1/2}.
\]
This volume form depends only on the affine fundamental form $h$.
\end{definition}

\subsubsection{Equi--affine conditions}

\begin{definition}[Equi--affine Structure]
A transversal field $\xi$ is called \textbf{equi--affine} if the induced connection $\nabla$ satisfies:
\begin{enumerate}
\item $\nabla$ is torsion-free: $\nabla_XY - \nabla_YX = [X,Y]$
\item $\nabla\theta_\xi = 0$ (the volume form is $\nabla$-parallel)
\item $\tau = 0$ (the transversal connection form vanishes)
\end{enumerate}
\end{definition}

\subsection{Existence and uniqueness}

\begin{theorem}[Existence and Uniqueness of Affine Normal]
Let $X: M^n \to \mathbb{R}^{n+1}$ be a nondegenerate hypersurface immersion. There exists a unique (up to sign) transversal vector field $\xi$ such that:

\begin{enumerate}
\item \textbf{Equi--affine Condition}: $\xi$ is equi--affine, i.e., $\nabla\theta_\xi = 0$ and $\tau = 0$

\item \textbf{Volume Compatibility}: The induced volume form is a constant multiple of the affine metric volume form:
\[
\theta_\xi = c \cdot \omega_h \quad \text{for some constant } c > 0
\]

\item \textbf{Normalization}: In local coordinates, if $\det(\partial_1X, \ldots, \partial_nX, \xi) = 1$, then the affine fundamental form $h$ becomes the affine metric.
\end{enumerate}

This unique $\xi$ is called the \textbf{affine normal} or \textbf{Blaschke normal}.
\end{theorem}

\subsection{Local coordinate expressions}

\subsubsection{Coordinate formulation}

Let $(u^1, \ldots, u^n)$ be local coordinates on $M$, and write the immersion as $X(u^1, \ldots, u^n)$. The coordinate frame is:
\[
X_i = \frac{\partial X}{\partial u^i}, \quad X_{ij} = \frac{\partial^2 X}{\partial u^i\partial u^j}.
\]

The Gauss formula becomes:
\[
X_{ij} = \sum_{k=1}^n \Gamma_{ij}^k X_k + h_{ij}\xi.
\]

The Weingarten formula is:
\[
\xi_i = \frac{\partial \xi}{\partial u^i} = -\sum_{j=1}^n S_i^j X_j.
\]

\subsubsection{Determinant formulation}

In the equi--affine normalization with $\det(X_1, \ldots, X_n, \xi) = 1$, the affine metric components are given by:
\[
h_{ij} = -\det(X_1, \ldots, X_n, \xi_{ij}).
\]

The affine normal can be expressed in terms of the position vector as:
\[
\xi = \frac{1}{n}\Delta X,
\]
where $\Delta$ is the Laplace-Beltrami operator with respect to the affine metric.

\subsection{Transformation properties}

\begin{theorem}[Affine Invariance]
The affine normal $\xi$ and affine metric $h$ are invariant under the unimodular affine group $\mathrm{SL}(n+1, \mathbb{R}) \ltimes \mathbb{R}^{n+1}$. Specifically:

If $\tilde{X} = AX + b$ with $A \in \mathrm{SL}(n+1, \mathbb{R})$, $b \in \mathbb{R}^{n+1}$, then:
\begin{itemize}
\item The affine normal transforms as $\tilde{\xi} = A\xi$
\item The affine metric transforms as $\tilde{h} = h$
\item The induced connection remains unchanged: $\tilde{\nabla} = \nabla$
\end{itemize}
\end{theorem}

\begin{theorem}[Conformal Relation to Euclidean Normal]
For a locally strictly convex hypersurface, let $\nu$ be the Euclidean unit normal and $H$ the mean curvature. Then the affine normal is related to the Euclidean normal by:
\[
\xi = H^{1/(n+2)} \nu + \text{tangential component}.
\]
In particular, for surfaces in $\mathbb{R}^3$ ($n=2$), we have $\xi = H^{1/4}\nu + \cdots$.
\end{theorem}

\subsection{Special cases and examples}

\subsubsection{Ellipsoids and affine spheres}

\begin{definition}[Affine Sphere]
A hypersurface is called an \textbf{affine sphere} if all affine normals meet at a point (proper affine sphere) or are parallel (improper affine sphere).
\end{definition}

\begin{example}[Ellipsoids]
For an ellipsoid $\frac{x^2}{a^2} + \frac{y^2}{b^2} + \frac{z^2}{c^2} = 1$, the affine normals all pass through the center. This explains why YAND finds the minimum of quadratic functions in one step.
\end{example}

\begin{example}[Paraboloids]
For a paraboloid $z = \frac{1}{2}(ax^2 + by^2)$, the affine normals are parallel to the $z$-axis, making it an improper affine sphere.
\end{example}

\subsubsection{Graphs of functions}

For a hypersurface given as the graph of a function $x_{n+1} = f(x_1, \ldots, x_n)$, the affine normal has the explicit formula:

In normal-aligned coordinates (where $\nabla f$ points in the $x_{n+1}$ direction), the affine normal direction is:
\[
d_{\mathrm{AN}} \propto \left(f^{ij}\left(-\frac{1}{n+2}\|\nabla f\| f^{pq}f_{pqi} + f_{n+1,i}\right), -1\right),
\]
where $[f^{ij}]$ is the inverse of the tangent-tangent Hessian block.

\subsubsection{Regularization and robustness}

When the affine metric degenerates (parabolic points), regularization strategies include:

\begin{itemize}
\item Adding a small multiple of the identity: $h_\epsilon = h + \epsilon I$
\item Switching to Euclidean normal in degenerate regions
\item Using trust-region modifications
\end{itemize}

\subsection{Generalizations and extensions}

\subsubsection{Higher-order affine normals}

The theory extends to \textbf{affine normal of higher order}, defined through higher-order affine invariants and related to higher-order optimization methods.

\subsubsection{Relative affine geometry}

In \textbf{relative affine geometry}, one fixes a reference hypersurface and studies affine invariants relative to this reference, leading to preconditioned optimization methods.

\subsubsection{Affine spectral theory}

The eigenvalues of the shape operator $S$ with respect to the affine metric $h$ define \textbf{affine principal curvatures}, which characterize the affine shape of the hypersurface.

\bibliographystyle{plain}
\bibliography{references}

\section*{Author Information}

\noindent
Yi-Shuai Niu$^{1}$, Artan Sheshmani$^{1,3}$, and Shing-Tung Yau$^{1,2}$

\medskip

\noindent
$^{1}$ Beijing Institute of Mathematical Sciences and Applications (BIMSA), Beijing 101408, China

\noindent
$^{2}$ Yau Mathematical Sciences Center, Tsinghua University, Beijing 100084, China

\noindent
$^{3}$ IAIFI Institute, Massachusetts Institute of Technology, Cambridge, MA 02139, USA

\medskip

\noindent
\textit{E-mail addresses:}
niuyishuai@bimsa.cn (Yi-Shuai Niu),
artan@mit.edu (Artan Sheshmani),
styau@tsinghua.edu.cn (Shing-Tung Yau)

\end{document}